\documentclass[reqno,11pt]{amsart}

\usepackage{xcolor}
\usepackage{graphicx}
\usepackage[hidelinks]{hyperref}

\usepackage{latexsym,array}
\usepackage{amsfonts}
\usepackage{shadow}
\usepackage{amsbsy}
\usepackage{amssymb}
\usepackage{dsfont}
\usepackage{mathtools}
\usepackage{enumitem}
\usepackage{doi}

\usepackage[notref, notcite, final]{showkeys}
\usepackage{fullpage}
\usepackage{comment}
\usepackage[latin1]{inputenc}

\usepackage{cleveref}

\numberwithin{equation}{section}

\newcommand{\todo}[1]{\footnote{\textcolor{red}{\bf #1}}}

\newcommand{\ran}{{\rm ran}\,}

\newcommand{\id}{\mathcal{I}}

\newcommand{\sderiv}{\partial_S}
\newcommand{\clos}[1]{\overline{#1}}

\newcommand{\lhol}{\mathcal{SH}_L}
\newcommand{\rhol}{\mathcal{SH}_R}

\newcommand{\lholZI}{\mathcal{SH}_{L,0}^{\infty}}
\newcommand{\rholZI}{\mathcal{SH}_{R,0}^{\infty}}

\newcommand{\intrin}{\mathcal{N}}
\newcommand{\intrinZI}{\mathcal{N}_{0}^{\infty}}

\newcommand{\dom}{\mathcal{D}}

\theoremstyle{theorem}
\newtheorem{theorem}{Theorem}[section]
\newtheorem{lemma}[theorem]{Lemma}
\newtheorem{corollary}[theorem]{Corollary}
\newtheorem{proposition}[theorem]{Proposition}

\theoremstyle{definition}
\newtheorem{remark}[theorem]{{\bf Remark}}
\newtheorem{definition}[theorem]{Definition}

\newcommand{\cc}{\mathbb{C}}

\newcommand{\hh}{\mathbb{H}}
\newcommand{\nn}{\mathbb{N}}

\newcommand{\fdom}{\operatorname{dom}}

\newcommand{\rr}{\mathbb{R}}
\renewcommand{\SS}{\mathbb{S}}

\newcommand{\boundOP}{\mathcal{B}}
\newcommand{\closOP}{\mathcal{K}}
\newcommand{\sectOP}{\operatorname{Sect}}
\newcommand{\EL}{\mathcal{E}_L}
\newcommand{\ER}{\mathcal{E}_R}
\newcommand{\Eint}{\mathcal{E}_{int}}

\newcommand{\meroL}{\mathcal{M}_L}

\newcommand{\meroInt}{\mathcal{M}_{int}}

\newcommand{\VEC}{\operatorname{Vec}}
\newcommand{\DIV}{\operatorname{div}}

\newcommand{\Q}{\mathcal{Q}}

\renewcommand{\Re}{\mathrm{Re}}

\newcommand{\prodL}{{\star_\ell}}
\newcommand{\prodR}{{\star_r}}

\newcommand{\dist}{\mathrm{dist}}
\newcommand{\sector}[1]{\Sigma_{#1}}

\newcommand{\I}{\mathbf{I}}
\newcommand{\J}{\mathbf{J}}
\newcommand{\K}{\mathbf{K}}

\crefname{enumi}{}{}
\crefname{enumii}{}{}

\title[An application of the $S$-functional calculus]
{An application of the $S$-functional calculus to fractional diffusion processes}

\author[F. Colombo]{Fabrizio Colombo}
\address{(FC)
Politecnico di Milano\\Dipartimento di Matematica\\Via E. Bonardi, 9\\20133
Milano, Italy}
\email{fabrizio.colombo@polimi.it}
\author[J. Gantner]{Jonathan Gantner}
\address{(JG)
Politecnico di Milano\\Dipartimento di Matematica\\Via E. Bonardi, 9\\20133
Milano, Italy
} \email{jonathan.gantner@polimi.it}

\begin{document}
\maketitle

\begin{abstract}
In this paper we show how the spectral theory based on the notion of $S$-spectrum allows us to study new classes
of fractional diffusion and of fractional evolution processes.
 We prove new results on the quaternionic version of the
 $H^\infty$ functional calculus and we use it to define the fractional powers of vector operators.
The Fourier laws for the propagation of the heat in non homogeneous materials is a vector operator of the form
$$
T=e_1\,a(x)\partial_{x_1} +  e_2\,b(x)\partial_{x_2} +  e_3\,c(x)\partial_{x_3}
$$
where $e_\ell$, $e_\ell=1,2,3$ are orthogonal unit vectors,
 $a$, $b$, $c$ are suitable real valued function that depend on the space variables $x=(x_1,x_2,x_3)$ and possibly also on time.
In this paper we develop a general theory to define the fractional powers of quaternionic operators which contain as a
particular case the operator $T$ so we can define the  non local version $T^\alpha$, for $\alpha\in (0,1)$, of the Fourier law defined by $T$.
Our  new mathematical tools open the way to a large class of fractional evolution problems that can be defined and studied using our theory based on the $S$-spectrum for vector operators.
 This paper is devoted to researchers in different research  fields such as: fractional diffusion and fractional evolution
 problems, partial differential equations, non commutative operator theory,  and quaternionic analysis.
\end{abstract}
\vskip 1cm

\section{Introduction}

 The study of some classes of operators that do not act by pointwise differentiation, but by global
integration with respect to a singular kernel is a very active research field
 since non local operators appear in several branches of science and  technology.
Non local operators model fractional diffusion and fractional evolution processes.
The most studied non local operators are the fractional
powers of the negative Laplacian, which can be defined in different ways.
If the function $u$ belongs to the Schwartz space of rapidly decreasing functions, then  $(-\Delta)^\alpha u$ can for example be defined via the Fourier transform.
A different approach considers the semigroup  generated by $-\Delta$,
and this definition can be generalized to more general elliptic operators that are the generators of semigroups.
Another approach to define the fractional Laplace operator, which can be found in \cite{CafSil}, is via an extension problem. An explicit integral representation of the fractional Laplacian is given by
$$
(-\Delta)^\alpha u (x)=c(n,\alpha)P.V.\int_{\mathbb{R}^n}\frac{u(x)-u(y)}{|x-y|^{n+2\alpha}}dy,
$$
where the integral is defined in the sense of the principal value, $c(n,\alpha)$ is a known constant, and $u: \mathbb{R}^n\to \mathbb{R}$ belongs to a suitable function space.
In the literature there are several non linear models that involve the fractional Laplacian and even the fractional powers
of more general elliptic operators.
With no claim of completeness, we refer to the books
\cite{BocurValdinoci,Vazquez} and the references therein  for a more extended introduction to fractional diffusion problems. For more recent results see also \cite{CSV,CV2,GMPunzo,MRT}.
\\
In 1960 Balakrishnan (see \cite{Balakrishnan}) obtained  a   construction  for  the fractional  powers of a closed linear operator $A$ on a Banach space, in which  he does  not  require  that  $A$  generates  a  semigroup. He proved that if any $\lambda>0$ belongs to the resolvent set of $A$ and there exists a positive constant $M$ such that
   $$
   \|\lambda (\lambda -A)^{-1}\|< M, \qquad \lambda > 0,
   $$
 i.e. if $-A$ is a sectorial operator in today's terminology, then the fractional powers of $-A$ can be defined by the  integral
   $$
   (-A)^\alpha v=\frac{\sin(\alpha \pi )}{\pi}\int_0^\infty \lambda^{\alpha-1}(\lambda  - A)^{-1}(-A)x\, d\lambda,\ \ \ v\in \dom(A),
   $$
   for $\alpha\in (0,1)$. This formula can also be obtained as a consequence of the $H^\infty$-functional calculus introduced by A. McIntosh in \cite{McI1}.
   Some of the first papers on fractional powers of operators are \cite{Guzman1,Kato,Komatsu1,Watanabe,Yosida}.
\\
\\
In this paper we offer a new point of view to define fractional diffusion processes that is based the on
quaternionic version of the $H^\infty$-functional calculus based on the $S$-spectrum.
In fact, using the quaternionic $H^\infty$-functional calculus we deduce the quaternionic Balakrishnan formula
and we apply it to vector operators  that are special cases of quaternionic operators.
Precisely, when represented in components, quaternionic linear operators, are of the form
$$
T=T_0+e_1T_1+e_2T_2+e_3T_3,
$$
where $T_\ell$, $\ell=0,1,2,3$ are linear operators acting on a real Banach spaces.  So purely imaginary quaternionic operators $e_1T_1+e_2T_2+e_3T_3$ are vector operators.
 The quaternionic version (or vector version)
 of the Riesz-Dunford functional calculus  requires the notions of $S$-spectrum, $S$-resolvent set and $S$-resolvent operators.
 We denote by  $\mathbb{H}$ the algebra of quaternions. The real part of a quaternion $s=s_0+e_1s_1+e_2s_2+e_3s_3$  is indicated by $s_0$ or ${\rm Re}(s)$ and $|s|^2=s_0^2+s_1^2+s_2^2+s_3^2$ is the square of the Euclidean norm. We furthermore denote by
$\mathcal{B}(V)$ the space of all bounded quaternionic right-linear operators defined on a two-sided quaternionic Banach space $V$.
For  a closed quaternionic right-linear operator $T$, we define the $S$-resolvent set of $T$ as
$$
\rho_S(T)=\{ s\in \mathbb{H}\ \ :\ \ (T^2-2 s_0T+|s|^2\mathcal{I})^{-1} \in \mathcal{B}(V)
\},
$$
 where $\mathcal{I}$ is the identity operator, and we define the $S$-spectrum of $T$ as
$$
\sigma_S(T)=\mathbb{H}\setminus \rho_S(T).
$$
Due to the noncommutativity of the quaternions there are two resolvent operators associated with a quaternionic
linear operator $T$ and hence there are two formulations of the quaternionic functional calculus.
We set
$$
\Q_s(T):=T^2-2 s_0T+|s|^2\mathcal{I} \ \ \text{for} \ \ s\in\hh.
$$
For a right linear operator $T$ the left  $S$-resolvent operator is defined as
\begin{equation}
S_L^{-1}(s,T):=-T\Q_s(T)^{-1} + \Q_s(T)^{-1}\overline{s},\ \ \ s \in \rho_S(T),
\end{equation}
while the right $S$-resolvent operator is defined as
\begin{equation}
S_R^{-1}(s,T):=-(T-\overline{s}\mathcal{I})\Q_s(T)^{-1},\ \ \ s \in  \rho_S(T).
\end{equation}
To define the quaternionic functional calculus, also called $S$-functional calculus, we replace the notion of holomorphicity by the notion of slice hyperholomorphicity, see the sequel for the precise definition.
Let $T$ be a bounded quaternionic linear operator and let $U\subset \hh$ be a suitable domain that contains the $S$-spectrum of $T$.
For a left slice hyperholomorphic functions $f:U \to \hh$, we defined
\begin{equation}\label{SFL}
f(T)={{1}\over{2\pi }} \int_{\partial (U\cap \mathbb{C}_I)} S_L^{-1} (s,T)\  ds_I \ f(s),
\end{equation}
where $ds_I=- ds I$, and $I$ is any purely imaginary quaternion such that $I^2=-1$.
For a right slice hyperholomorphic function $f: U \to\hh$, we define
\begin{equation}\label{SFR}
f(T)={{1}\over{2\pi }} \int_{\partial (U\cap \mathbb{C}_I)} \  f(s)\ ds_I \ S_R^{-1} (s,T).
\end{equation}
These definitions are well posed since the integrals depend neither on the open set $U$ nor on the complex plane
$\mathbb{C}_I$. They are moreover obtained by writing $f(q)$ in terms of the corresponding slice hyperholomorphic Cauchy formula and formally replacing the scalar variable $q$ by the operator $T$, just as it is done in the Riesz-Dunford functional calculus for holomorphic functions in the complex linear setting.

Since the function $s\mapsto s^\alpha$ is both left and right slice-hyperholomorphic, we can define fractional powers of a vector operator  $T = e_1T_1+e_2T_2+e_3T_3$ using (\ref{SFL}) or (\ref{SFR}) as
\begin{equation}\label{AAREL}
T^\alpha={{1}\over{2\pi }} \int_{\partial (U\cap \mathbb{C}_I)} \  s^\alpha\ ds_I \ S_R^{-1} (s,T)
={{1}\over{2\pi }} \int_{\partial (U\cap \mathbb{C}_I)} S_L^{-1} (s,T)\  ds_I \, s^\alpha,
\end{equation}
if $\sigma_S(T) \subset U$ is contained in the domain of $s^{\alpha}$.

If $T$ is an unbounded operator then the $S$-spectrum is in general unbounded,
so to extend the $S$-functional calculus we need to impose restrictions on the class of functions and we have to assume
that $T$ is a  quaternionic  sectorial operator in order to guarantee that the integrals (\ref{SFL}) and (\ref{SFR}) are convergent.

The  quaternionic $H^\infty$-functional calculus based on the $S$-spectrum was initiated in \cite{Hinfty}. We consider an extended version of the quaternionic $H^\infty$-functional calculus  for a larger class of quaternionic operators and we prove several important properties.  We use this version of the functional calculus to define fractional powers of vector operators.

We first give meaning to the $S$-functional calculus for quaternionic sectorial operators taking slice hyperholomorphic functions
 with grow conditions such as
$$
|f(s)|\leq c(|s|^k+|s|^{-k}), \ \ {\rm for }\ \ c>0, \ k>0,
$$
via a slice hyperholomorphic Cauchy integral as in \eqref{AAREL}. This calculus extends naturally to the set $\EL[\sector{\omega}]$ of functions that are left slice hyperholomorphic on a suitable sector $\sector{\omega}$ around the positive real axis and  that have finite polynomial limits at $0$ and infinity. With this class of functions,  the $S$-functional calculus for quaternionic sectorial operators turns out to be well defined. 

We can extend this functional calculus so that it can even generate unbounded operators in the following way: we call an intrinsic function $e\in\EL[\sector{\omega}]$ a regularizer for a slice meromorphic function $f$, if $ef \in \EL[\sector{\omega}]$ and $e(T)^{-1}$ is injective. 
We define the quaternionic $H^\infty$-functional calculus as:
$$
f(T):=(e(T))^{-1}(e f)(T),
$$
where the operators $e(T)$ and $(e f)(T)$ are defined using the $S$-functional calculus for sectorial operators. If in particular $T$ is injective, then we can use the function
$$
e(s):=\Big(s(1+s^2)^{-1}\Big)^{k+1}, \ \ k\in \mathbb{N}
$$
as a regularizer, which yields exactly the $H^{\infty}$-functional calculus as introduced in \cite{Hinfty}.
Using this functional calculus we can define  the fractional powers of quaternionic unbounded operators.
\\
\\
We now explain our strategy to fractional diffusion processes. In order to do so, we first recall the classical approach to the fractional heat equation.
We denote by  $u$ the temperature on and by$\mathbf{q}$  the heat flow
and we set the thermal diffusivity equal to $1$.The heat equation is then deduced from the two laws
\begin{align}
\label{Fourier}\mathbf{q}  = -\nabla u \ \ &\text{(Fourier's law)}
\\
\label{Cons}\partial_t u  +\operatorname{div} \mathbf{q} = 0\ \ \ &\text{(Conservation of Energy)}
\end{align}
where $u$ and $\mathbf{q}$ are defined on $\mathbb{R}^3$.
Plugging the first relation into the second one we find
\[
\partial_t u  - \Delta u = 0.
\]
The fractional heat equation is an alternative model, that  takes into account  non local interactions and it is obtained by replacing the negative Laplacian in the heat equation by its fractional power so that
\begin{equation}\label{fracHeatEQ}
\partial_t u  +(-\Delta)^{\alpha} u = 0,\ \ \ \alpha\in (0,1),
\end{equation}

Our approach is different, very general, and in the case $\mathbf{q}  = -\nabla u$ it reduces to the fractional Laplace operator.
We explain the main points of our new strategy to fractional diffusion problems.
We identify
\[
\mathbb{R}^3 \ \cong \{ s\in\mathbb{H}: {\rm Re}(s) = 0\}
\]
and we consider the gradient $\nabla$ as the quaternionic nabla operator
\[
\nabla= e_1\partial_{x_1} + e_2 \partial_{x_2} +  e_3\partial_{x_3}.
\]
Instead of replacing the negative Laplacian in the heat equation, we wanted to replace the gradient in \eqref{Fourier} by its fractional power $\nabla^{\alpha}$
using the quaternionic $H^\infty$-functional before combining it with the law of conservation of energy.
But this cannot be done directly for the following reasons:
\begin{itemize}
\item[(i)] The $S$-spectrum of the $\nabla$ operator  on $L^2(\mathbb{R}^3,\mathbb{H})$ is the real line
\[
\sigma_{S}(\nabla) = \mathbb{R}
\]
and  $s^{\alpha}$ is not defined on $(-\infty,0)$. So the operator  $\nabla^{\alpha}$ cannot be defined on the entire $S$-spectrum.
\item[(ii)] The fractional power of a vector operator is not in general a vector operator, but it is a quaternionic operator of type
 $T_1+e_1T_2+e_2T_3+e_3T_3$. To such an operator  we cannot apply the divergence operator to get the fractional Laplace operator.
\end{itemize}
So we have to proceed as follows:
\begin{itemize}
\item[(ia)] First, instead of considering the fractional powers of the nabla operator $\nabla^{\alpha}$, 
we consider the projections of the fractional powers of $\nabla^{\alpha}$ indicated by  $f_{\alpha}(\nabla)$ to the subspace associated with the subset $[0+\infty)$ of the $S$-spectrum of $\nabla$, on which the function $s^{\alpha}$ is well defined.
\item[(iia)]  Second  we take just the vector part  ${\rm Vect}(f_{\alpha}(\nabla))=e_1T_1+e_2T_2+e_3T_3$ of the quaternionic operator $f_{\alpha}(\nabla)=T_0+e_1T_1+e_2T_2+e_3T_3$
so that  we can apply the divergence operator.
\end{itemize}
We point out that if  we proceed as in the points (ia) and (iia) for the gradient operator
 we re-obtain the classical fractional Laplace operator, but our approach is very general and
it is applicable to a large class of operators such as
\[
\widetilde{\nabla}= e_1\,a(x)\partial_{x_1} + e_2\, b(x)\partial_{x_2} +  e_3\,c(x)\partial_{x_3},
\]
where $a$, $b$, $c$ are suitable real valued functions that depend on the space variables $x=(x_1,x_2,x_3)$ and possibly also on time.
More precisely, the definition of $\nabla^{\alpha}$ only on the subspace associated to $[0,\infty)$  is given by
\[
f_{\alpha}(\nabla) v = \frac{1}{2\pi} \int_{-I\mathbb{R}} S_L^{-1}(s,\nabla)\,ds_{I}\, s^{\alpha} \nabla v,
\]
for $v:\mathbb{R}^3\to \mathbb{H}$  in $\dom(\nabla)$. This corresponds to the Balakrishnan formula, which is a consequence of the quaternionic $H^\infty$-functional calculus, in which only positive spectral values are taken into account.
With this definition and the surprising expression for the left $S$-resolvent operator
\[
S_L^{-1}(-tI ,\nabla) = (-tI + \nabla)\underbrace{(-t^2 + \Delta)^{-1}}_{= R_{-t^2}(-\Delta)},
\]
 the operator $f_{\alpha}(\nabla)$
 becomes with some computations 
$$
f_{\alpha}(\nabla) v
= \underbrace{ \frac{1}{2} (-\Delta)^{\frac{\alpha}{2}-1}\nabla^2 v}_{\mathrm{Scal} f_{\alpha}(\nabla)v} + \underbrace{\frac{1}{2} (-\Delta)^{\frac{\alpha-1}{2}}\nabla v}_{=\mathrm{Vec}f_{\alpha}(\nabla) v}.
$$
We define the scalar part  of the operator $f_{\alpha}(\nabla) $ applied to $v$ as
$$
\mathrm{Scal} f_{\alpha}(\nabla)v:=\frac{1}{2}  (-\Delta)^{\frac{\alpha}{2}-1}\nabla^2 v,
$$ 
and the vector part as
$$
\mathrm{Vec}f_{\alpha}(\nabla) v:=\frac{1}{2} (-\Delta)^{\frac{\alpha-1}{2}}\nabla v.
$$
Now we observe hat
\[
 {\rm div} \mathrm{Vec}f_{\alpha}(\nabla) v = -\frac{1}{2} (-\Delta)^{\frac{\alpha}{2}+1} v
\]
This proves that in the case of the gradient we get the same result, that is the fractional Laplacian.
 The fractional heat equation for $\alpha\in(1/2,1)$
\[
\partial_t u(t,x) + (-\Delta)^{\alpha} u(t,x) = 0
\]
 can hence be written as
\[
\partial_t u(t,x) - 2 {\rm div}\left(\mathrm{Vec} f_{\beta}(\nabla)v\right)  = 0, \qquad \beta = 2\alpha - 1.
\]
 For any suitable vector operator $T$, we hence propose the fractional evolution equation:
\[
\partial_t u(t,x) - 2 {\rm div}\left(\mathrm{Vec} f_{\beta}(T)v\right)  = 0.
\]
This allows us to study different fractional evolution problem for example a new fractional evolution equations
can be deduced when we consider the following Fourier's law:
\[
 T = e_1x_1\partial_{x_1}  + e_2x_2 \partial_{x_2} + e_3x_3\partial_{x_3}.
 \]
Working in the space $L^2(\mathbb{R}^3_{+}, \mathbb{H},d\mu)$ with
$$
\mathbb{R}_{+}^3=\{e_1x_{1} + e_2x_{2} + e_3x_{3}: x_{\ell}>0\}
$$
 and $d\mu =
(x_1x_2x_3)^{-1}dx$ we get
the operator
 $$
 {\rm Vec}  f_{\beta}(T) v(\xi)
 =
  \frac{1}{2(2\pi)^3} \int_{\mathbb{R}^3}\int_{\mathbb{R}^3}-|y|^{2\alpha}
  e^{e_1 \sum_{k=1}^3\xi_k y_k} e^{-e_1x\cdot {y}}
  \begin{pmatrix}
   e^{x_1}v_{\xi_1}(e^{x_1},e^{x_2},e^{x_3})
   \\
   e^{x_2} v_{\xi_2}(e^{x_1},e^{x_2},e^{x_3})
   \\
   e^{x_3} v_{\xi_3}(e^{x_1},e^{x_2}, e^{x_3})
   \end{pmatrix}
   \, dx\, dy.
 $$

This approach has several advantages:
\begin{itemize}
\item[(i)] It  modifies the Fourier law but keeps the law of Conservation of Energy.
\item[(ii)] It is applicable to a large class of operators that includes the gradient but also operators with variable
coefficients such as operator $\mathbf{q}(x,\partial_{x})$. Moreover, $\mathbf{q}$ can also depend on time.
\item[(iii)]
The fractional powers of the operator $\mathbf{q}(x,\partial_{x})$  is  very useful for non homogeneous materials.
 \item[(iv)] The fact that we keep the evolution equation in divergence form allows an immediate definition of the weak solution of the
fractional evolution problem.
 \item[(iv)]
To represent the fractional powers of an operator $T$ we have to write an explicit expression for the inverse of the operator 
$T^2-2 s_0T+|s|^2\mathcal{I}$ and this can be done on bounded or unbounded domains.
\end{itemize}

The development of the theory of slice hyperholomorphic functions, see the books \cite{ACSBOOK,ACSBOOK2,MR2752913}, was the crucial step towards the discovery of the $S$-spectrum, which was inspired by their
  Cauchy formula. Their Cauchy kernels suggested the notion of $S$-resolvent operators, and the formulations of the
  $S$-functional calculus as a consequence, see the original papers \cite{acgs,CGTAYLOR,JGA,CLOSED,DA}.\\
The quaternionic formulation of quantum mechanics, see \cite{adler,BvN}, has stimulated the research on the spectral theorem based on the $S$-spectrum, see \cite{ack,acks2}.
Using Clifford-algebra-valued slice hyperholomorphic functions, one can define the $S$-functional calculus for $n$-tuples of not necessarily commuting operators, see \cite{CSSJFA} and the book \cite{MR2752913}.\\
In the case quaternionic or vector operators are generators of groups or semigroups, see \cite{perturbation,MR2803786,GR},
 it is possible to define the function of the generators \cite{FUCGEN} via the Laplace-Stieltjes transform analogue to the Philips functional calculus.
\\
\\
The plan of the paper.
Section 1 is the introduction. Section 2 contains the definition and some properties of slice hyperholomorphic functions
 and of the $S$-functional calculus. In Section 3 we develop the quaternionic $H^\infty$-functional calculus for arbitrary quaternionic sectorial operators and several
  crucial properties like the spectral mapping theorem. In Section 4
we define the fractional powers of quaternionic operators using the results of Section 3. In Section 5 we develop the spectral theory based on the $S$-spectrum of 
the nabla operator and we give applications to fractional diffusion processes.

\section{Preliminaries}
The skew-field of quaternions consists of the real vector space 
\[
\hh:=\left\{p = \xi_0 + \sum_{i=1}^3\xi_ie_i: \xi_i\in\rr\right\},
\]
 which is endowed with an associative product with unity $1$ such that
$e_i^2 = - 1$ and $e_ie_j = -e_je_i$ for $i,j\in\{1,2,3\}$ with $i\neq j$.
The real part of a quaternion $p = \xi_0 + \sum_{i=1}^3\xi_ie_i$ is defined as $\Re(p) := \xi_0$, its imaginary part as $\underline{p} := \sum_{i=1}^3\xi_i e_i$ and its conjugate as $\overline{p} := \Re(p) - \underline{p}$.

Each element of the set
\[\SS := \{ p\in\hh: \Re(p) = 0, |p| = 1 \}\]
 is a square-root of $-1$ and is therefore called an imaginary unit. For any $I\in\SS$, the subspace $\cc_I := \{p_0 + I p_1: p_0,p_1\in\rr\}$ is an isomorphic copy of the field of complex numbers. If $I,J\in\SS$ with $I\perp J$, set $K=IJ = -JI$. Then $1$, $I$, $J$ and $K$ form an
 orthonormal basis of $\hh$ as a real vector space and $1$ and $J$ form a basis of $\hh$ as a left or right vector space over the complex plane $\cc_I$, that is
 \[ \hh = \cc_I + \cc_I J \quad\text{and}\quad \hh = \cc_I + J\cc_I.\]
 Any quaternion $p$ belongs to such a complex plane: if we set
 \[I_p := \begin{cases}\frac{1}{|\underline{p}|}\underline{p},& \text{if  }\underline{p} \neq 0 \\ \text{any }I\in\SS, \quad&\text{if }\underline{p}  = 0,\end{cases}\]
 then $p = p_0 + I_p p_1$ with $p_0 =\Re(p)$ and $p_1 = |\underline{p}|$. The set
 \[
 [p] := \{p_0 + Ip_1: I\in\SS\},
 \]
is a 2-sphere, that reduces to a single point if $p$ is real. The following result is well-known.
\begin{lemma}\label{HComLem}
If $p\in\hh$ and $h\in\hh\setminus\{0\}$, then $q =  h^{-1} p h\in [p]$. Conversely, if $q\in [p]$, then there exists $h\in\hh\setminus\{0\}$ such that $q = h^{-1} p h$. 
\end{lemma}

\subsection{Slice hyperholomorphic functions} The theory of complex linear operators is based on the theory of holomorphic functions. In a similar way, the theory of quaternionic linear operators is based on the theory of slice hyperholomorphic functions. We shall give a brief introduction to this field as it is fundamental for the rest of the paper. The proofs of the results stated in this subsection can be found in the book \cite{MR2752913}.

\begin{definition}
A set $U\subset\hh$ is called
\begin{enumerate}[label = (\roman*)]
\item axially symmetric if $[p]\subset U$ for any $p\in U$ and
\item a slice domain if $U$ is open, $U\cap\rr\neq \emptyset$ and $U\cap\cc_I$ is a domain for any $I\in\SS$.
\end{enumerate}
\end{definition}

\begin{definition}\label{sHolDef}
Let $U\subset\hh$ be open and axially symmetric. A function $f: U \to \hh$ is called left slice hyperholomorphic, if it is of the form
\begin{equation}
\label{lHolDef}f(p) = \alpha(p_0,p_1) + I_p \beta(p_0,p_1) \quad \forall p\in U,
\end{equation}
where $\alpha$ and $\beta$ are functions that take values in $\hh$, satisfy the compatibility condition
\begin{equation}\label{CCond}
\alpha(p_0,p_1) = \alpha(p_0,-p_1)\quad\text{and}\quad \beta(p_0,p_1) = -\beta(p_0,-p_1)
\end{equation}
and the Cauchy-Riemann-differential equations
\begin{equation}\label{CR}
\frac{\partial}{\partial p_0} \alpha(p_0,p_1) = \frac{\partial}{\partial p_1} \beta(p_0,p_1)\quad\text{and}\quad \frac{\partial}{\partial p_1}\alpha(p_0,p_1) = - \frac{\partial}{\partial p_0} \beta(p_0,p_1).
\end{equation}
A function $f: U \to \hh$ is called right slice hyperholomorphic, if it is of the form
\begin{equation}\label{rHolDef}
f(p) = \alpha(p_0,p_1) +\beta(p_0,p_1)  I_p \quad \forall p\in U,
\end{equation}
with functions $\alpha$ and $\beta$ satisfying \eqref{CCond} and \eqref{CR}. Finally, a left slice hyperholomorphic function $f = \alpha + I \beta$ is called intrinsic, if $\alpha$ and $\beta$ are real-valued.

We denote the set of all left slice hyperholomorphic functions on $U$ by $\lhol(U)$, the set of all right slice hyperholomorphic functions on $U$ by $\rhol(U)$ and the set of all intrinsic functions by $\intrin(U)$. Finally, we say that $f\in\lhol(C)$, $f\in\rhol(C)$ or $f\in\intrin(C)$ for an arbitrary axially symmetric set $C$, if there exists an axially symmetric open set $U$ with $C\subset U$ such that $f\in\lhol(U)$, $f\in\rhol(U)$ resp. $f\in\intrin(U)$.
\end{definition}
\begin{remark}
Observe that any quaternion $p$ can be represented using two different  imaginary units, namely $p = p_0 + I_p p_1 = p_0 + (-I_p)(-p_1)$. If $p\in\rr$, then we can even choose any imaginary unit we want in this representation. The compatibility condition \eqref{CCond} assures that the choice of this imaginary unit is actually irrelevant.
\end{remark}
Important examples of slice hyperholomorphic functions are power series with quaternionic coefficients: series of the form $\sum_{n=0}^{+\infty}p^na_n$ are left slice hyperholomorphic and series of the form $\sum_{n=0}^{\infty} a_np^n$ are right slice hyperholomorphic on their domain of convergence. A power series is intrinsic if and only if its coefficients are real.
Conversely, any slice hyperholomorphic function can be expanded into a power series at any real point in its domain of definition.
\begin{definition}
The slice-derivative of a function $f\in\lhol(U)$ is defined as
\[\sderiv  f(p) = \lim_{\cc_{I_p}\ni q\to p} (q-p)^{-1}(f(q)-f(p))\qquad\text{for $p = p_0 + I_pp_1\in U$,} \]
where $\lim_{\cc_{I_p}\ni q\to p} g(q)$ denotes the limit of a function $g$ as $q$ tends to $p$ in $\cc_{I_p}$.
The slice-derivative of a function $f\in\rhol(U)$ is defined as
\[\sderiv  f(p) = \lim_{\cc_{I_p}\ni q\to p} (f(q)-f(p))(q-p)^{-1}\qquad\text{for $p = p_0 + I_pp_1\in U$.} \]
\end{definition}
\begin{corollary}\label{SDProp}
The slice derivative of a left (or right) slice hyperholomorphic function is again left (or right) slice hyperholomorphic. Moreover, it coincides with the derivative with respect to the real part, that is
\[\sderiv  f(p) = \frac{\partial}{\partial p_0} f(p)\quad\text{for $ p = p_0 + I_p p_1$. }\]
\end{corollary}
\begin{theorem}\label{Taylor}
If $f$ is left slice hyperholomorphic on the ball $B_r(\alpha)$ of radius $r>0$ centered at $\alpha\in\rr$, then
\[f(p) = \sum_{n=0}^{+\infty} (p-\alpha)^n \frac{1}{n!}\sderiv ^n f(\alpha)\quad\text{for $p\in B_r(\alpha)$.}\]
If $f$ is right slice hyperholomorphic on $B_r(\alpha)$, then
\[f(p) = \sum_{n=0}^{+\infty}\frac{1}{n!}\sderiv ^n f(\alpha) (p-\alpha)^n \quad\text{for $x\in B_r(\alpha)$.}\]
\end{theorem}

The importance of the class of intrinsic functions is due to the fact that the multiplication and the composition with intrinsic functions preserve slice hyperholomorphicity. This is not true for arbitrary slice hyperholomorphic functions.
\begin{corollary}\label{HolStruct}
 If $f\in\intrin(U)$ and $g\in\lhol(U)$, then $fg\in\lhol(U)$. If $f\in\rhol(U)$ and $g\in\intrin(U)$, then $fg\in\rhol(U)$.

 If $g\in\intrin(U)$ and $f\in\lhol(g(U))$, then $f\circ g\in \lhol(U)$. If $g\in\intrin(U)$ and $f\in\rhol(g(U))$, then $f\circ g\in \rhol(U)$.
\end{corollary}

An intrinsic function $f = \alpha + I\beta$ is both left and right slice hyperholomorphic because $\alpha$ and $\beta$ commute with the imaginary unit $I$. The converse is not true: the constant function $p\mapsto b\in\hh\setminus\rr$ is left and right slice hyperholomorphic, but it is not intrinsic. There are several ways to characterise intrinsic functions.
\begin{corollary}\label{IntrinChar}
If $f\in\lhol(U)$ or $f\in\rhol(U)$, then the following statements are equivalent
\begin{enumerate}[label = (\roman*)]
\item The function $f$ is intrinisic.
\item We have $f(\overline{p}) = \overline{f(p)}$ for any $p\in U$. 
\item We have $f(U\cap\cc_I) \subset \cc_I$ for any $I\in\SS$.
\end{enumerate}
\end{corollary}

\begin{lemma}\label{HolLem}
Let $U\subset\hh$ be an axially symmetric open set. If $f\in\lhol(U)$, then for any $I\in\SS$ the restriction $f_I := f|_{U\cap\cc_I}$ is left holomorphic, i. e.
\begin{equation}\label{SplitEQL}
 \frac{1}{2}\left(\frac{\partial}{\partial p_0} f(p) + I\frac{\partial}{\partial p_1} f(p)\right) = 0,\qquad \forall p = p_0 + Ip_1 \in U\cap\cc_I.
 \end{equation}
If $f\in\rhol(U)$, then for any $I\in\SS$ the restriction $f_I := f|_{U\cap\cc_I}$ is right holomorphic, i. e.
\begin{equation}\label{SplitEQR}
\frac{1}{2}\left(\frac{\partial}{\partial p_0} f(p) + \frac{\partial}{\partial p_1} f(p)I\right) = 0,\qquad \forall p = p_0 + Ip_1 \in U\cap\cc_I.
\end{equation}
\end{lemma}
\begin{remark}
At the beginning of the development of the theory, slice hyperholomorphic functions were defined as functions that satisfy the properties mentioned in \Cref{HolLem}, which yields a potentially larger class of functions. The representation formula in \Cref{RepFo} assures that such functions can be represented in the form \eqref{lHolDef} resp. \eqref{rHolDef}---with the original definition, this formula however only holds for functions that are defined on axially symmetric slice domains. For applications in operator theory, it is therefore more convenient to start directly from \Cref{sHolDef}. The critical function theoretic results such as the Cauchy formula depend on the representation of the form $f = \alpha + I \beta $ resp. $f = \alpha + \beta I$ and not on the properties in \Cref{HolLem}. At the same time, for instance in order to generate spectral projections, it is essential to consider functions that are defined on sets that are not connected and not only functions that are defined on slice domains.
\end{remark}

The values of a slice hyperholomorphic function are uniquely determined by its values on an arbitrary complex plane $\cc_I$. Therefore, any function that is holomorphic on a suitable subset of a complex plane possesses a unique slice hyperholomorphic extension.
\begin{theorem}[Representation Formula]\label{RepFo}
Let $U\subset \hh$ be an axially symmetric open set and let $I\in\SS$. For any $p = p_0 + I_p p_1\in U$ set $p_I := p_0 + Ip_1$. If $f\in\lhol(U)$, then
\[f(p) = \frac12(1-I_pI)f(p_I) + \frac12(1+I_pI)f(\overline{p_I}) \quad\text{for all $p\in U.$}\]
If $f\in\rhol(U)$, then
\[f(p) = f(p_I)(1-II_p)\frac12 + f(\overline{p_I})(1+II_p)\frac12 \quad\text{for all $p\in U$.}\]
\end{theorem}

\begin{corollary}\label{extLem}
Let $I\in\SS$ and let $f:O\to\hh$ be real differentiable, where $O$ is a domain in $\cc_I$ that is symmetric with respect to the real axis.
We define the axially symmetric hull $[O]: = \bigcup_{z\in O}[z]$ of $O$.
\begin{enumerate}[label = (\roman*)]
\item If $f$ satisfies \eqref{SplitEQL}, then there exists a unique left slice hyperholomorphic extension of $f$ to~$[O]$.
\item If $f$ satisfies \eqref{SplitEQR}, then there exists a unique right slice hyperholomorphic extension of $f$ to~$[O]$.
\end{enumerate}
\end{corollary}

\begin{remark}
If $f$ has a left and a right slice hyperholomorphic extension, they do not necessarily coincide. Consider for instance the function $z\mapsto bz$ on $\cc_I$ with a constant $b\in\cc_I\setminus\rr$. Its left slice hyperholomorphic extension to $\hh$ is $p\mapsto pb$, but its right
slice hyperholomorphic extension is $p\mapsto bp$.
\end{remark}

Developing a theory of zeros of slice hyperholomorphic functions and in particular defining a proper notion of order is a nontrivial task. However, we shall only be interested in a certain type of zeros.
\begin{definition}
Let $f\in \lhol(U)$ (or $f\in\rhol(U)$). A point $p\in U\setminus\rr$ is called a spherical zero of $f$ if $f(\tilde{p}) = 0$ for any $\tilde{p}\in[ p ]$. In this case there exist unique $n\in\nn$ and $\tilde{f}\in\lhol(U)$ (resp. $\tilde{f}\in\rhol(U)$) such that $\tilde{f}$ does not have a spherical zero at $[p]$ and such that $f(q) = \Q_{p}(q)^n  \tilde{f}(q)$ (resp. $f(q) =  \tilde{f}(q)\Q_{p}(q)^n $) with 
\[
\Q_{p}(q) = q^2 - 2 \Re(p) q + |p|^2.
\]
 We call $n$ the order of the spherical zero $[p]$.
\end{definition}
Observe that  an intrinsic function $f = \alpha + I \beta$ satisfies $f(p) = 0$ if and only if $\alpha(p_0,p_1) = 0$ and $\beta(p_0,p_1) = 0$ as $\alpha$ and $\beta$ take real values. Hence any zero of an intrinsic function is either real or  spherical. 

\begin{corollary}Any intrinsic polynomial $p$ can be factorised into a product of the form
\[P(p) = \prod_{\ell=1}^N (\lambda_{\ell}-p)^{n_{\ell}}\prod_{\kappa=1}^M\Q_{s_\kappa}(p)^{m_{\kappa}}\]
where $\lambda_{1},\ldots, \lambda_{N}$ are the real zeros of $P$ and $[s_1],\ldots, [s_{M}$ are the spherical zeros of $P$ and $n_{\ell}$ and $m_{\kappa}$ are the orders of $\lambda_{\ell}$ resp $[s_{\kappa}]$.
\end{corollary}
The pointwise product of two slice hyperholomorphic functions is in general not slice hyperholomorphic. There exists however a regularised product that preserves slice hyperholomorphicity.
\begin{definition}
For $f = \alpha + I\beta, g= \gamma + I\delta \in \lhol(U)$, we define their left slice hyperholomorphic product as
\[ f\prodL g  =  (\alpha\gamma - \beta\delta) + I (\alpha\delta + \beta\gamma).\]
For $f = \alpha + \beta I, g= \gamma + \delta I \in \rhol(U)$, we define their right slice hyperholomorphic product as
\[ f\prodR g =  (\alpha\gamma - \beta\delta) +  (\alpha\delta + \beta\gamma)I.\]
\end{definition}
The slice hyperholomorphic product is associative and distributive, but in general not commutative. If however $f$ is intrinsic, then $f\prodL g$ coincides with the pointwise product $fg$ and $f\prodL g = fg = g\prodL f$. Similarly, if $g$ is intrinisc, then $f\prodR g$ coincides with the pointwise product $fg$ and $f\prodR g= fg = g\prodR f$.
\begin{definition}
We define for $f = \alpha + I\beta\in\lhol(U)$ its slice hyperholomorphic conjugate as $f^c = \overline{\alpha} + I \overline{\beta}$ and its symmetrisation as $f^s = f\prodL f^c = f^c\prodL f$. Similarly, we define for $f = \alpha + \beta I\in\rhol(U)$ its slice hyperholomorphic conjugate as $f^c = \overline{\alpha} +  \overline{\beta} I$ and its symmetrisation as $f^s = f\prodR f^c = f^c\prodR f$.
\end{definition}
The symmetrisation of a left slice hyperholomorphic function $f = \alpha + I \beta$ is explicitly given by
\[
 f^s = |\alpha|^2 - |\beta|^2 + I 2\Re\left(\alpha \overline{\beta}\right).
\]
Hence, it is an intrinsic function. It is $f^s(p) = 0$ if and only if $f(\tilde{p}) = 0$ for some $\tilde{p}\in [p]$. Furthermore, one has
\begin{equation}\label{Piul}
 f^c (p ) = \overline{\alpha(p_0,p_1)} + I_p \overline{\beta(p_0,p_1)} = \overline{\alpha(p_0,p_1)} + \overline{\beta(p_0,p_1)(-I_p)} = \overline{f(\overline{p})} 
 \end{equation}
and an easy computation shows that 
\[
 f\prodL g(p) = f(p) g\left(f(p)^{-1}pf(p)\right) \quad \text{if $f(p) \neq 0$}
\]
and or $f(p) \neq 0$ it is
\begin{equation}\label{IIoJU}
f^s(p) = f(p) f^c\left(f(p)^{-1} p f(p)\right) = f(p) \overline{f\left(\overline{f(p)^{-1} p f(p)}\right)} = f(p) \overline{f\left( f(p)^{-1}\overline{p}f(p)\right)}.
\end{equation}
Similar computations hold true in the right slice hyperholomorphic case. Finally, if $f$ is intrinsic, then $f^c(p) = f(p)$ and $f^s(p) = f(p)^2$.

\begin{corollary}The following statements hold true.
\begin{enumerate}[label = (\roman*)]
\item For $f\in\lhol(U)$ with $f\not\equiv 0$, its slice hyperholomorphic inverse $f^{-\prodL}$ that satisfies $f^{-\prodL}\prodL f = f\prodL f^{-\prodL}=1$ is given by $f^{-\prodL} = (f^s)^{-1} \prodL f^c = (f^s)^{-1} f^c$ and it is defined on $U\setminus [N_f]$, where $N_f = \{s\in U: f(s) = 0\}$.
\item For $f\in\rhol(U)$ with $f\not\equiv 0$, its slice hyperholomorphic inverse $f^{-\prodR}$ that satisfies $f^{-\prodR}\prodR f = f\prodR f^{-\prodR} = 1$ is given by $f^{-\prodR} =  f^c \prodR  (f^s)^{-1}=  f^c(f^s)^{-1}$. 
\item If $f \in \intrin(U)$ with $f\not\equiv 0$, then $f^{-\prodL} = f^{-\prodR} = f^{-1}$.
\end{enumerate}
\end{corollary}
We will later need that $|f^{-\prodL}|$ is in a certain sense comparable to $1/|f|$. Since $f^s$ is intrinsic, we have $|f^s(p)| = |f^s(\tilde{p})|$ for any $\tilde{p}\in [p]$. Since $f(p)pf(p)^{-1} \in [p]$ by \Cref{HComLem}, we find for $f(p)\neq 0$, because of \eqref{IIoJU}. that
\[
 \left|f^s(p)\right| = \left| f^s\left( f(p) p f(p)^{-1}\right)\right| = \left| f\left(f(p) p f(p)^{-1} \right) \overline{f \left(\overline{p}\right)}\right| = \left| f\left(f(p) p f(p)^{-1} \right)\right|\left|f \left(\overline{p}\right)\right|.
\]
Therefore we have, because of \eqref{Piul}, that
\begin{equation*}
 \left|f^{-\prodL}(p)\right| = \left|f^s(p)^{-1}\right| \left|f^c(p)\right| = \frac{1}{ 
 \left| 
 f\left(f(p) p f(p)^{-1} \right) \right| \left|f \left(\overline{p}\right)
 \right|
 }\left|f\left(\overline{p}\right)\right| =  \frac{1}{\left|f\left(f(p)\overline{p} f(p)^{-1}\right)\right|}
 \end{equation*}
 and so
 \begin{equation}\label{SEst2}
 \left| f^{-\prodL}(p)\right| = \frac{1}{|f(\tilde{p})|}\quad \text{ with } \tilde{p} = f(p)\overline{p}f(p)^{-1} \in [p].
 \end{equation}
An analogous estimate holds for the slice hyperholomorphic inverse of a right slice hyperholomorphic function.

Finally, slice hy\-per\-ho\-lo\-mor\-phic functions satisfy an adapted version of Cauchy's integral theorem and an integral formula of Cauchy-type with a modified kernel.
\begin{definition}
We define the  left slice hyperholomorphic Cauchy kernel as
\[S_L^{-1}(s,p) = -(p^2-2\Re(s)p + |s|^2)^{-1}(p-\overline{s})\quad\text{for }p\notin[s]\]
and the right slice hyperholomorphic Cauchy kernel as
\[S_R^{-1}(s,p) = -(p-\overline{s})(p^2-2\Re(s)p + |s|^2)^{-1}\quad\text{for }p\notin[s].\]
\end{definition}
\begin{corollary}
The left slice hyperholomorphic Cauchy-kernel $S_L^{-1}(s,p)$ is left slice hyperholomorphic in the variable $p$ and right slice hy\-per\-ho\-lo\-mor\-phic in the variable $s$ on its domain of definition. Moreover, we have $S_R^{-1}(s,p) = - S_L^{-1}(p,s)$.
\end{corollary}
\begin{remark}
If $p$ and $s$ belong to the same complex plane, they commute and the slice hyperholomorphic Cauchy-kernels reduce to the classical one:
\[ \frac{1}{s-p} = S_L^{-1}(s,p) = S_R^{-1}(s,p).\]
\end{remark}
\begin{theorem}[Cauchy's integral theorem]
Let $U\subset\hh$ be an axially symmetric open set, let $I\in\SS$ and let $D_I$ be an open subset of $O\cap\cc_I$ with $\overline{D_I}\subset O\cap\cc_I$ such that its boundary consists of a finite number of continuously differentiable Jordan curves. For any $f\in\rhol(U)$ and $g\in\lhol(U)$, it is
\[\int_{\partial D_I}f(s)\,ds_I\,g(s) = 0,\]
where $ds_I = -I\ ds$.
\end{theorem}
\begin{definition}
An axially symmetric open set $U\subset \hh$ is called a slice Cauchy domain if $U\cap\cc_I$ is a Cauchy domain for any $I\in\SS$, that is
\begin{enumerate}[label = (\roman*)]
\item $U\cap\cc_I$ is open,
\item $U\cap\cc_I$ has a finite number of components (i.e. maximal connected subsets), the closures of any two of which are disjoint,
\item the boundary of $U\cap\cc_I$ consists of a finite positive number of closed piecewise continuously differentable Jordan curves.
\end{enumerate}
\end{definition}
\begin{remark}
A slice Cauchy domain is either bounded or has exactly one unbounded component. If it is unbounded, then the unbounded component contains a neighbourhood of infinity.
\end{remark}
\begin{theorem}[Cauchy's integral formula]\label{Cauchy}
Let $U\subset\hh$ be a slice Cauchy domain. Let $I\in\SS$ and set $ds_I = -I\, ds$. If $f$ is left slice hyperholomorphic on an open set that contains $\overline{U}$, then
\[f(p) = \frac{1}{2\pi}\int_{\partial(U\cap\cc_I)} S_L^{-1}(s,p)\,ds_I\, f(s)\quad\text{for all }p\in U.\]
If $f$ is right slice hyperholomorphic on an open set that contains $\overline{U}$, then
\[f(p) = \frac{1}{2\pi}\int_{\partial(U\cap\cc_I)}f(s)\, ds_I\, S_R^{-1}(s,p)\quad\text{for all }p\in U.\]
\end{theorem}

%
%

\begin{remark}\label{OpVal}
The results presented in this section can easily be extended to functions with values in a two-sided quaternionic Banach space. As in the complex setting, problems concerning vector-valued functions can be reduced to scalar problems by applying elements of the dual space. For the details see \cite{Alpay:2015a}.
\end{remark}
\subsection{The S-functional calculus}
The natural extension of the Riesz-Dunford-functional calculus for complex linear operators  to quaternionic linear operators is the so-called $S$-functional calculus. It is based on the theory of slice hyperholomorphic functions and follows the idea of the classical case: to formally replace the scalar variable $p$ in the Cauchy formula by an operator. The proofs of the results stated in this subsection can be found in \cite{acgs, MR2752913, DA}.

Throughout the paper, let $V$ denote a two-sided quaternionic Banach space. We denote the set of all bounded quaternionic right-linear operators on $V$ by $\boundOP(V)$ and the set of all closed and quaternionic right-linear operators on $V$ by $\closOP(V)$. 
For $T\in\closOP(V)$, we define
\[ \Q_{s}(T) := T^2 - 2\Re(s)T + |s|^2\id, \qquad \text{for $s\in\hh$}. \]
\begin{definition}
Let $T\in\closOP(V)$. We define the $S$-resolvent set of  $T$ as
\[\rho_S(T):= \{ s\in\hh: \Q_{s}(T)^{-1}\in\boundOP(V)\}\]
and the $S$-spectrum of $T$ as
\[\sigma_S(T):=\hh\setminus\rho_S(T).\]
If $s\in\rho_S(T)$, then $\Q_{s}(T)^{-1}$ is called the pseudo-resolvent of $T$ at $s$. Furthermore, we define the extended $S$-spectrum $\sigma_{SX}(T)$ as
\begin{equation*}
\sigma_{SX}(T) := \begin{cases} \sigma_{S}(T) & \text{if  $T$ is bounded,}\\
\sigma_{S}(T)\cup\{\infty\} & \text{if  $T$ is unbounded.}
\end{cases}
\end{equation*}
\end{definition}
The $S$-spectrum has properties that are similar to those of the spectrum of a complex linear operator.
\begin{theorem}\label{SpecProp}
Let $T\in\closOP(V)$.
\begin{enumerate}[label = (\roman*)]
\item The $S$-spectrum $\sigma_{S}(T)$ of $T$ is axially symmetric.
\item The $S$-spectrum $\sigma_{S}(T)$ is a closed subset of $\hh$ and the extended $S$-spectrum $\sigma_{SX}(T)$ is a closed and hence compact subset of $\hh_{\infty}:= \hh \cup\{\infty\}$.
\item If $T$ is bounded, then $\sigma_{S}(T)$ is nonempty and bounded by the norm of $T$, i.e. $\sigma_{S}(T)\subset \overline{B_{\|T\|}(0)}$, where $B_{\|T\|}(0)$ is the open ball of radius $\|T\|$ centered at $0$. 
\end{enumerate}
\end{theorem}
\begin{remark}
If $V$ is only a right-sided Banach space (i.e. one in which multiplication with scalars is only defined from the right side) then the space $\boundOP(V)$ of bounded right linear operators on $V$ is only a real Banach space and in particular multiples of the identity operator $\id$ are only defined for $s\in\rr$ as one would expect  $s\id$ to act as $(s\id) v = (\id s) v = \id (sv)$. For this reason, one usually works with two-sided Banach spaces when defining the $S$-functional calculus, because the definition of the $S$-resolvents (see \Cref{SResDef}) requires a quaternionic multiplication on $\boundOP(V)$. However, the $S$-spectrum itself can also be defined on one-sided Banach spaces because the operator $\Q_{s}(T)$ involves only real scalars. Its properties given in \Cref{SpecProp} still remain true in this case, since their proofs do not rely on a quaternionic multiplication on $\boundOP(V)$. One only has to pay attention when showing that $\sigma_{S}(T)$ is bounded by $\|T\|$ if $T$ is bounded, which is usually shown using a power series expansion of the $S$-resolvent that involves quaternionic coefficients, cf. \cite{MR2752913}. However,  \cite{CGTAYLOR} introduced a series expansion for the pseudo-resolvent $\Q_{s}(T)^{-1}$ that converges for $|s|>\|T\|$ and involves only real coefficients. This series expansion can therefore serve for showing that $\sigma_{S}(T)\subset \overline{B_{\|T\|}(0)}$ on one-sided Banach spaces such that all properties in \Cref{SpecProp} are also true in this setting. This will be important in the proof of \Cref{SpecIncSP}, where we encounter an operator on such a space.
\end{remark}
\begin{definition}\label{SResDef}
Let $T\in\closOP(V)$. For $s\in\rho_S(T)$, the left $S$-resolvent operator is defined as
\begin{equation}\label{SresolvoperatorL}
S_L^{-1}(s,T):= \Q_s(T)^{-1}\overline{s} -T\Q_s(T)^{-1}
\end{equation}
and the right $S$-resolvent operator is defined as
\begin{equation}\label{SresolvoperatorR}
S_R^{-1}(s,T):=-(T-\id \overline{s})\Q_s(T)^{-1}.
\end{equation}
\end{definition}
\begin{remark}\label{RkResExtension} One clearly obtains the right $S$-resolvent operator by formally replacing the variable $p$ in the right slice hyperholomorphic Cauchy kernel by the operator $T$. The same procedure yields
\begin{equation}\label{LResShort}
S_L^{-1}(s,T)v = -\Q_s(T)^{-1}(T-\overline{s}\id)v,\quad\text{for }v\in\dom(T)
\end{equation}
for the left $S$-resolvent operator. This operator is only defined  on  the domain $\dom(T)$ of $T$ and not on the entire space $V$. However,  $\Q_s(T)^{-1}Tv = T\Q_{s}(T)^{-1}v$ for $v\in\dom(T)$ and commuting $T$ and $Q_s(T)$ in \eqref{LResShort} yields \eqref{SresolvoperatorL}. For arbitrary $s\in\hh$, the operator $\Q_{s}(T) = T^2 - 2\Re(s)T + |s|^2\id$ maps $\dom(T^2)$ to $V$. Hence, the pseudo-resolvent $\Q_s(T)^{-1}$ maps $V$ to $\dom(T^2)\subset \dom(T)$ if $s\in\rho_S(T)$. Since $T$ is closed and $\Q_s(T)^{-1}$ is bounded, equation \eqref{SresolvoperatorL} then defines a continuous and therefore bounded right linear operator on the entire space $V$. Hence, the left resolvent $S_L^{-1}(s,T)$ is the natural extension of the operator \eqref{LResShort} to~$V$. In particular, if $T$ is bounded, then $S_L^{-1}(s,T)$ can directly be defined by \eqref{LResShort}.

If one considers left linear operators, then one must modify the definition
of the right $S$-resolvent operator for the same reasons.
\end{remark}
\begin{remark}
The $S$-resolvent operators reduce to the classical resolvent if $T$ and $s$ commute, that is
\[S_L^{-1}(s,T) = S_R^{-1}(s,T) = (s\id - T)^{-1}.\]
This is in particular the case if $s$ is real.
\end{remark}
The following lemma is crucial and can be shown by straightforward computations for bounded operators. In the case of unbounded operators, several additional technical difficulties have to be overcome. The respective proof can be found in \cite{FJTAMS}.
\begin{lemma}\label{ResHol2342}
Let $T\in\closOP(V)$. The map $s\mapsto S_L^{-1}(s,T)$ is a right slice-hyperholmorphic function on $\rho_S(T)$ with values in the two-sided  quaternionic Banach space $\boundOP(V)$. The map $s\mapsto S_R^{-1}(s,T)$ is a left slice-hyperholmorphic function on $\rho_S(T)$ with values in the two-sided quaternionic Banach space $\boundOP(V)$. 
\end{lemma}
The $S$-resolvent equation is the analogue of the classical resolvent equation in the quaternionic setting. It has first been proved for the case that $T$ is a bounded operator in \cite{acgs}. Later on it was generalized to the case that $T$ is unbounded in \cite{FJTAMS}.
\begin{theorem}[$S$-resolvent equation]Let $T\in\closOP(V)$. For  $s,p \in  \rho_S(T)$ with $s\notin[p]$, it is
\begin{equation}\label{SresEQ1}
\begin{split}
S_R^{-1}(s,T)S_L^{-1}(p,T)=&\big[[S_R^{-1}(s,T)-S_L^{-1}(p,T)]p
\\
&
-\overline{s}[S_R^{-1}(s,T)-S_L^{-1}(p,T)]\big](p^2-2s_0p+|s|^2)^{-1}.
\end{split}
\end{equation}
\end{theorem}
An easy straightforward computation (see also \cite[Lemma~3.18]{acgs}) shows that for any bounded operator $B\in\boundOP(V)$ and any $s,p$ with $s\notin [p]$
\begin{equation}\label{ComRel}
(s^2 - 2p_0s + |p|^2)^{-1}(s B - B\overline{p}) = (\overline{s}B - Bp)(p^2 - 2s_0 p + |s|^2)^{-1}.
\end{equation}
Applying this identity with $B = S_R^{-1}(s,T)  - S_L^{-1}(p,T)$ to \eqref{SresEQ1}, we find that the $S$-resolvent equation can also be reformulated as
\begin{equation}\label{SresEQ2}
\begin{split}
S_R^{-1}(s,T)S_L^{-1}(p,T)=&(s^2-2p_0s+|p|^2)^{-1}\big[[S_R^{-1}(s,T)-S_L^{-1}(p,T)]\overline{p}
\\
&
-s[S_R^{-1}(s,T)-S_L^{-1}(p,T)]\big].
\end{split}
\end{equation}
Finally, before introducing the $S$-functional calculus, we recall a technical lemma that will be important at several occasions \cite[Lemma~3.23]{acgs}.
\begin{lemma}\label{AlpayLemma}
Let $B\in\boundOP(V)$ and let $U$ be a bounded Cauchy domain. If $f\in\intrin\left(\overline{U}\right)$, then we have for $p\in U$ that
\[ B f(p) = \frac{1}{2\pi}\int_{\partial(U\cap\cc_I)}f(s)\,ds_I\,(\overline{s}B-Bp)\left(p^2-2s_{0}p + |s|^2\right)^{-1}.\]
\end{lemma}
Formally replacing the slice hyperholomorphic Cauchy-kernels  in the Cauchy-formula by the $S$-resolvent operators leads to the natural generalization of the Riesz-Dunford-functional calculus to quaternionic linear operators.

\begin{definition}[$S$-functional calculus for bounded operators]\label{SCalcBd}
Let $T\in\boundOP(V)$, choose $I\in\SS$ and set $ds_I = -I\, ds$. For $f\in\lhol(\sigma_S(T))$, we choose a bounded slice Cauchy domain $U\subset \fdom(f)$ with $\sigma_{S}(T)\subset U$ and define
\begin{equation}\label{EQLCalcBd}
f(T) := \frac{1}{2\pi}\int_{\partial(U\cap\cc_I)} S_L^{-1}(s,T)\, ds_I\, f(s).
\end{equation}
For $f\in\rhol(\sigma_S(T))$, we choose again an unbounded slice Cauchy domain $U\subset \fdom(f)$ with $\sigma_{S}(T)\subset U$ and define
\begin{equation}\label{EQRCalcBd}
f(T) := \frac{1}{2\pi}\int_{\partial(U\cap\cc_I)} f(s)\, ds_I\, S_R^{-1}(s,T).
\end{equation}
These integrals are independent of the choices of the slice domain $U$ and the imaginary unit $I\in\SS$.
\end{definition}
These definitions can be extended to closed operators. As in the complex case, this was originally done via a transformation and the $S$-functional calculus for bounded operators, cf. \cite{MR2752913}. Since this procedure works in the quaternionic setting only if the $S$-resolvent set of the operator contains a real point, we shall instead directly start from a Cauchy integral, cf. \cite{DA}.

We say that a function is left (or right) slice hyperholomorphic at infinity, if $f$ is left (or right) slice hyperholomorphic, if there exists $r>0$ such that the ball $B_{r}(0)$ of radius $r>0$ centered at $0$ is contained in the domain $\fdom(f)$ of $f$ and if $f(\infty) := \lim_{p\to\infty}f(p)$ exists.
\begin{definition}[$S$-functional calculus for closed operators]\label{SCalcCl}
Let $T\in\closOP(V)$ with $\rho_S(T) \neq 0$, choose $I\in\SS$ and set $ds_I = -I\, ds$. For $f\in\lhol(\sigma_{S}(T)\cup\{\infty\})$, we choose an unbounded slice Cauchy domain $U\subset \fdom(f)$ with $\sigma_{S}(T)\subset U$ and define
\begin{equation}\label{EQLCalcCl}
f(T) := f(\infty)\id + \frac{1}{2\pi}\int_{\partial(U\cap\cc_I)} S_L^{-1}(s,T)\, ds_I\, f(s).
\end{equation}
For $f\in\rhol(\sigma_S(T)\cup\{\infty\})$, we choose again a bounded slice Cauchy domain $U\subset \fdom(f)$ with $\sigma_{S}(T)\subset U$ and define
\begin{equation}\label{EQRCalcCl}
f(T) := f(\infty)\id + \frac{1}{2\pi}\int_{\partial(U\cap\cc_I)} f(s)\, ds_I\, S_R^{-1}(s,T).
\end{equation}
These integrals are independent of the choices of the slice domain $U$ and the imaginary unit $I\in\SS$.
\end{definition}
\begin{remark}\label{IntrinRemark}
If the function $f$ is intrinsic, then \eqref{EQLCalcBd} and \eqref{EQRCalcBd} resp. \eqref{EQLCalcCl} and \eqref{EQRCalcCl} yield the same operator, which is not immediate. (Indeed, if $f$ is both left and right slice hyperholomorphic but not intrinsic, then the two representations might yield different operators, cf. \cite{DA}!) However, the path of integration $\partial(U\cap\cc_I)$ is symmetric with respect to the real axis because $U$ is axially symmetric. We can thus split the above integrals into the sums of two integrals, one over the part of the path in the positive halfplane of $\cc_I$ and one over the part of the path in the negative halfplane of $\cc_I$. Exploiting the aformentioned symmetry and the fact that $f(\overline{p}) = \overline{f(p)}$, we end up with an integral of a function that is the product of real numbers and the operators $\Q_{s}(T)^{-1}$ and $T\Q_{s}(T)^{-1}$. Depending on whether we started from the integrals involving the right or the left $S$-resolvent, these objects appear in a different order, but the factors are the same. However, as these objects commute mutually we can change the order of the factors and switch  from one representation to the other. For details, we refer to \cite{DA}, but the technique is also used in the proof of \Cref{ComLem}.

This procedure can be used not only for the $S$-functional calculus but for any Cauchy integral of the form as in \eqref{EQLCalcBd} or \eqref{EQRCalcBd}. Keeping in mind this remark, we shall use the equivalence of these two representations for intrinsic functions at several occasions in this paper without explicitly proving it every time.
\end{remark}
Since $\sigma_{SX}(T) = \sigma_{S}(T)$ if $T$ is bounded and $\sigma_{SX}(T) = \sigma_{S}(T)\cup\{\infty\}$ if $T$ is unbounded, we shall,  for neatness, denote the classes of admissible functions by $\lhol(\sigma_{SX}(T))$ and $\rhol(\sigma_{SX}(T))$ in the following lemma.
\begin{lemma}\label{SCProp}
Let $T\in\closOP(V)$ with $\rho_S(T)\neq \emptyset$. The $S$-functional calculus has the following properties:
\begin{enumerate}[label = (\roman*)]
\item If $f\in\lhol(\sigma_{SX}(T))$ or $f\in\rhol(\sigma_{SX}(T))$, then $f(T)$ is a bounded operator.
\item If $f,g\in\lhol(\sigma_{SX}(T))$ and $a\in\hh$, then $(fa+g)(T) = f(T)a+g(T)$.  If $f,g\in\rhol(\sigma_{SX}(T))$ and $a\in\hh$, then $(af+g)(T) = af(T)+g(T)$.
\item If $f\in\intrin(\sigma_{SX}(T))$ and $g\in\lhol(\sigma_{SX}(T))$ or if $f\in\rhol(\sigma_{SX}(T))$ and $g\in\intrin(\sigma_{SX}(T))$, then $(fg)(T) = f(T)g(T)$.
\item If $g\in\intrin(\sigma_{SX}(T))$, then $ \sigma_S(g(T)) = g(\sigma_{SX}(T))$. Moreover $f(g(T)) = (f\circ g)(T)$ if $f\in \lhol(g(\sigma_S(T)))$ or $f\in\rhol(g(\sigma_S(T)))$.
\item If $P$ is an intrinsic polynomial of order $n\geq 1$, then $P$ is not slice hyperholomorphic at infinity. However, if $T$ is unbouned and $f\in\intrin(\sigma_{SX}(T))$ has a zero of order greater or equal to $n$ at infinity (i.e. $\lim_{p\to\infty}P(p)f(p) = 0$), then $P(T)f(T) = (Pf)(T)$.
\item If $\sigma$ is an open and closed subset of  $\sigma_{SX}(T)$, let $\chi_{\sigma}$ be equal to $1$ on an axially symmetric neighbourhood of $\sigma$ in $\hh_{\infty}$ and equal to $0$ on an axially symmetric neighbourhood of $\sigma_{SX}(T)\setminus \sigma$ in $\hh_{\infty}$. Then $\chi_{\sigma}\in\intrin(\sigma_{SX}(T))$ and $\chi_{\sigma}(T)$ is a projection onto an invariant subspace of $T$. Moreover, if we denote the restriction of $T$ to the range of $\chi_{\sigma}(T)$ by $T_{\sigma}$, then $\sigma_{SX}(T_{\sigma}) = \sigma$.
\end{enumerate}
\end{lemma}

\subsection{Fractional powers of operators}
In order to recall the main results on fractional powers of quaternionic operators that have already been shown, let us start with the definition of fractional powers of quaternionic scalars.
\begin{definition}\label{ScalLogDef}
The (slice hyperholomorphic) logarithm on $\hh$ is defined as
\[
\log s:=\ln |s|+I_s \arg(s)\qquad\text{for } s\in \mathbb{H}\setminus (-\infty, 0],
\]
where $\arg(s) =  \arccos(s_0/|s|)$ is the unique angle $\varphi\in[0,\pi]$ such that $s = |s| e^{I_s\varphi}$. 
\end{definition}
Observe that for $s=s_0\in[0,\infty)$ we have  $\arccos(s_0/|s|) = 0$ and so $\log s = \ln s$. Therefore, $\log s$ is well defined on the positive real axis and does not depend on the choice of the imaginary unit $I_s$. One has
\[ e^{\log s} = s  \quad\text{for } s\in\hh\]
and
\[ \log e^s = s \quad\text{for }s\in\hh\ \ \text{ with }|\underline{s}|<\pi.\]
The quaternionic logarithm is both left and right slice hyperholomorphic (and actually even intrinsic) on $\hh\setminus (-\infty,0]$ and for any $I\in\SS$ its restriction to the complex plane $\cc_I$ coincides with the principal  branch of the complex logarithm on $\cc_I$. 
We define the fractional powers of exponent $\alpha\in \mathbb{R}$ of a quaternion $s$ as
\[
s^{\alpha}:= e^{\alpha\log s}=e^{\alpha (\ln |s|+I_s \arccos(s_0/|s|))}, \ \ \  s\in \mathbb{H}\setminus (-\infty, 0].
\]
This function is obviously also left and right slice hyperholomorphic on $\hh\setminus(-\infty, 0]$. Note however that if we define fractional powers $s^{\alpha}$ with $\alpha\in\hh\setminus\rr$ by the above formula, we do not obtain a slice hyperholomorphic function because the composition of two slice hyperholomorphic functions is in general only slice hyperholomorphic if the inner function is intrinsic.

In \cite{FJTAMS} we showed that two approaches for defining fractional powers of complex linear operators can be generalized to the quaternionic setting. The first approach taken from \cite{Engel:2000} defines fractional powers $T^{-\alpha}$ with $0<\alpha < +\infty$ for a sectorial operator   $T\in\closOP(V)$  with bounded inverse, i.e. this operator satisfies
\begin{equation}\label{A1Hyp}
\left\| S_{R}^{-1}(s,T) \right\| \leq \frac{M}{1 + |s|} \quad \text{for $s\in (-\infty,0]$.}
\end{equation}
This estimate implies that there exist $0<\theta_0<\pi$ and $a_0>0$ such that  \eqref{A1Hyp} holds (with a different constant) true for any $s$ in  the open sector $\Sigma(\theta_0,a_0) = \{ s\in\hh : |\arg(s-a_0)|<\theta_0\} $. 
\begin{definition}
Let $T\in\closOP(V)$ satisfy \eqref{A1Hyp}, let $I\in\SS$ and let $\Gamma$ be a piecewise smooth path in $\Sigma\cap\cc_I$ that goes from $\infty e^{I\theta}$ to $\infty e^{-I\theta}$ with $ \theta_0<\theta < \pi$ avoiding $(-\infty,0]$. For $\alpha >0$ we define
\begin{equation}\label{ERERE}
T^{-\alpha}:= \frac{1}{2\pi}\int_{\Gamma}s^{-\alpha}\,ds_I\, S_{R}^{-1}(s,T).
\end{equation}
\end{definition}
This definition is independent of the choice of $\Gamma$ and $I$. The fractional powers satisfy $T^{-(\alpha + \beta)} = T^{-\alpha}T^{-\beta}$ and the mapping $\alpha\mapsto T^{-\alpha}$ is a strongly continuous semi-group of quaternionic linear operators. Moreover, besides other properties, for $\alpha\in(0,1)$ the operator $T^{-\alpha}$ has the following integral representation
\begin{equation}\label{A1Int}
T^{-\alpha} = -\frac{\sin(\alpha\pi)}{\pi}\int_{0}^{+\infty}t^{-\alpha}S_{R}^{-1}(-t,T)\,dt,
\end{equation}
which is analogue to the classical case, cf. also \Cref{TaInjProp}.

for a sectorial operator $T$, cf. \Cref{SectDef}, one can define fractional powers for $\alpha\in(0,1)$ via Kato's approach in \cite{Kato}: one can show that there exists a unique closed operator $B_{\alpha}$ such that the right $S$-resolvent operator of $B_{\alpha}$ at any point $s$ that is sufficiently close to the  negative real axis given by
\[ S_{R}^{-1}(s,B_{\alpha}) = \frac{\sin(\alpha\pi)}{\pi}\int_{0}^{+\infty} t^{\alpha}\left(s^2 - 2st^{\alpha}\cos(\alpha\pi) + t^{2\alpha}\right)^{-1} S_R^{-1}(-t,T)\,dt.\]
Since this integral corresponds to a slice hyperholomorphic integral as in the $S$-functional calculus resp. as in \eqref{A1Int} for the function $S_R^{-1}(s,p^{\alpha})$, in which the path of integration is transformed into the negative real axis, it is meaningful to define $T^{\alpha}:=B_{\alpha}$. If  $\omega$ is the spectral angle of $T$, then $T^{\alpha}$ is a sectorial operator with spectral angle $\alpha\omega$. Moreover, if $\alpha\omega < \pi/2$, then $T^{\alpha}$ is the infinitesimal generator of a strongly continuous semigroup that is analytic in time.

\section{The  $H^{\infty}$-functional calculus arbitrary for sectorial operators}\label{HInfty}
The $H^{\infty}$-functional calculus was originally introduced in \cite{McI1} by McIntosh in 1986. His approach was generalized to quaternionic sectorial operators  that are injective and have dense range in \cite{Hinfty} by Alpay et al. In order to introduce fractional powers of quaternionic linear operators in the most general setting, we shall now define the $H^{\infty}$-functional calculus for arbitrary sectorial operators following the strategy of \cite{Haase}. This approach does not require neither the injectivity of $T$ nor that $T$ has dense range. Several proofs do not need a lot of additional work and the strategies of the complex setting can be applied in a quite straightforward way. We shall therefore in particular focus on the proof of the chain rule and of the the spectral mapping theorem, where more severe technical difficulties appear.

In order to define sectorial operators we introduce the sector $\sector{\varphi }$ for $\varphi \in(0,\pi]$ as
\begin{gather*}
\Sigma_{\varphi } := \{s\in\hh: \arg(s) < \varphi  \}.
\end{gather*}
\begin{definition}\label{SectDef}
Let $\omega\in[0,\pi)$. An operator $T\in\closOP(V)$ is called sectorial of angle $\omega$ if
\begin{enumerate}[label = (\roman*)]
\item we have $\sigma_S(T)\subset \overline{\sector{\omega}}$ and
\item for every $\varphi \in(\omega,\pi)$ there exists a constant $C>0$ such that for $s\notin\overline{\sector{\varphi}}$
\begin{equation}\label{SectCond}
 \left\| S_L^{-1}(s,T)\right\| \leq \frac{C}{|s|} \quad\text{and}\quad \left\| S_R^{-1}(s,T)\right\| \leq \frac{C}{|s|}. 
 \end{equation}
 We denote the infimum of all these constants by $C_{\varphi}$ resp. by $C_{\varphi,T}$ if we also want to stress its dependence on $T$.
\end{enumerate}
We denote the set of all operators in $\closOP(V)$ that are sectorial of angle $\omega$ by $\sectOP(\omega)$. Furthermore, if $T$ is a sectorial operator, we call $\omega_T = \min\{\omega: T\in\sectOP(\omega)\}$ the spectral angle of $T$.

Finally, a family of operators $(T_{\ell})_{\ell\in\Lambda}$ is called uniformly sectorial of angle $\omega$ if $T_{\ell}\in\sectOP(\omega)$ for all $\ell\in\Lambda$ and $\sup_{\ell\in\Lambda} C_{\varphi,T_{\ell}} < \infty$ for all $\varphi\in(\omega,\pi)$.
\end{definition}

\begin{definition}
We say that a slice hyperholomorphic function $f$ has polynomial limit $c\in\hh$ in $\sector{\varphi}$ at $0$ if there exists $\alpha > 0$ such that $f(p) - c = O\left(|p|^{\alpha}\right)$ as $p\to 0$ in $\sector{\varphi}$ and that it has polynomial limit $\infty$ in $\sector{\varphi}$ at $0$ if $f^{-\prodL}$ resp. $f^{-\prodR}$ has polynomial limit $0$ at $0$ in $\sector{\varphi}$. (By \eqref{SEst2} this is equivalent to $1/|f(p)| \in O(|p|^{\alpha})$ for some $\alpha > 0$ as $p\to 0$ in $\sector{\varphi}$.)

Similarly, we say that $f$ has polynomial limit $c\in\hh_{\infty}$ at $\infty$ in $\sector{\varphi}$ if $p \mapsto f(p^{-1})$ has polynomial limit $c$ at $0$.
If a function has polynomial limit $0$ at $0$ or $\infty$, we say that it decays regularly at $0$ resp. $\infty$.
\end{definition}
Observe that the mapping $p \mapsto p^{-1}$ leaves $\sector{\varphi}$ invariant such that the above relation between polynomial limits at $0$ and $\infty$ makes sense.
\begin{definition}
Let $\varphi\in(0,\pi]$.  We define $\lholZI(\sector{\varphi})$ as the set of all bounded functions in $\lhol(\sector{\varphi})$ that decay regularly at $0$ and $\infty$. Similarly, we define $\rholZI(\sector{\varphi})$ and $\intrinZI(\sector{\varphi})$ as the set of all bounded functions in $\rhol(\sector{\varphi})$ resp. $\intrin(\sector{\varphi})$ that decay regularly at $0$ and $\infty$.
\end{definition}
The following Lemma is an immediate consequence of \Cref{HolStruct}.
\begin{lemma}
Let $\varphi\in(0,\pi]$.
\begin{enumerate}[label = (\roman*)]
\item If $f,g\in\lholZI(\sector{\varphi})$ and $a\in\hh$, then $fa+g\in\lholZI(\sector{\varphi})$. If in addition even $f\in\intrinZI(\sector{\varphi})$, then also $fg\in\intrinZI(\sector{\varphi})$.
\item If $f,g\in\rholZI(\sector{\varphi})$ and $a\in\hh$, then $af+g\in\rholZI(\sector{\varphi})$. If in addition even $g\in\intrinZI(\sector{\varphi})$, then also $fg\in\intrinZI(\sector{\varphi})$.
\item The space $\intrinZI(\sector{\varphi})$ is a real algebra.
\end{enumerate}
\end{lemma}
\begin{definition}[S-functional calculus for sectorial operators]\label{SectSCalc}
Let $T\in\sectOP(\omega)$. If $f\in\lholZI(\sector{\varphi})$ with $\omega < \varphi < \pi$, we choose $\varphi'$ with $\omega < \varphi' < \varphi$ and $I\in\SS$ and define
\begin{equation}\label{SectIntL}
 f(T) := \frac{1}{2\pi}\int_{\partial (\sector{\varphi'}\cap\cc_I) } S_L^{-1}(s,T)\,ds_I\, f(s).
\end{equation}
Similarly, for $f\in\rholZI(\sector{\varphi})$ with $\omega < \varphi < \pi$, we choose $\varphi'$ with $\omega < \varphi' < \varphi$ and $I\in\SS$ and define
\begin{equation}\label{SectIntR}
 f(T) := \frac{1}{2\pi}\int_{\partial (\sector{\varphi'}\cap\cc_I) }f(s) \,ds_I\, S_R^{-1}(s,T).
 \end{equation}
\end{definition}
\begin{remark}
Since $T$ is sectorial of angle $\omega$, the estimates in \eqref{SectCond}
 assure the convergence of the above integrals. A standard argument using the slice hyperholomorphic version of Cauchy's integral theorem show that the integrals are independent of the choice of the angle $\varphi'$ and standard slice hyperholmorphic techniques based on the representation formula show that they are independent of the choice of the imaginary unit $I$. For details we refer to the proof of \cite[Theorem~4.9]{Hinfty}. Finally, computations as in the proof of \cite[Theorem~3.12]{DA} show that \eqref{SectIntL} and \eqref{SectIntR} yield the same operator if $f$ is intrinsic, cf. \Cref{IntrinRemark}.
\end{remark}
If $T\in\sectOP(\omega)$, then $f(T)$ in \Cref{SectSCalc} can be defined for any function that belongs to $\lholZI(\sector{\varphi})$ for some $\varphi \in(\omega,\pi]$. We thus introduce a notation for the space of all such functions.
\begin{definition}
For $\omega\in(0,\pi)$, we define
$ \lholZI[\sector{\omega}] = \bigcup_{\omega<\varphi \leq \pi} \lholZI(\sector{\varphi})$ and similarly also $ \rholZI[\sector{\omega}] = \bigcup_{\omega<\varphi \leq \pi} \rholZI(\sector{\varphi})$ and $ \intrinZI[\sector{\omega}] = \bigcup_{\omega<\varphi \leq \pi} \intrinZI(\sector{\varphi})$.
\end{definition}

The following properties of the $S$-functional calculus for sectorial operators can be proved by standard slice-hyperholomorphic techniques as, cf. for instance \cite[Theorem~4.12]{Hinfty}.
\begin{lemma}\label{HZIProp}
If $T\in\sectOP(\omega)$, then the following statements hold true.
\begin{enumerate}[label = (\roman*)]
\item If $f\in\lholZI[\sector{\omega}]$ or $f\in\rholZI[\sector{\omega}]$, then the operator $f(T)$ is bounded.
\item If $f,g\in\lholZI[\sector{\omega}]$ and $a \in \hh$, then $(fa+g)(T) = f(T)a+g(T)$. If $f,g\in\rholZI[\sector{\omega}]$ and $a\in\hh$, then $(af+g)(T) = af(T) + g(T)$. 
\item If $f\in\intrinZI[\sector{\omega}]$ and $g\in\lholZI[\sector{\omega}]$, then $(fg)(T) = f(T)g(T)$. If $f\in\rholZI[\sector{\omega}]$ and $g\in\intrinZI[\sector{\omega}]$, then also $(fg)(T) = f(T)g(T)$.
\end{enumerate}
\end{lemma}
We recall that a closed operator $A\in\closOP(V)$ is said to commute with $B\in\boundOP(V)$, if $BA\subset AB$.  
\begin{lemma}\label{ComLem}
Let $T\in\sectOP(\omega)$ and $A\in\closOP(V)$ commute with $\Q_{s}(T)^{-1}$ and $T\Q_{s}(T)^{-1}$ for any $s\in\rho_S(T)$. Then $A$ commutes with $f(T)$ for any $f\in\intrinZI[\sector{\omega}]$. In particular $f(T)$ commutes with $T$ for any $f\in\intrinZI[\sector{\omega}]$.
\end{lemma}
\begin{proof}
If $f\in\intrinZI[\sector{\omega}]$, then for suitable $\varphi \in(\omega,\pi)$ and $I\in\SS$, we have
\begin{align*}
f(T) =& \frac{1}{2\pi}\int_{\partial(\sector{\varphi}\cap\cc_I)} f(s)\,ds_I S_R^{-1}(s,T) \\
 =&\frac{1}{2\pi}\int_{-\infty}^{0} f\left(-te^{I\varphi}\right)\left(-e^{I\varphi}\right) (-I) \left(-te^{-I\varphi}-T\right)\Q_{-te^{I\varphi}}(T)^{-1} \,dt\\
 &+\frac{1}{2\pi}\int_{0}^{+\infty} f\left(te^{-I\varphi}\right)\left(e^{-I\varphi}\right) (-I) \left(te^{I\varphi}-T\right)\Q_{te^{-I\varphi}}(T)^{-1} \,dt.
\end{align*}
After the changing $t\mapsto -t$ in the first integral, we find 

\begin{align*}
f(T) =&\frac{1}{2\pi}\int_{0}^{+\infty} f\left(te^{I\varphi}\right)\left(e^{I\varphi} I \right) \left(te^{-I\varphi}-T\right)\Q_{te^{I\varphi}}(T)^{-1} \,dt\\
 &+\frac{1}{2\pi}\int_{0}^{+\infty} f\left(te^{-I\varphi}\right)\left(-e^{-I\varphi}I\right) \left(te^{I\varphi}-T\right)\Q_{te^{-I\varphi}}(T)^{-1} \,dt\\
 =&\frac{1}{2\pi}\int_{0}^{+\infty} 2\Re\left[f\left(te^{I\varphi}\right) I t\right]\Q_{te^{I\varphi}}(T)^{-1} \,dt \\ 
& - \frac{1}{2\pi}\int_{0}^{+\infty} 2\Re\left[f\left(te^{I\varphi}\right)Ie^{I\varphi} \right] T\Q_{te^{I\varphi}}(T)^{-1} \,dt,
\end{align*}
where the last identity holds because $f(\overline{s}) = \overline{f(s)}$ as $f$ is intrinisic and $\Q_{te^{I\varphi}}(T)^{-1} = \Q_{te^{-I\varphi}}(T)^{-1}$.

If now $v\in\dom(A)$, then the fact that $A$ commutes with $\Q_{s}(T)^{-1}$ and $T\Q_{s}(T)^{-1}$ and any real scalar implies
\begin{align*}
f(T) A v = &\frac{1}{2\pi}\int_{0}^{+\infty} 2\Re\left[f\left(te^{I\varphi}\right) I t\right]\Q_{te^{I\varphi}}(T)^{-1}  Av\,dt \\ 
& - \frac{1}{2\pi}\int_{0}^{+\infty} 2\Re\left[f\left(te^{I\varphi}\right)Ie^{I\varphi} \right] T\Q_{te^{I\varphi}}(T)^{-1} Av\,dt\\
 = &A \frac{1}{2\pi}\int_{0}^{+\infty} 2\Re\left[f\left(te^{I\varphi}\right) I t\right]\Q_{te^{I\varphi}}(T)^{-1}  v\,dt \\ 
& - A \frac{1}{2\pi}\int_{0}^{+\infty} 2\Re\left[f\left(te^{I\varphi}\right)Ie^{I\varphi} \right] T\Q_{te^{I\varphi}}(T)^{-1} v\,dt = A f(T)v.
\end{align*}
We thus find $v \in \dom(Af(T))$ with $f(T) Av = Af(T)v$ and in turn $f(T)A \subset Af(T)$.

\end{proof}

\begin{lemma}\label{ResCanc}
Let $T\in\sectOP(\omega)$. If $\lambda\in(-\infty,0)$ and $f\in\lholZI[\sector{\omega}]$ then $s\mapsto (\lambda - s)^{-1}f(s)\in\lholZI[\sector{\omega}]$ and
\[ \left((\lambda - s)^{-1}f(s)\right)(T) = (\lambda-T)^{-1}f(T) = S_L^{-1}(\lambda,T)f(T).\]
Similarly, if $\lambda\in(-\infty,0)$ and $f\in\rholZI[\sector{\omega}]$ then $s\mapsto f(s) (\lambda - s)^{-1}\in\lholZI[\sector{\omega}]$ and
\[ \left(f(s)(\lambda - s)^{-1}\right)(T) = f(T) (\lambda-T)^{-1}= f(T)S_R^{-1}(\lambda,T).\]
\end{lemma}
\begin{proof}
Let $\lambda\in(-\infty,0)$ and observe that, since $\lambda$ is real, the $S$-resolvent equation \eqref{SresEQ1} turns into
\begin{equation*}
(\lambda - T)^{-1}S_L^{-1}(s,T)   = S_R^{-1}(\lambda, T) S_L^{-1}(s,T) = \left(S_R^{-1}(\lambda,T)- S_L^{-1}(s,T)\right)(s-\lambda)^{-1}.
\end{equation*}
If now $f\in\lholZI[\sector{\omega}]$, then for suitable $\varphi\in(\omega,\pi)$ and $I\in\SS$, we have
\begin{align*}
(\lambda\id - T)^{-1}f(T) =& \frac{1}{2\pi}\int_{\partial( \sector{\varphi}\cap\cc_I)} (\lambda\id-T)^{-1}S_L^{-1}(s,T)\,ds_I\,f(s)\\
=&\frac{1}{2\pi}\int_{\partial( \sector{\varphi}\cap\cc_I)}  \left(S_R^{-1}(\lambda,T)- S_L^{-1}(s,T)\right)(s-\lambda)^{-1}\,ds_I\,f(s)\\
=&S_R^{-1}(\lambda,T)\frac{1}{2\pi}\int_{\partial( \sector{\varphi}\cap\cc_I)} \,ds_I\,(s-\lambda)^{-1}f(s) \\
&+  \frac{1}{2\pi}\int_{\partial( \sector{\varphi}\cap\cc_I)}  S_L^{-1}(s,T)\,ds_I\,(\lambda-s)^{-1}f(s) = \left((\lambda- s)^{-1}f(s)\right)(T),
\end{align*}
where the last equality holds because $\frac{1}{2\pi}\int_{\partial( \sector{\varphi}\cap\cc_I)} \,ds_I\,(s-\lambda)^{-1}f(s) = 0$ by Cauchy's integral theorem.

\end{proof}

Similar to \cite{Haase}, we can extend the class of functions that are admissible to this functional calculus to the analogue of the extended Riesz class.

\begin{lemma}
For $0<\varphi <\pi$, we define
\[\EL(\sector{\varphi}) = \left\{ f(p) = \tilde{f}(p) + (1+p)^{-1}a + b: \tilde{f}\in\lholZI(\sector{\varphi}), a,b\in\hh \right\}\]
and similarly
 \[\ER(\sector{\varphi}) = \left\{ f(p) = \tilde{f}(p) + a(1+p)^{-1} + b: \tilde{f}\in\rholZI(\sector{\varphi}), a,b\in\hh \right\}.\]
 Finally, we define $\Eint(\sector{\varphi})$ as the set of all intrinsic functions in $\EL(\sector{\varphi})$, i.e.
\[\Eint(\sector{\varphi}) = \left\{ f(p) = \tilde{f}(p) + (1+p)^{-1}a + b: \tilde{f}\in\intrinZI(\sector{\varphi}), a,b\in\rr \right\}.\]
\end{lemma}
Simple calculations as in the classical case show the following two corollaries, cf. also \cite[Lemma~2.2.3]{Haase}.
\begin{corollary}
Let $0<\varphi<\pi$.
\begin{enumerate}[label = (\roman*)]\label{EStruct}
\item The set $\EL(\sector{\varphi})$ is a quaternionic right vector space and it is closed under multiplication with functions in $\Eint(\sector{\varphi})$ from the left.
\item The set $\ER(\sector{\varphi})$ is a quaternionic left vector space and it is closed under multiplication with functions in $\Eint(\sector{\varphi})$ from the right.
\item The set $\Eint(\sector{\varphi})$ is a real algebra.
\end{enumerate}
\end{corollary}
\begin{corollary}\label{EChar}
Let $0<\varphi <\pi$. A function $f\in\lhol(\sector{\varphi})$ (or $f\in\rhol(\sector{\varphi})$ or $f\in\intrin(\sector{\varphi})$) belongs to $\EL(\sector{\varphi})$ (resp. $\ER(\sector{\varphi})$ or $\Eint(\sector{\varphi})$) if and only if it is bounded and has finite polynomial limits at $0$ and infinity.
\end{corollary}

\begin{definition}
For $\omega\in(0,\pi)$, we denote $ \EL[\sector{\omega}] = \bigcup_{\omega < \varphi < \pi} \EL(\sector{\varphi})$ as well as $ \ER[\sector{\omega}] = \bigcup_{\omega < \varphi < \pi} \ER(\sector{\varphi})$ and $ \Eint[\sector{\omega}] = \bigcup_{\omega < \varphi < \pi} \Eint(\sector{\varphi})$.
\end{definition}
\begin{definition}\label{DefECalc}
Let $T\in\sectOP(\omega)$. We define for  $f\in\EL[\sector{\omega}]$ with $f(s) = \tilde{f}(s) + (1+s)^{-1} a+ b$ the bounded operator 
\[ f(T) := \tilde{f}(T) + (1+T)^{-1}a + \id b\]
and for  $f\in\ER[\sector{\omega}]$ with $f(s) = \tilde{f}(s) + a(1+s)^{-1} + b$ the bounded operator 
\[ f(T) := \tilde{f}(T) + a(1+T)^{-1} + b\id,\]
where $\tilde{f}(T)$ is intended as in \Cref{SectSCalc}.
\end{definition}

\begin{lemma}\label{Kakao}
Let $T\in\sectOP(\omega)$ and let $f \in\EL[\sector{\omega}]$. If $f$ is left slice hyperholomorphic at $0$ and decays regularly at infinity, then 
\begin{equation}\label{CIDL1}
 f(T) = \frac{1}{2\pi}\int_{\partial(U(r)\cap\cc_I)} S_L^{-1}(s,T)\,ds_I\,f(s),
 \end{equation}
with $I\in\SS$ arbitrary and $ U(r) = \sector{\varphi}\cup B_{r}(0)$, where $\varphi \in (\omega,\pi)$ is such that $f\in\EL(\sector{\varphi})$ and $r>0$ is such that $f$ is left slice hyperholomorphic on $\overline{B_{r}(0)}$. Moreover, if $f$ is left slice hyperholomorphic both at $0$ and at infinity, then 
\begin{equation}\label{CIDL2}
 f(T) = f(\infty)\id + \frac{1}{2\pi}\int_{\partial(U(r,R)\cap\cc_I)} S_L^{-1}(s,T)\,ds_I\,f(s),
 \end{equation}
with $I\in\SS$ arbitrary and $ U(r,R) = U(r) \cup (\hh\setminus B_{R}(0))$, where $\varphi \in (\omega,\pi)$ is such that $f\in\EL(\sector{\varphi})$, $r>0$ is such that $f$ is left slice hyperholomorphic on $\overline{B_{r}(0)}$ and $R>r$ is such that $f$ is left slice-hyperholmorphic on $\hh\setminus B_{R}(0)$.

Similarly, if $f\in\ER[\sector{\omega}]$, is right slice hyperholomorphic at $0$ and decays regularly at infinity, then 
\[ f(T) = \frac{1}{2\pi}\int_{\partial(U(r)\cap\cc_I)} f(s)\,ds_I\,S_R^{-1}(s,T),\]
with $I\in\SS$ arbitrary and $ U(r)$ chosen as above. Moreover, if $f\in\ER[\sector{\omega}]$ is right slice hyperholomorphic both at $0$ and at infinity, then 
\[ f(T) = f(\infty)\id + \frac{1}{2\pi}\int_{\partial(U(r,R)\cap\cc_I)} f(s)\,ds_I\, S_R^{-1}(s,T),
\]
with $I\in\SS$ arbitrary and $ U(r,R)$ is chosen as above. 

\end{lemma}
\begin{proof}
Let us first assume that $f\in\EL[\sector{\omega}]$ is left slice hyperholomorphic at $0$ and regularly decaying at infinity. Then $f(s) = \tilde{f}(s) + (1+s)^{-1}a$, where $\tilde{f}\in\lholZI(\sector{\varphi'})$ with $\omega < \varphi < \varphi'$, and the function $\tilde{f}$ is moreover also left slice hyperholomorphic at $0$. For  $I\in\SS$ and $\omega < \varphi < \varphi'$, we therefore have
\begin{align*}
 &\frac{1}{2\pi}\int_{\partial(U(r)\cap\cc_I)} S_L^{-1}(s,T)\,ds_I\,f(s) \\
 =& \frac{1}{2\pi}\int_{\partial(U(r)\cap\cc_I)} S_L^{-1}(s,T)\,ds_I\,\tilde{f}(s) + \frac{1}{2\pi}\int_{\partial(U(r)\cap\cc_I)} S_L^{-1}(s,T)\,ds_I\,(1+s)^{-1}a. 
 \end{align*}
If $r'>r>0$ is sufficiently small such that $\tilde{f}$ is left slice hyperholomorphic at $\overline{B_{r'}(0)}$, then Cauchy's integral theorem implies that the value of the first integral remains constant as $r$ varies. Letting $r$ tend to $0$ we find that this integral equals $\tilde{f}(T)$ in the sense of \Cref{SectSCalc}. For the second integral we find that
\begin{align*}
&\frac{1}{2\pi}\int_{\partial(U(r)\cap\cc_I)} S_L^{-1}(s,T)\,ds_I\,(1+s)^{-1}a \\
=& \lim_{R\to+\infty} \frac{1}{2\pi}\int_{\partial(U(r,R)\cap\cc_I)} S_L^{-1}(s,T)\,ds_I\,(1+s)^{-1}a  = (1+T)^{-1}a, 
\end{align*}
where the last identity can be deduced either from the compatibility of the $S$-functional calculus for closed operators with intrinsic polynomials in \cite[Lemma~4.4]{DA} or as in the complex case in \cite[Lemma~2.3.2]{Haase} from the residue theorem. Altogether, we obtain~\eqref{CIDL1}.

If $f\in\EL[\omega]$ is left slice hyperholomorphic both at $0$ and at infinity, then $f(s) = \tilde{f}(s) + (1+s)^{-1}a + b$ where $\tilde{f}\in\lholZI(\sector{\varphi'})$ with $\omega<\varphi'<\pi$ is left slice hyperholomorphic both at $0$ and infinity and $a,b\in\hh$. We therefore have
\begin{align*}
 &f(\infty) \id + \frac{1}{2\pi}\int_{\partial(U(r,R)\cap\cc_I)} S_L^{-1}(s,T)\,ds_I\,f(s) \\
 =&  \frac{1}{2\pi}\int_{\partial(U(r,R)\cap\cc_I)} S_L^{-1}(s,T)\,ds_I\,\tilde{f}(s) \\
 +&f (\infty) \id  +  \frac{1}{2\pi}\int_{\partial(U(r,R)\cap\cc_I)} S_L^{-1}(s,T)\,ds_I\,\left((1+s)^{-1}a + b\right). 
 \end{align*}
As before, because of the left slice hyperholomorphicity of $\tilde{f}$ at $0$ and infinity, Cauchy's integral theorem allows us to vary the values of $r$ and $R$ for sufficiently small $r$ and sufficiently large $R$ without changing the value of the first integral. Letting $r$ tend to $0$ and $R$ tend to $\infty$, we find that this integral equals $\tilde{f}(T)$ in the sense of \Cref{SectSCalc}. Since $f(\infty) = b$, the remaining terms however equal $(1+T)^{-1}a + \id b$, which can again either be deduced by a standard application of the the residue theorem and Cauchy's integral theorem as in \cite[Corollary~2.3.5]{Haase}, or from the properties of the $S$-functional calculus for closed operators since the function $s\mapsto (1+s)^{-1}a+b$ is left slice hyperholomorphic on the spectrum of $T$ and at infinity. Altogether, we find that also \eqref{CIDL2} holds true. 

The right slice hyperholomorphic case finally follows by analogous arguments.

\end{proof}

\begin{corollary}\label{SCalcCompatible}
The $S$-functional calculus for closed operators and the $S$-functional calculus for sectorial operators are compatible.
\end{corollary}
\begin{proof}
Let $T\in\sectOP(\omega)$. If $f\in\EL[\sector{\omega}]$ is admissible for the $S$-functional calculus for closed operators, then it is left slice hyperholomorphic at infinity such that \eqref{CIDL2} holds true. The set $U(r,R)$ in this representation is however a slice Cauchy domain and therefore admissible as a domain of integration in the $S$-functional calculus for closed operators. Hence, both approaches yield the same operator.

\end{proof}

\Cref{DefECalc} is compatible with the algebraic structures of the underlying function classes.
\begin{lemma}\label{HOLOLULU}
If $T\in\sectOP(\omega)$, then the following statements hold true.
\begin{enumerate}[label = (\roman*)]
\item If $f,g\in\EL[\sector{\omega}]$  and $a \in \hh$, then $(fa+g)(T) = f(T)a+g(T)$. If $f,g\in\ER[\sector{\omega}]$ and $a\in\hh$, then $(af+g)(T) = af(T) + g(T)$. 
\item If $f\in\Eint[\sector{\omega}]$ and $g\in\EL[\sector{\omega}]$, then $(fg)(T) = f(T)g(T)$. If $f\in\ER[\sector{\omega}]$ and $g\in\Eint[\sector{\omega}]$, then also $(fg)(T) = f(T)g(T)$.
\end{enumerate}
\end{lemma}
\begin{proof}
The compatibility with the respective vector space structure is trivial. In order to show the product rule, consider $f\in\Eint[\sector{\omega}]$ and $g\in\EL[\sector{\omega}]$ with $f(s) = \tilde{f}(s) + (1+s)^{-1}a + b$ with $\tilde{f}\in\intrinZI[\sector{\omega}]$ and  $a,b\in\rr$ and $g(s) = \tilde{g}(s) + (1+s)^{-1}c + d$ with $\tilde{g}\in\lholZI[\sector{\omega}]$ and $c,d\in\hh$. By \Cref{HZIProp}, \Cref{ResCanc} and the identity $(\id+T)^{-2} = (\id + T)^{-1} - T(\id + T)^{-2}$, we then have
\begin{align*}
f(T)g(T) =& \tilde{f}(T)\tilde{g}(T) + \tilde{f}(T)(\id+T)^{-1}c + \tilde{f}(T) d + (\id+T)^{-1}\tilde{g}(T) a\\
&+  (\id+T)^{-2}ac + (\id+T)^{-1}ad + \tilde{g}(T)b + (\id+T)^{-1}bc + bd\id\\
=&\left( \tilde{f}\tilde{g} + \tilde{f}(1+s)^{-1}c + \tilde{f} d + (1+s)^{-1}\tilde{g} a + \tilde{g}b\right)(T)\\
&-  T(\id+T)^{-2}ac + (\id+T)^{-1}(ad + ac +bc) + bd\id.
\end{align*}
Since $-s(1+s)^{-2}\in\EL[\sector{\omega}]$ is left slice hyperholomorphic at zero and infinity, the compatibility of the $S$-functional calculus for sectorial and the $S$-functional calculus for for closed operators and the properties of the latter imply $(-s(1+s)^2)(T) = -T(\id+T)^{-2}$ such that
\begin{align*}
f(T)g(T) =&\left[ \tilde{f}\tilde{g}+ \tilde{f}(1+s)^{-1}c + \tilde{f} d +  (1+s)^{-1}\tilde{g} a+  \tilde{g}b - s(1+s)^{-2}ac\right](T) \\
&+ (\id+T)^{-1}(ad + ac +bc) + bd\id = (fg)(T)
\end{align*}
since
\begin{align*}
 (fg)(s) =&\tilde{f}(s)\tilde{g}(s) + \tilde{f}(s)(1+s)^{-1}c + \tilde{f}(s) d +  (1+s)^{-1}\tilde{g}(s) a\\
&\ +  \tilde{g}(s)b - s(1+s)^{-2}ac + (1+s)^{-1}(ad + ac +bc) + bd.
\end{align*}
The product rule in the right slice hyperholomorphic case can be shown with analogous arguments.

\end{proof}

\begin{lemma}\label{ghj}
If $T\in\sectOP(\omega)$, then the following statements hold true.
\begin{enumerate}[label = (\roman*)]
\item \label{ghj1}  We have $(s(1+s)^{-1})(T) = T(\id+T)^{-1}$.
\item \label{ghj2} If $A$ is closed and commutes with $Q_{s}(T)^{-1}$ and $TQ_{s}(T)^{-1}$ for all $s\in\rho_S(T)$, then $A$ commutes with $f(T)$ for any $f\in\Eint[\sector{\omega}]$. In particular $T$ commutes with $f(T)$ for any $f\in\Eint[\sector{\omega}]$.
\item  \label{ghj3} If $v\in\ker(T)$ and $f\in\ER[\sector{\omega}]$, then $f(A)v = f(0)v$. In particular this holds true if $f\in\Eint[\sector{\omega}]$.
\end{enumerate}
\end{lemma}
\begin{proof}
The first statement holds as 
\[
\left(s(1+s)^{-1}\right)(T)  = (1 -(1+s)^{-1} )(T)= \id - (\id+T)^{-1} = T(\id+T)^{-1}
\]
 and the second one follows from \Cref{ComLem}. Finally, if $v\in\ker(T)$, then 
 \[
 Q_{s}(T)v = \left(T^2 - 2s_0T + |s|^2\right)v = |s|^2v
 \]
  and hence 
 \[
 S_R^{-1}(s,T)v =(\overline{s}\id - T) \Q_{s}(T)^{-1} v = \overline{s}|s|^{-2}v =  s^{-1}|s|^{-2}v = S_R^{-1}(s,T)v.
 \]
  For $\tilde{f}\in\rholZI[\sector{\varphi}]$, we hence have
\[\tilde{f}(T) v = \frac{1}{2\pi}\int_{\partial(\sector{\varphi}\cap\cc_I)}\tilde{f}(s)\,ds_I\,S_R^{-1}(s,T)v = \frac{1}{2\pi}\int_{\partial(\sector{\varphi}\cap\cc_I)}\tilde{f}(s)\,ds_I\,s^{-1}v = 0\]
by Cauchy's integral theorem such that for $f(s) = \tilde{f}(s) + a(1+s)^{-1} + b$ and $v\in\ker(T)$
\[ f(T)v = \tilde{f}(T)v + a(\id+T)^{-1}v + b\id v = av+bv = f(0) v.\]

\end{proof}
\begin{remark}\label{YXSW}
If $f\in\EL(\sector{\omega})$, then we cannot expect \ref{ghj3} in \Cref{ghj} to hold true. In this case 
\[\tilde{f}(T)v = \frac{1}{2\pi} \int_{\partial(\sector{\varphi}\cap\cc_I)}S_L^{-1}(s,T)\,ds_I \, f(s) v,\]
 but $v$ and $ds_I\,f(s)$ do not commute. So we cannot exploit the fact that $v\in\ker(T)$ to simplify $S_L^{-1}(s,T)v = s^{-1}v$. Indeed, also this  identity does not necessarily hold true as $S_L^{-1}(s,T) = \Q_{s}(T)^{-1}(\overline{s}-T)v = \Q_{s}(T)^{-1}\overline{s}v$  for $v\in\ker(T)$. But the kernel of $T$ is in general not a left linear subspace of $T$ and hence we cannot assume $\overline{s}v \in\ker(T)$. The simplification $\Q_{s}(T)^{-1}(s,T)\overline{s}v = |s|^2\overline{s} v = s^{-1}v$ is not possible. 
 
 \end{remark}

The $H^{\infty}$-functional calculus for complex linear sectorial operators in \cite{Haase} applies to meromorphic functions that are regularisable. Defining the orders of zeros and hence also of poles of slice-hyperholomorphic functions properly is a very complicated. We shall hence use the following simple definition, which is sufficient for our purposes.

\begin{definition}
Let $s\in\hh$ and let $f$ be left slice hyperholomorphic on an axially symmetric neighbourhood $[B_r(s)]\setminus\{s\}$ of $s$ with $[B_{r}(s)] = \{p\in\hh: \dist([s],p) < r\}$ and assume that $f$ does not have a left slice hyperholomorphic continuation to all of $[B_r(s)]$. We say that $f$ has a pole at the sphere $[s]$ if there exists $n\in\nn$ such that $p\mapsto \Q_{s}(p)^nf(p)$ has a left slice hyperholomorphic continuation to $[B_r(s)]$ if $s\notin\rr$ resp. if there exists $n\in\nn$ such that $p\mapsto (p-s)^{-1}f(p)$ has a left slice hyperholomorphic continuation to $[B_r(s)]$ if $s\in\rr$. 
\end{definition}
\begin{remark}
If $[s]$ is a pole of $f$ and $p_n$ is a sequence with $\lim_{n\to+\infty}\dist(p_n,[s]) = 0$, then not necessarily $\lim_{n\to+\infty} |f(p_n)| = +\infty$. One can see this easily if one restricts $f$ to one of the complex planes $\cc_{I}$. If $I,J\in\SS$ with $J\perp I$, then the function $f_{I} := f|_{[B_{r}(s)]\cap\cc_I}$ a meromorphic function with values in the complex (left) vector space $\hh \cong \cc_I + \cc_IJ$ over $\cc_I$. It must have a pole at $s_I = s_0 + I s_1$ or $\overline{s_I} =s_0 - I s_1$. Otherwise we could extend $f_{I}$ to a holomorphic function on $B_{r}(s)\cap\cc_{I}$. The representation formula would allow us then to define a slice hyperholomorphic extension of $f$ to $B_{r}(s)$. However, $s_{I}$ and $\overline{s_I}$ are not necessarily  both poles of $f_{I}$.  Consider for instance the function $f(p) = S_{L}^{-1}(s,p) = (p^2 - 2 s_0 p + |s|^2)^{-1}(\overline{s}-p)$, which is defined on $U = \hh\setminus[s]$. If we choose $I= I_s$, then $f|_{U\cap\cc_I} = (s-p)^{-1}$, which does obviously not have a pole at $\overline{s}$.  Hence, if $p_n\in\cc_{I}$ tends to $\overline{s}$, then $|f(p_n)|$ remains bounded. 

However, the representation formula implies that there exists at most one complex plane $\cc_I$ such that only one of the points $\overline{s_I}$ and $s_I$ is a pole of $f_I$. Otherwise we could use it again to find a slice hyperholomorphic extension of $f$ to $B_r(0)$. For intrinsic functions always both points $s_{I}$ and $\overline{s_{I}}$ need to be poles of $f_I$ as in this case $f_I(\overline{p}) = \overline{f_I(p)}$. 

In general we therefore do not have $\lim_{\dist(p,[s])\to 0} |f(p)| = +\infty$, but at least for the limit superior the equality 
\[
\limsup_{\dist(p,[s])\to 0} |f(p)| = +\infty
\]
holds. If $f$ is intrinsic, then even $\lim_{\dist(p,[s])\to 0} |f(p)| = +\infty$ holds true.
\end{remark}
\begin{definition}
Let $U\subset\hh$ be axially symmetric. A function $f$ is said to be left meromorphic on $U$ if there exist isolated spheres $[p_n]\subset U$ for $n\in\Theta$, where $\Theta$ is a subset of $\nn$, such that $f|_{\tilde{U}}\in \lhol(\tilde{U})$ with $\tilde{U} = U \setminus \bigcup_{n\in\Theta} [p_n]$ and such that each sphere $[p_n]$ is a pole of $f$. We denote the set of all such functions by $\meroL(U)$ and the set of all such functions that are intrinsic by $\meroInt(U)$. 

For $U= \sector{\omega}$ with $0<\omega <\pi$, we furthermore denote 
\[
\meroL[\sector{\omega}]_T = \cup_{\omega < \varphi < \pi} \meroL(\sector{\varphi})\qquad\text{ and }\qquad\meroInt[\sector{\omega}]_T = \cup_{\omega < \varphi < \pi} \meroInt(\sector{\varphi}).
\]
\end{definition}

\begin{definition}
Let $T\in\sectOP(\omega)$. A left slice hyperholomorphic function $f$ is said to be regularisable if $f\in\meroL(\sector{\varphi})$ for some $\omega < \varphi <\pi$  and there exists $e\in \Eint(\sector{\varphi})$ such that $e(T)$ defined in the sense of \Cref{DefECalc} is injective and $ef \in \EL(\sector{\varphi})$. In this case we call $e$ a regulariser for $f$.

We denote the set of all such functions by $\meroL[\sector{\omega}]_T$. Furthermore, we denote the subset of intrinsic functions in $\meroL[\sector{\omega}]_T$ by $\meroInt[\sector{\omega}]_T$. 
\end{definition}
\begin{lemma}\label{MStruct}
Let $T\in\sectOP(\omega)$. 
\begin{enumerate}[label = (\roman*)]
\item If $f,g\in\meroL[\sector{\omega}]_T$ and $a\in\hh$, then $fa+g\in\meroL[\sector{\omega}]_T$. If furthermore $f\in\meroInt[\sector{\omega}]_T$, then also $fg\in\meroL[\sector{\omega}]_T$.
\item The space $\meroInt[\sector{\omega}]_T$ is a real algebra.
\end{enumerate}
\end{lemma}
\begin{proof}
If $e_1$ is a regulariser for $f$ and $e_2$ is a regulariser for $g$, then $e = e_1e_2$ is a regulariser for $fa+g$ and also for $fg$ if $f$ is intrinsic. Hence the statement follows.

\end{proof}
\begin{remark}
If $T$ is injective, then $f$ does not need to have finite polynomial limit at $0$ in $\sector{\omega}$. Indeed, the function $p\mapsto p(1+p)^{-2}$ or the function $p\mapsto p\left(1+p^2\right)^{-1}$ and their powers can then serve as regularisers that may compensate a singularity at $0$. Choosing the latter as a specific regulariser yields exactly the approach chosen in \cite{Hinfty}, where the $H^{\infty}$-function calculus was first introduced for quaternionic linear operators.
\end{remark}
\begin{definition}[$H^{\infty}$-functional calculus]\label{HIDef}
Let $T\in\sectOP(\omega)$. For regularisable $f\in\meroL[\sector{\omega}]_T$, we define
\[ f(T) := e(T)^{-1}(ef)(T),\]
where $e(T)^{-1}$ is the closed inverse of $e(T)$ and $(ef)(T)$ is intended in the sense of \Cref{DefECalc}.
\end{definition}
\begin{remark}
The operator $f(T)$ is independent of the regulariser $e$ and hence well-defined. Indeed, if $\tilde{e}$ is a different regulariser, then $e$ and $\tilde{e}$ commute because they both belong to $\Eint[\sector{\omega}]$. Hence, $\tilde{e}(T) e(T) = (\tilde{e}e)(T) = (e\tilde{e})(T) = e(T)\tilde{e}(T)$ by \Cref{HOLOLULU}. Inverting this equality yields $e(T)^{-1}\tilde{e}(T)^{-1} = \tilde{e}(T)^{-1}e(T)$ such that
\begin{align*}
e(T)^{-1}(ef)(T) &= e(T)^{-1} \tilde{e}(T)^{-1} \tilde{e}(T)(ef)(T) = e(T)^{-1} \tilde{e}(T)^{-1} (\tilde{e}ef)(T) \\
&= \tilde{e}(T)^{-1} e(T)^{-1}  (e\tilde{e}f)(T) = \tilde{e}(T)^{-1} e(T)^{-1}  e(T)(\tilde{e}f)(T) = \tilde{e}(T)^{-1}(\tilde{e}f)(T).
\end{align*}
If $f\in\EL[\sector{\omega}]$, then we can use the constant function $1$ with $1(T) = \id$ as a regulariser in order to see that \Cref{HIDef} is consistent with \Cref{DefECalc}.
\end{remark}
\begin{remark}
Since we are considering right linear operators, \Cref{HIDef}  is not possible for right slice hyperholomorphic functions. Right slice hyperholomorphic functions maintain slice hyperholomorphicity under multiplication with intrinsic functions from the right. A regulariser of a functions $f$ would hence be a function $e$ such that $e(T)$ is injective and $fe\in\ER(\sector{\varphi})$. The operator $f(T)$ would then be defined as $(fe)(T)e(T)^{-1}$, but this operator is only defined on $\ran e(T)$ and can hence not be independent of the choice of $e$. If we consider left linear operators, the situation is of course vice versa, which is a common phenomenon in quaternionic operator theory, cf. for instance also \Cref{YXSW} or \cite{FUCGEN}.
\end{remark}
The next lemma shows that the function $f$ needs to have a proper limit behaviour at $0$ if $T$ is not injective.
\begin{lemma}\label{NonInjLem}
Let $T\in\sectOP(\omega)$ and $f\in\meroL[\sector{\omega}]_{T}$. If $T$ is not injective, then $f$ has finite polynomial limit $f(0)\in\hh$ in $\sector{\omega}$ at $0$. If furthermore $f$ is intrinsic, then $f(T)v = f(0)v$ for any $v\in\ker(T)$.
\end{lemma}
\begin{proof}
Assume that $T$ is not injective and let $e$ be a regulariser for $f$. Since $e(T)v = e(0) v$ for all $v\in\ker(T)$ because of \ref{ghj3} in  \Cref{ghj}, we have $e(0) \neq 0$ as $e(T)$ is injective. The limit $e(0)f(0):=\lim_{p\to 0}e(p)f(p)$ of $e(p)f(p)$ as $p$ tends to $0$ in $\sector{\omega}$ exists and is finite because $ef \in \EL(\sector{\omega}) $. Hence the respective limit  of $f(p) = e(p)^{-1} (e(p)f(p))$ exists too and is finite. Indeed, it is $ f(0) = \lim_{p\to 0 } f(p) =  e(0)^{-1}(e(0)f(0))$. We find that 
\[ f(p) - f(0) = e(p)^{-1}\left[  \left(e(p)f(p) - e(0)f(0)\right) - \left(e(p)-e(0)\right)f(0)\right] = O(|p|^{\alpha})\]
as $p$ tends to $0$ in $\sector{\omega}$ because both $ef$ and $e$ have polynomial limit at $0$. Hence, $f$ has polynomial limit $f(0)$ at $0$ in $\sector{\omega}$.

If $f$ is intrinisic, then $ef$ is intrinsic too and $e(0)$, $(ef)(0)$ and $ f(0)$ are all real. Hence, for any $v \in \ker(T)$, we have $(ef)(0)v = v(ef)(0)\in\ker(T)$. As $\ker(T)$ is a right linear subspace of $V$, we conclude that also $(ef)(0)v\in\ker(T)$ and so \ref{ghj3} in \Cref{ghj} yields 
\[f(T)v = e(T)^{-1}(ef)(T)v = e(T)^{-1} (ef)(0) v = e(0)^{-1}(ef)(0) v = f(0)v.\]

\end{proof}

The proof of the following lemma is analogue to the one of the corresponding complex results, Proposition~1.2.2 and Corollary~1.2.4 in \cite{Haase}, and does not employ any specific quaternionic techniques. For the convenience of the reader, we nevertheless give the detailed proof as this result turns out to be crucial for what follows.
\begin{lemma}
Let $T\in\sectOP(\omega)$.
\begin{enumerate}[label = (\roman*)]\label{MoProp}
\item \label{MoProp1}If $A\in\boundOP(V)$ commutes with $T$, then $A$ commutes with $f(T)$ for any $f\in\meroInt[\sector{\omega}]_T$. Moreover, if  $f\in\meroInt[\sector{\omega}]_T$ and $f(T)\in\boundOP(V)$, then $f(T)$ commutes with $T$.
\item\label{MoProp2} If $f,g\in\meroL[\sector{\omega}]_T$, then 
\[f(T) + g(T) \subset (f+g)(T).\]
 If furthermore $f\in\meroInt[\sector{\omega}]_T$, then 
 \[ f(T)g(T) \subset (fg)(T)\]
 with $\dom(f(T)g(T)) = \dom((fg)(T)) \cap \dom(g(T))$. In particular, the above inclusion turns into an equality if $g(T)\in\boundOP(V)$.
\item \label{MoProp3} Let $f\in\meroInt[\sector{\omega}]_T$ and $g\in\meroInt[\sector{\omega}]$ be such that $fg \equiv 1$. Then $g \in \meroInt[\sector{\omega}]_T$ if and only if $f(T)$ is injective. In this case $f(T) = g(T)^{-1}$.
\end{enumerate}
\end{lemma}
\begin{proof}
If $A\in\boundOP(V)$ commutes with $T$, then it commutes with $\Q_{s}(T)^{-1}$ and $T\Q_{s}(T)^{-1}$ for any $s\in\rho_S(T)$. Hence it also commutes with $e(T)$ for any $e\in\Eint[\sector{\omega}]$ by \Cref{ghj}. If $f\in\meroInt[\sector{\varphi}]_T$ and $e$ is a regulariser for $f$, we thus have
\[Af(T) = Ae(T)^{-1}(ef)(T) \subset e(T)^{-1} A(ef)(T) = e(T)^{-1}(ef)(T)A = f(T)A\]
such that the first assertion in \ref{MoProp1} holds true. Because of \ref{ghj1} in \Cref{ghj}, the function $(1+p)^{-1}$ regularizes the identity function $p\mapsto p$ and we have $p(T) = T$. Once we have shown \ref{MoProp2}, we can hence obtain the second assertion in \ref{MoProp1} from
\[ f(T)T \subset (f(p)p)(T) = (pf(p))(T) = T f(T).\]

In order to show \ref{MoProp2} assume that $f,g\in\meroL[\sector{\omega}]_T$ and let $e_1$ be a regulariser for $f$ and $e_2$ be a regulariser for $g$. Then $e = e_{1}e_{2}$ regularises both $f$ and $g$ and hence also  $f+g$ such that
\begin{gather*}
 f(T) + g(T) = e(T)^{-1}(ef)(T) + e(T)^{-1} (eg)(T)
  \subset e(T)^{-1}\left[(ef)(T) + (eg)(T)\right]\\
   = e(T)^{-1}(e(f+g))(T) = (f+g)(T).
 \end{gather*}
Applying this relation to the functions $f+g$ and $-g$, we find that $(f+g)(T) - g(T) \subset f(T)$ and so $(f+g)(T) = f(T) + g(T)$ if $g(T)$ is bounded.

 If even $f\in\Eint[\sector{\omega}]_{T}$, then $f$ and $e_2$ are both intrinsic and hence commute. Thus $e(fg) = (e_1f)(e_2g)\in \EL[\sector{\omega}|_T$ by \Cref{EStruct} and so  $e$ regularises $fg$. Because of \ref{ghj2} in \Cref{ghj} the operator $(e_1f)(T)$ commutes with $e_2(T)$ and hence also with $e_2(T)^{-1}$. Because of \eqref{HOLOLULU}, we thus find that 
 \begin{gather*}
 f(T)g(T) = e_{1}(T)^{-1}(e_1f)(T) e_{2}(T)^{-1}(e_2g)(T) \subset e_1(T)^{-1}e_2(T)^{-1}(e_1f)(T)(e_2g)(T) \\
 =\left[ e_2(T)e_1(T)\right]^{-1}(e_1fe_2g)(T) = e(T)^{-1}(efg)(T) = (fg)(T).
 \end{gather*}
 
In order to prove the statement about the domains, we consider $v\in\dom((fg)(T))\cap\dom(g(T))$. Then $w := (e_2g)(T)v\in\dom\left(e_2(T)^{-1}\right)$. Since $(e_1f)(T)$ commutes with $e_{2}(T)^{-1}$, we conclude that  also $(e_1f)(T)w \in \dom\left(e_2(T)^{-1}\right)$.  Since $v\in\dom((fg)(T))$ and $(fg)(T)v = e(T)^{-1}(efg)(T)v$, we further have $(efg)(T)v \in \dom( e(T)^{-1})$. As $e(T)^{-1} = e_1(T)^{-1}e_2(T)^{-1}$ this implies $e_2(T)^{-1}(efg)(T) v \in \dom(e_1(T)^{-1})$. From the identity
\begin{gather*}
(e_1f)(T)g(T)v = (e_1f)(T) e_2(T)^{-1} w = e_2(T)^{-1} (e_1f)(T)w = e_2(T)^{-1}(efg)(T)v
\end{gather*}
we conclude that $(e_1f)(T)g(T)v\in\dom\left(e_1(T)^{-1}\right)$. Thus, $g(T)v\in\dom(f(T))$ and in turn $v\in\dom(f(T)g(T))$. Therefore 
\[\dom(f(T)g(T))\supset \dom((fg)(T))\cap\dom(g(T)).\]
 The other inclusion is trivial such that altogether we find equality. If $g(T)$ is bounded, then $\dom(g(T))=V$ and we find $\dom(f(T)g(T)) = \dom((fg)(T))$ such that both operators agree.

We show now the statement \ref{MoProp3} and assume that $f,g\in\meroInt[\sector{\omega}]$ with $fg=1$ and that $f$ is regularisable. If $g$ is regularisable too, then  \ref{MoProp3} implies $g(T)f(T) \subset (gf)(T) = 1(T) = \id$ with $\dom(g(T)f(T)) = \dom(\id)\cap\dom(f(T)) = \dom(f(T))$. Hence $f(T)$ is injective and interchanging the role of $f$ and $g$ shows that $f(T)g(T) = \id$ on $\dom(g(T))$ such that actually $f(T) = g(T)^{-1}$. Conversely, if $f(T)$ is injective and $e$ is a regulariser for $f$, then $(fe)g = e(fg) = e\in\Eint[\sector{\omega}]_T$. Moreover $(fe)(T)$ is injective as $f(T)$ and $e(T)$ are both injective and $(fe)(T) = f(T)e(T)$ by \ref{MoProp2}. Thus $fe$ is a regulariser for $g$, i.e. $g\in\meroInt[\sector{\omega}]_T$.

\end{proof}
We define polynomials of an operator $T$ are as usually as $P[T] = \sum_{k=0}^n T^ka_k$ with $\dom\left(P[T]\right) = \dom\left(T^n\right)$ for any polynomial $P(p) = \sum_{k=0}^np^ka_k$. We use the squared brackets to indicate that the operator $P[T]$ is defined via this functional calculus and not via the $H^{\infty}$-functional calculus. However, as the next lemma shows, both approaches are consistent.
\begin{lemma} \label{MoProp4} 
The $H^{\infty}$-functional calculus is compatible with intrinsic rational functions. More precisely, if $r(p) = P(p)Q(p)^{-1}$ is an intrinsic rational function with intrinsic polynomials $P$ and $Q$ such that the zeros of $Q$  lie in $\rho_S(T)$, then $r\in\meroInt[\sector{\omega}]_T$ and the operator $r(T)$ is given by $r(T) = P[T]Q[T]^{-1}$. 
 \end{lemma}
\begin{proof}
We first prove compatibility with intrinsic polynomials. For intrinsic polynomials of degree $1$ this follows from the linearity of the $H^{\infty}$-functional calculus and from \ref{ghj1} in \Cref{ghj}, which shows that $(1+p)^{-1}$ regularises the identity function $p\mapsto p$ and that 
\[
p(T) = \left((1+p)^{-1}(T)\right)^{-1} (p(1+p)^{-1})(T) =  (\id + T)T(\id+T)^{-1} = T.
\] 

Let us now generalise the statement by induction and let us assume that it holds for intrinsic polynomials of degree $n$. If $P$ is a polynomial of degree $n+1$, let us write $P(p) = Q(p)p + a$ with $a\in\rr$ and an intrinsic polynomial $Q$ of degree $n$. The induction hypothesis implies that $Q\in\meroInt[\sector{\omega}]_T$,  that $Q(T) = Q[T]$, and that $\dom(Q(T)) = \dom(T^n)$. Since $\meroInt[\sector{\omega}]_T$ is a real algebra, we find that $P$  also belongs to $\meroInt[\sector{\omega}]_T$ and we deduce from \ref{MoProp3} in \Cref{MoProp} that 
\[ P(T) \supset Q(T)T + a\id = Q[T]T + a\id = P[T]\]
with $\dom(P[T]) = \dom\left(T^{n+1}\right) = \dom(Q(T)T) = \dom(P(T))\cap\dom(T)$. Hence, if we show that $\dom(T)\subset\dom(P(T))$, the induction is complete. In order to do this, we consider $v\in\dom(P(T))$. Then  $(\id + T)^{-1}v$ does also belong to $\dom(P(T))$ because $(\id + T)^{-1}P(T) \subset P(T)(\id + T)^{-1}$ by \ref{MoProp1} in \Cref{MoProp}. But obviously also $(\id + T)^{-1}v \in\dom(T)$ and hence $(\id + T)^{-1}v \in \dom(P(T))\cap\dom(T) = \dom\left(T^{n+1}\right)$, which implies $v\in\dom(T^n) \subset \dom(T)$. We conclude $\dom(T)\subset \dom(P(T))$. 
 
Let us now turn to arbitrary intrinsic rational functions. If $s\in\rho_S(T)$ is not real, then $\Q_{s}(T)$ is injective because $\Q_{s}(T)^{-1}\in\boundOP(V)$ and hence $\Q_{s}(p)^{-1}\in\meroInt[\sector{\omega}]_T$ by \ref{MoProp3} in \Cref{MoProp}. Similarly, if $s\in\rho_S(T)$ is real, then $p\mapsto (s-p)^{-1}\in\meroInt[\sector{\omega}]_T$ because $(s-p)(T) = (s\id-T)$ is injective as $(s\id-T)^{-1} = S_L^{-1}(s,T) \in\boundOP(V)$. If now $r(p) = P(p)Q(p)^{-1}$ is an intrinsic rational function with poles in $\rho_S(T)$, then we can write $Q(p)$ as product of such factors, namely 
\[Q(p) = \prod_{\ell=1}^N (\lambda_{\ell}-p)^{n_{\ell}}\prod_{\kappa=1}^M\Q_{s_\kappa}(p)^{m_{\kappa}},\]
 where $\lambda_{1},\ldots, \lambda_{N}\in\rho_S(T)$ are the real zeros of $Q$ and $[s_1],\ldots, [s_{M}]\subset\rho_S(T)$ are the spherical zeros of $Q$ and $n_{\ell}$ and $m_{\kappa}$ are the orders of $\lambda_{\ell}$ resp $[s_{\kappa}]$. Since $\meroInt[\sector{\omega}]_T$ is a real algebra, we conclude that $Q\in\meroInt[\sector{\omega}]_T$ and because of \ref{MoProp3} we find $Q^{-1}(T) = Q(T)^{-1} = Q[T]^{-1}$. Moreover, \ref{MoProp2} in \Cref{MoProp} implies 
\[Q^{-1}(T) = \prod_{\ell=1}^N (\lambda_{\ell}\id-T)^{-n_{\ell}}\prod_{\kappa=1}^M\Q_{s_\kappa}(T)^{-m_{\kappa}}\in\boundOP(V)\]
because each of the factors in this product is bounded. Finally, we deduce from the boundedness of $Q^{-1}(T)$ and \ref{MoProp2} in \Cref{MoProp} that 
\[ r(T) = \left(PQ^{-1}\right)(T)= P(T) Q^{-1}(T) = P[T]Q[T]^{-1} = r[T].\]
\end{proof}

\subsection{The composition rule}\label{SubSecCompRule}
Let us now turn our attention to the composition rule, which will occur at several occasions when we consider fractional powers of sectorial operators. As always in the quaternionic setting, we can only expect such a rule to hold true if the inner function is intrinsic since the composition of two slice hyperholomorphic functions is only slice hyperholomorphic if the inner function is intrinsic.

\begin{theorem}\label{CompRul}
Let $T\in\sectOP(\omega)$ and $g\in\meroInt[\sector{\omega}]_T$ besuch that $g(T)\in\sectOP(\omega')$. Furthermore assume that for any $\varphi'\in(\omega',\pi)$ there exists $\varphi \in(\omega, \pi)$ such that $g\in\meroInt(\sector{\varphi})$ and $g(\sector{\varphi})\subset \overline{ \sector{\varphi'}}$. Then $f\circ g \in\meroInt[\sector{\omega}]_T$   for any $f\in\meroL[\sector{\omega'}]_{g(T)}$ and 
\[
(f\circ g)(T) = f(g(T)).
\]
\end{theorem}
\begin{proof}
Let us first assume that $g\equiv c$ is constant. In this case $g(T) = c \id$. Since $g$ is intrinsic, we have $\overline{c} = \overline{g(s)} = g(\overline{s}) = c$ and so $c\in\rr$. Since $g$ maps $\sector{\varphi}$ into $\overline{\sector{\varphi'}}$ for suitable $\varphi \in (\omega,\pi)$ and $\varphi'\in(\omega',\pi)$, we further find $c\in \overline{\sector{\varphi'}}\cap\rr = [0,+\infty)$. If $c\neq 0$, then $(f \circ g)(p) \equiv f(c) $ and we deduce easily, for instance from \Cref{SCalcCompatible}, that $(f\circ g) (T) = f(c)\id = f(g(T))$. If on the other hand $c = 0$, then \Cref{NonInjLem} implies that  $f(0) := \lim_{p\to 0}f(p)$ as $p$ tends to $0$ in $\sector{\omega}$ exists. Hence $f \circ g$ is well defined. It is the constant function $f\circ g \equiv f(0)$ and so $(f\circ g)(T) = f(0)\id$. If $f$ is intrinsic, then \Cref{NonInjLem} implies $f(g(T)) = f(0)\id = (f\circ g) (T)$. If $f$ is not intrinsic, then $f = f_0 + \sum_{\ell =1}^3 f_{\ell}e_{\ell}$ with intrinsic components $f_{\ell}$. Since $\ker g(T) = \ker (0\id) =V$, for any vector $v$ also the vectors $e_{\ell}v, \ell = 1,2,3$ belong to $\ker g(T)$ and we conclude, again from  \Cref{NonInjLem}, that
\begin{gather*}
f(g(T))v = f_0(g(T)) v + \sum_{\ell=1}^{3}f_{\ell}(g(T)) e_{\ell} v = f_{0}(0)v + \sum_{\ell=1}^3 f_{\ell}(0)e_{\ell}v\\
 =  \left(f_{0}(0) + \sum_{\ell=1}^3 f_{\ell}(0)e_{\ell}\right)v = f(0) v = (f\circ g)(T) v.
\end{gather*}
In the following we shall thus assume that $g$ is not constant.

Let $\varphi'$ and $\varphi$ be a couple of angles as in the assumptions of the theorem. Since $g$ is intrinsic, $g|_{\cc_I\cap\sector{\varphi}}$ is a non-constant holomorphic function on $\cc_I\cap\sector{\varphi}$. Hence, it maps the open set $g(\sector{\varphi}\cap \cc_I)$ to an open set. The set $g(\sector{\varphi}) = [g(\sector{\varphi}\cap\cc_I)]$ is therefore also open and so actually contained in $\sector{\varphi'}$, not only in $\overline{\sector{\varphi'}}$. In particular, we find that $f\circ g$ is defined and slice hyperholomorphic on $\sector{\varphi}$.

We assume for the moment that $f\in\EL(\sector{\varphi'})$ with $\varphi'\in(\omega',\pi)$ and choose $\varphi\in(\omega,\pi)$ such that  $g\in\meroInt(\sector{\varphi})$ and 
 $g(\sector{\varphi})\subset \sector{\varphi'}$. Since $f$ is bounded on $\sector{\varphi'}$, the function $f\circ g$ is a bounded function in $\lhol(\sector{\varphi})$. If $T$ is injective, then $e(p) = p(1+p)^{-2}\in\Eint(\sector{\varphi})$  such that $e(T)T(\id+T)^{-2}$ is injective. Moreover, the function $p\mapsto e(p)(f\circ g )(p)$ decays regularly at $0$ and infinity in $\sector{\varphi}$ and hence belongs to $\EL(\sector{\varphi})$. In other words, $e$ is a regulariser for $f\circ g$ and hence $f\circ g \in \meroL[\sector{\omega}]_T$. If $T$ is not injective, then $g$ has polynomial limit $g(0)$ at $0$ by \Cref{NonInjLem}. Since $g$ is intrinsic, it only takes real values on the real line and so $g(0)\in\rr$. It furthermore maps $\sector{\varphi}$ to $\sector{\varphi'}$ and so $g(0)\in\overline{\sector{\varphi'}}\cap\rr = [0,+\infty)$. Therefore $f$ has polynomial limit at $g(0)$: if $g(0) = 0$ this follows from \Cref{EChar}, otherwise it follows from the Taylor expansion of $f$ at $g(0)\in(0,\infty)$, cf. \Cref{Taylor}. As a consequence, $f\circ g $ has polynomial limit at $0$. Therefore the function $p\mapsto (1+p)^{-1}(f\circ g)(p)$ belongs to $\EL(\sector{\varphi})$. Since $(\id + T)^{-1}$ is injective because $-1\in\rho_S(T)$, we find that $(1+p)^{-1}$ is a regularizer for $f\circ g$ and hence $f\circ g \in\meroL[\sector{\omega}]_T$. 

We have $f(p) = \tilde{f}(p) + (1+p)^{-1}a + b$ with $\tilde{f}\in\lholZI(\sector{\varphi'})$ and $a,b\in\hh$. Because of the additivity of the functional calculus, we can treat each of these pieces separately. The case that $f \equiv b$ has already been considered above. For $f(p) = (1+p)^{-1}a$, the identity $(f\circ g)(T) = (\id + g(T))^{-1}$ follows from \ref{MoProp3} in \Cref{MoProp} because $p\mapsto 1 + g(p)$ and $p\mapsto (f\circ g) (p) = (1 +g(p) )^{-1}$ do both belong to $\meroL[\sector{\omega}]_T$. Hence let us assume that $f = \tilde{f} \in \lholZI(\sector{\varphi'})$ with $\varphi' \in ( \omega',\pi)$.

We choose $\theta' \in (\omega', \varphi')$ and $I\in\SS$ and set $\Gamma_p = \partial (\sector{\theta'}\cap \cc_I)$. We furthermore choose $\rho' \in (\omega', \theta')$ and by our assumptions on $g$, we can find $\varphi\in(\omega,\pi)$ such that $g(\sector{\varphi})\subset \sector{\rho'}\subsetneq \sector{\theta'}$. We choose $\theta \in (\omega, \varphi)$ and set $\Gamma_s = \partial(\sector{\theta}\cap\cc_I)$. The subscripts $s$ and $p$ in $\Gamma_s$ and $\Gamma_p$ refer to the corresponding variable of integration in the following computations.

For any $p \in \Gamma_p$, the functions $s \mapsto \Q_{p}(g(s))^{-1} = (g(s)^2 - 2p_0 g(s) + |p|^2)^{-1}$ and $s\mapsto S_L^{-1}(p, g(s))$ do then belong to $\EL(\sector{\varphi})$ and $\left[\Q_{p}(g(\cdot))^{-1}\right](T) = \Q_{p}(g(T))^{-1}$ and $\left[S_L^{-1}(p,g(\cdot))\right](T) = S_L^{-1}(p,g(T))$. Indeed, by \cref{MoProp2} in \Cref{MoProp}, we have
\begin{align}\label{AMURZI}
\left[\Q_{p}(g(\cdot))\right](T) = (g^2 - 2 p_0 g + |p|^2)(T) \supset g(T)^2 - 2p_0g(T) + |p|^2\id =  \Q_{p}(g(T)).
\end{align}
Taking the closed inverses of these operators, we deduce from \cref{MoProp3} in \Cref{MoProp} that
\begin{equation}\label{AMSti}
\left[\Q_{p}(g(\cdot))^{-1}\right](T) = \left[\Q_{p}(g(\cdot))\right](T)^{-1} \supset  \Q_{p}(g(T))^{-1}.
\end{equation}
Since $p\in\rho_S(T)$, the $\Q_{p}(g(T))^{-1}$ is a bounded operator and hence already defined on all of $V$. Hence, the inclusion $\supset$ in \eqref{AMSti} and \eqref{AMURZI} is actually an equality and we find $\left[\Q_{p}(g(\cdot))^{-1}\right](T) =  \Q_{p}(g(T))^{-1}$. From \cref{MoProp2} we further conclude that also
\begin{align*}
\left[S_L^{-1}(p,g(\cdot))\right](T) =& \left[ \Q_{p}(g(\cdot))^{-1}\overline{p} - g(\cdot) \Q_{p}(g(\cdot))^{-1}\right] (T)\\
=&  \Q_{p}(g(T))^{-1}\overline{p} - g(T) \Q_{p}(g(T))^{-1} = S_L^{-1}(p,g(T)).
\end{align*}
We hence have
\begin{align*}
f(g(T)) = & \frac{1}{2\pi}\int_{\Gamma_p}   S_L^{-1}(p,g(T))\,dp_I\, f(p)= \frac{1}{2\pi} \int_{\Gamma_p} \left[S_L^{-1}(p,g(\cdot))\right](T) \,dp_I\, f(p).
\end{align*}

Let us first assume that $T$ is injective. Since $f$ and in turn also $f\circ g$ are bounded, we can use $e(p) = p(\id + p)^{-2}$ as a regulariser for $f\circ g$. As $e$ decays regularly at $0$ and infinity, also the functions $s\mapsto e(s) S_L^{-1}(p,g(s))$ decays regularly at $0$ and infinity for any $p\in\Gamma_p$. Hence it belongs to $\lholZI(\sector{\varphi})$ and so 
\begin{equation}\label{TZER}
\begin{split}
f(g(T)) =& e(T)^{-1} e(T) f(g(T))  \\
=& e(T)^{-1} \frac{1}{2\pi} \int_{\Gamma_p} e(T) S_L^{-1}(p,g(T))\,dp_I \, f(p)\\
=& e(T)^{-1} \frac{1}{2\pi} \int_{\Gamma_p}  \left[e(\cdot)S_L^{-1}(p,g(\cdot))\right](T)\, dp_I f(p)\\
=& e(T)^{-1} \frac{1}{(2\pi)^2} \int_{\Gamma_p} \left( \int_{\Gamma_s} S_L^{-1}(s,T)\,ds_I\, s(1+s)^{-2}S_L^{-1}(p,g(s))\right)\, dp_I f(p).
\end{split}
\end{equation} 

We can now apply Fubini's theorem in order to exchange the order of integration: estimating the resolvent using \eqref{SectCond}, we find that the integrand in the above integral is bounded by the function
\begin{equation}\label{BDMAJ}
 F(s,p):= C_{\theta}\left|p S_L^{-1}(p,g(s))\right| \frac{1}{|1+s|^2}  \frac{|f(p)|}{|p|}.
 \end{equation}
Since $p$, $s$ and $g(s)$ belong to the same complex plane as $g$ is intrinsic, we have due to  \eqref{SEst2} that
\begin{equation}\label{sax}
  \left|p  S_L^{-1}(p,g(s)) \right| \leq \max_{\tilde{s}\in[s]}\frac{|p|}{\left| p - g(\tilde{s})\right|} = \max \left\{ \frac{1}{|1-p^{-1}g(s)|}, \frac{1}{|1 - p^{-1}g(\overline{s})|}\right\}
\end{equation}
 Since  $g(\Gamma_s) \subset \sector{\rho'}\cap\cc_I \subsetneq \sector{\theta'}\cap\cc_I$ and $\Gamma_p = \partial(\sector{\theta'}\cap\cc_I)$, these expressions are bounded by a constant depending on $\theta'$ and $\rho'$ but neither on $p$ nor on $s$. Hence $\left|p  S_L^{-1}(p,g(s)) \right|$ is uniformly bounded on $\Gamma_s\times\Gamma_p$ and $F(s,p)$ is in turn integrable on $\Gamma_p\times \Gamma_s$ because $f$ has polynomial limit $0$ both at $0$ and infinity.
  
 After exchanging the order of integration in \eqref{TZER}, we deduce from Cauchy's integral formula that
  \begin{equation*}
  \begin{split}
f(g(T)) =& e(T)^{-1} \frac{1}{(2\pi)^2} \int_{\Gamma_s} S_L^{-1}(s,T)\,ds_I\, s(1+s)^{-2}\left(  \int_{\Gamma_p} S_L^{-1}(p,g(s))\, dp_I f(p)\right)\\
=& e(T)^{-1} \frac{1}{2\pi} \int_{\Gamma_s} S_L^{-1}(s,T)\,ds_I\, e(s)f(g(s)) = e(T)^{-1} e(T) (f\circ g)(T)  = (f\circ g) (T).
\end{split}
  \end{equation*}

Let us now consider the case that $T$ is not injective. By \Cref{NonInjLem}, the function $g$ has then  finite polynomial limit $g(0)\in\rr$ in $\sector{\varphi}$ and hence the function $\tilde{g}(p) = g(p) - g(0)\in\meroInt(\sector{\varphi})_{T}$ has finite polynomial limit $0$ in at $0$. Let us choose a regulariser $e$ for $\tilde{g}$ with polynomial limit $0$ at infinity. (This is always possible: if $\tilde{e}$ is an arbitrary regulariser for $\tilde{g}$,  we can choose for instance $e(s) = (1+s)^{-1}\tilde{e}(s)$.) We have then $e\tilde{g} \in \lholZI(\sector{\varphi})$. Since $g(0)$ is real, we have $S_L^{-1}(p,g(0))= (p-g(0))^{-1}$. Moreover  $g(s)$ and $\Q_{p}(g(s))^{-1}$ commute for any $s\in\Gamma_s$. For $p\notin\overline{\sector{\rho'}}$ we find thus
\begin{equation}\label{blulb}
\begin{split}
&e(s) S_L^{-1}(p,g(s)) - e(s)S_L^{-1}(p,g(0))\\
 = & e(s) \Q_{p}(g(s))^{-1}\left[ (\overline{p}- g(s))(p - g(0)) - \Q_{p}(g(s))\right](p-g(0))^{-1}\\
=& e(s) \Q_{p}(g(s))^{-1}\Big[ (\overline{p}- g(s))p - g(0)(\overline{p}- g(s)) \\
&\qquad \qquad \qquad \quad+ g(s)(\overline{p} - g(s)) - ( \overline{p} - g(s))p\Big](p-g(0))^{-1}\\
=& e(s) (g(s) - g(0)) S_L^{-1}(p,g(s)) (p - g(0))^{-1} \\
=&  e(s) \tilde{g}(s) S_L^{-1}(p,g(s)) S_L^{-1}(p,g(0)).
\end{split}
\end{equation}
Hence, $e$ regularises also $s\mapsto S_L^{-1}(p,g(s)) - S_L^{-1}(p,g(0))$ and $e(\cdot)\left(S_L^{-1}(p,g(\cdot)) - S_L^{-1}(p,g(0))\right)$ does even belong to $\lholZI(\sector{\varphi})$.  
We thus have 
\begin{align*}
f(g(T)) = & e(T)^{-1}e(T)f(g(T)) \\
= & e(T)^{-1} \frac{1}{2\pi}\int_{\Gamma_p} e(T) S_L^{-1}(p,g(T))\,dp_I\, f(p)\todo{Why is the $S$-resolvent compatible with the functional calculus?}\\
=&  e(T)^{-1}\frac{1}{2\pi} \int_{\Gamma_p} \left[e(\cdot) S_L^{-1}(p,g(\cdot))\right](T) \,dp_I\, f(p)\\
 =& e(T)^{-1}\frac{1}{2\pi} \int_{\Gamma_p} \left[e(\cdot) \tilde{g}(\cdot) S_L^{-1}(p,g(\cdot))S_L^{-1}(p,g(0))\right](T)\,dp_I f(p) \\
& + e(T)^{-1}\frac{1}{2\pi} \int_{\Gamma p} e(T) S_L^{-1}(p,g(0))\,dp_I f(p).
\end{align*}
For the second integral, Cauchy's integral formula yields 
\begin{equation}
e(T)^{-1}\frac{1}{2\pi} \int_{\Gamma p} e(T) S_L^{-1}(p,g(0))\,dp_I f(p) = e(T)^{-1}e(T) f(g(0)) = f(g(0))\id
\end{equation}
as $f$ decays regularly at infinity in $\sector{\theta}$. For the first integral, we have
\begin{equation}\label{cuko}
\begin{split}
 &e(T)^{-1}\frac{1}{2\pi} \int_{\Gamma_p} \left[e(\cdot) \tilde{g}(\cdot) S_L^{-1}(p,g(\cdot))S_L^{-1}(p,g(0))\right](T)\,dp_I f(p)\\
=& e(T)^{-1}\frac{1}{(2\pi)^2} \int_{\Gamma_p} \left(\int_{\Gamma_s} S_L^{-1}(s,T)\,ds_I\, e(s) \tilde{g}(s) S_L^{-1}(p,g(s))S_L^{-1}(p,g(0))\right)\,dp_I f(p)\\
 \overset{(A)}{=}& e(T)^{-1}\frac{1}{(2\pi)^2}\int_{\Gamma_s} S_L^{-1}(s,T)\,ds_I\, \left(  \int_{\Gamma_p} e(s) \tilde{g}(s) S_L^{-1}(p,g(s))S_L^{-1}(p,g(0))\,dp_I f(p)\right)\\
 \overset{(B)}{=}& e(T)^{-1}\frac{1}{(2\pi)^2}\int_{\Gamma_s} S_L^{-1}(s,T)\,ds_I\, e(s)\left(  \int_{\Gamma_p} S_L^{-1}(p,g(s)) - S_L^{-1}(p,g(0))\,dp_I f(p)\right)\\
 \overset{(C)}{=}& e(T)^{-1}\frac{1}{2\pi}\int_{\Gamma_s} S_L^{-1}(s,T)\,ds_I \left(e(s) f(g(s)) - f(g(0))\right)\\
 = & e(T)^{-1} (e(T) f\circ g(T) - e(T) f(g(0))\id) =  f\circ g(T) - f(g(0))\id,
\end{split}
\end{equation}
where the identity $(A)$ follows from Fubini's theorem, the identity $(B)$ follows from \eqref{blulb} and the identity $(C)$ finally follows from Cauchy's integral formula. 
Altogether, we have
\[ f(g(T)) = f\circ g (T) - f(g(0))\id + f(g(0))\id = f\circ g(T).\]
In order to justify the application of Fubini's theorem in $(A)$, we observe that the integrand is bounded the function by
\[
 F(s,p) = C_{\theta} \left|p S_L^{-1}(p,g(s))\right| \frac{|e(s)\tilde{g}(s)|}{|s|}\frac{|f(p)|}{|p|} \frac{1}{|p-g(0)|},
\]
where we used \eqref{SectCond} in order to estimate the $S$-resolvent $S_L^{-1}(s,T)$.

If $g(0) \neq 0$ then $|p-g(0)|^{-1}$ is uniformly bounded in $p$. Just as before, also $\left|p S_L^{-1}(p,g(s))\right|$ is uniformly bounded on $\Gamma_s\times \Gamma_p$. Since $\tilde{g}$ decays regularly at $0$, $e$ decays regularly at infinity and $f$ decays regularly both at $0$ and infinity, the function $F$  is hence integrable on $\Gamma_s\times \Gamma_p$ and we can apply Fubini's theorem.

If on the other hand $g(0) = 0$, then $g = \tilde{g}$ and we can write
\begin{align}
\notag
 F(s,p) =& C_{\theta} \left| S_L^{-1}(p,g(s))\right| \frac{|e(s)\tilde{g}(s)|}{|s|}\frac{|f(p)|}{|p|}\\
\label{vbh}
 =&C_{\theta} \left|p^{\alpha} S_L^{-1}(p,g(s))g(s)^{1-\alpha}\right| \frac{|e(s)g(s)^{\alpha}|}{|s|}\frac{|f(p)|}{|p|^{1+\alpha}},
 \end{align}
with $\alpha\in(0,1)$ such that $|f(p)|/|p|^{1+\alpha}$ is integrable. This is possible because $f$ decays regularly at~$0$. Just as in \eqref{sax}, we can estimate the first factor in \eqref{vbh} by
\[
 \left|p^{\alpha} S_L^{-1}(p,g(s))g(s)^{1-\alpha}\right| \leq \max\left\{ \frac{|g(s)|^{1-\alpha}}{|p|^{1-\alpha}} \frac{1}{|1 - p^{-1}g(s)|}, \frac{|g(\overline{s})|^{1-\alpha}}{|p|^{1-\alpha}} \frac{1}{|1 - p^{-1}g(\overline{s})|} \right\},
\]
where we used that $|g(s)| = |\overline{g(\overline{s})}| = |g(\overline{s})|$ because $g$ is intrinsic. This expression is  as before uniformly bounded on $\Gamma_s\times\Gamma_p$ because $g(\Gamma_s)\subset \sector{\rho'}\cap\cc_I$. Hence, $F$ is again integrable and it is actually possible to apply Fubini's theorem.

Altogether, we have so far shown that $f(g(T)) = ( f\circ g )(T)$ for any $f\in\EL[\sector{\omega'}]$. Finally, we consider now a general function $f\in \meroL[\sector{\omega'}]_{g(T)}$ that does not necessarily belong to $\EL[\sector{\omega'}]$. If $e$ is a regulariser for $f$, then $e$ and $ef$ both belong to $\EL[\omega']$. By what we have just shown,  we hence have $e_g := e\circ g  \in\meroInt[\sector{\omega}]_T$ and  $(ef)_g := (ef)\circ g \in\meroL[\sector{\omega}]_T$ with $e_g(T) = e(g(T))$ and $(ef)_g(T) = (ef)(g(T))$. 

Let $\tau_1$ and $\tau_2$ be regularisers for $e_ g$ and $(ef)_ g$. Then $\tau = \tau_1\tau_2$ regularises both of them and hence
\[ e_g(T) = \tau^{-1}(T) (\tau e_g)(T).\]
Since $e_g(T) = (e\circ g)(T) = e(g(T))$ is injective because $e$ is a regulariser for $f$, the operator $ (\tau e_g)(T)$ is injective too. Moreover, for $f_g :=f\circ g$, we find $(\tau e_g) f_g = \tau(e_gf_g) = \tau (ef)_g \in \EL[\omega]$ because $\tau$ was chosen to regularise both  $e_g$ and $(ef)_g$. Therefore $\tau e_ g$ is a regulariser for $f_g$ and hence $f_g \in\meroL[\sector{\omega}]_T$.  Finally, we deduce from \Cref{MoProp} that
\begin{align*}
 f(g(T)) =& e(g(T))^{-1} (ef)(g(T)) = (e_ g)(T)^{-1} ((ef)_ g )(T) \\
 =& (e_ g)(T)^{-1}\tau(T)^{-1}\tau(T) ((ef)_ g )(T) \\
 = &(\tau e_ g)(T)^{-1}((\tau e)_ g f_g)(T) =  f_g (T) = (f\circ g)(T).
 \end{align*}

\end{proof}

\begin{corollary}\label{InvCor}
Let $T\in\sectOP(\omega)$ be injective and let $f\in\meroL[\sector{\omega}]$. Then $f\in\meroL[\sector{\omega}]_T$ if and only if $p\mapsto f(p^{-1})\in\meroL[\sector{\omega}]_{T^{-1}}$ and in this case
\[ f(T) = f(p^{-1}) (T^{-1}).\]
\end{corollary}
\begin{proof}
Since $T$ is injective, the function $p^{-1}$ belongs to $\meroInt[\sector{\omega}]_{T}$ and the statement follows from \Cref{CompRul}.

\end{proof}

\subsection{Extensions according to spectral conditions}
As in the complex case, cf. \cite[Section 2.5]{Haase}, one can extend the $H^{\infty}$-functional calculus for sectorial operators to a larger class of functions if the operator satisfies additional spectral conditions. We shall mention the following three cases, which are relevant in the proof the spectral mapping theorem in \Cref{SpecMapSubSec}. In order to explain them, we introduce the notation
\[\sector{\varphi,r,R} = (\sector{\varphi} \cap B_{R}(0))\setminus B_{r}(0)\]
 for $0\leq r<R\leq \infty$. (We set $B_{\infty}(0) = \hh$ for $R = \infty$.)
\begin{enumerate}[label = (\roman*)]
\item \label{itemA} If the operator $T\in\sectOP(\omega)$ has a bounded inverse, then $B_{\varepsilon}(0)\subset \rho_S(T)$ for sufficiently small $\varepsilon >0$. We can thus define the class\[ \EL^{\infty}(\sector{\varphi}) = \{ f = \tilde{f} + a \in\lhol(\sector{\varphi}) : a\in\hh, \tilde{f}\in\lhol(\sector{\varphi}) \text{ decays regularly at $\infty$}\}\]
and $\Eint^{\infty}(\sector{\varphi})$ as the set of all intrinsic functions in $\sector{\varphi}$. For any function $f\in\EL^{\infty}(\varphi)$ with $\varphi>0$, we can define $f(T)$ as
\[ f(T) = \frac{1}{2\pi} \int_{\partial( \sector{\varphi,r,\infty}\cap\cc_I)} S_L^{-1}(s,T)\,ds_I\,f(s) + a\id,\]
with $0 < r < \varepsilon$ arbitrary. It follows as in \Cref{Kakao} from Cauchy's integral theorem that this approach is consistent with the usual one if $f\in\EL(\sector{\varphi})$, but the class of admissible functions $\EL^{\infty}(\sector{\varphi})$ is now larger. We can further extend this functional calculus by calling $e\in\EL^{\infty}(\sector{\varphi})$ a regulariser for $f\in\meroL(\sector{\varphi})$, if $e(T)$ is injective and $ef \in\EL^{\infty}(\sector{\varphi})$. In this case, we define $f(T) = e(T)^{-1}(ef)(T)$.

Obviously all the results shown so far still hold for this extended functional calculus since the respective proofs can be carried out in this setting with marginal and obvious modifications. Only in the case of the composition rule we have to consider several cases, just as in the complex case, namely the combinations
\begin{enumerate}[label = \alph*)]
\item \label{AA}$T$ is sectorial and $g(T)$ is invertible and sectorial
\item $T$ is invertible and sectorial and $g(T)$ is sectorial
\item \label{AC} $T$ and $g(T)$ are both invertible and sectorial.
\end{enumerate}
In the cases \ref{AA} and \ref{AC} one needs the additional assumption $0\notin \overline{g(\sector{\omega})}$ on the function $g$.

\item \label{itemB} If the operator $T\in\sectOP(\omega)$ is bounded, then $\hh\setminus B_{\rho}(0)\subset \rho_S(T)$ for sufficiently large $\rho >0$. We can thus define the class
\[ \EL^{0}(\sector{\varphi}) = \{ f = \tilde{f} + a \in\lhol(\sector{\varphi}) : a\in\hh, \tilde{f}\in\lhol(\sector{\varphi}) \text{ decays regularly at $0$}\}\]
and $\Eint^{0}(\sector{\varphi})$ as the set of all intrinsic functions in $\EL^{0}(\sector{\varphi})$. For any function $f\in\EL^{\infty}(\varphi)$ with $\varphi>0$, we can define $f(T)$ as
\[ f(T) = \frac{1}{2\pi} \int_{\partial( \sector{\varphi,0,R}\cap\cc_I)} S_L^{-1}(s,T)\,ds_I\,f(s) + a\id,\]
with $0 < \rho < R$ arbitrary. As before this approach is consistent with the usual one if $f\in\EL(\sector{\varphi})$, but the class of admissible functions $\EL^{0}(\sector{\varphi})$ is again larger than $\EL(\sector{\varphi})$. We can further extend this functional calculus by calling $e\in\EL^{0}(\sector{\varphi})$ a regulariser for $f\in\meroL(\sector{\varphi})$, if $e(T)$ is injective and $ef \in\EL^{0}(\sector{\varphi})$ and define again $f(T) = e(T)^{-1}(ef)(T)$ for such $f$.

As before, all  results shown so far hold also for this extended functional calculus because the respective proofs can be carried out in this setting with marginal and obvious modifications. For showing the composition rule, we have to consider again several cases and distinguish the following situations:
\begin{enumerate}[label = \alph*)]
\item \label{BA} $T$ is sectorial and  $g(T)$ is bounded and sectorial 
\item \label{BB} $T$ is invertible and sectorial and $g(T)$ is bounded and sectorial
\item \label{BC} $T$ and $g(T)$ are both bounded and sectorial
\item $T$ is bounded and sectorial and $g(T)$ is sectorial
\item \label{BE}$T$ is bounded and sectorial and $g(T)$ is invertible and sectorial.
\end{enumerate}
In the cases \ref{BA}, \ref{BB} and \ref{BC} one needs the additional assumption $\infty\notin \overline{g(\sector{\omega})}^{\hh_{\infty}}$ and in the case one needs the additional assumption \ref{BE} $0\notin \overline{g(\sector{\omega})}$ on the function $g$.

\item If finally $T\in\sectOP(\omega)$ is bounded and has a bounded inverse, then we can set $\EL^{0,\infty}(\sector{\varphi}) = \lhol(\sector{\varphi})$ and $\Eint^{0,\infty}(\sector{\varphi})$ and define for such functions
\[ f(T) = \frac{1}{2\pi} \int_{\partial (\sector{\varphi,r,R}\cap\cc_I)} S_L^{-1}(s,T)\,ds_I\, f(s)\]
for sufficiently small $r$ and sufficiently large $R$. Choosing  regularisers in $\Eint^{0,\infty}(\sector{\varphi})$ gives again an extension of the $H^{\infty}$-functional calculus and of the two extended functional calculi presented in \ref{itemA} and \ref{itemB}. All the results presented so far still hold for this extended functional calculs, where the composition rule can be shown again under suitable conditions on the function $g$.

\end{enumerate}

\subsection{The spectral mapping theorem}\label{SpecMapSubSec}\label{SpecThmSection}
Finally, let us now show the spectral mapping theorem for the $H^{\infty}$-functional calculus. We point out that a substantial technical difficulty will appear here that does not occur in the classical situation: the proof of the spectral mapping theorem in the complex setting makes use of the fact that $f\left(T|_{V_{\sigma}}\right) = \left.f(T)\right|_{V_\sigma}$ if $\sigma$ is a spectral set and $V_{\sigma}$ is the invariant subspace associated with $\sigma$, i.e. the range of the spectral projection $\chi_{\sigma}(T)$ defined in \Cref{SCProp}. However,  subspaces that are invariant under right linear operators are in general only right linear subspaces, but not necessarily left linear subspaces. Hence, they are not two-sided Banach spaces and we cannot define $f\left(T|_{V_{\sigma}}\right)$ because the $S$-functional calculus as introduced in \Cref{SCalcBd}, \Cref{SCalcCl} and \Cref{SectSCalc} requires the Banach space to be two-sided. The $S$-resolvents can otherwise not be defined. (A different approach to the $S$-functional calculus, that does not require a left multiplication on the Banach space was recently introduced in \cite{SpecOP}. This approach applies however only to intrinsic functions.) Instead of using the properties of the $S$-functional calculus for $T|_{V_{\sigma}}$ we thus have to find a workaround and prove several steps directly, which is essentially done in \Cref{fProj}.

We start with two technical lemmas that are necessary in order to show the spectral inclusion theorem.
\begin{lemma}\label{HLkalulu}
Let $T\in\sectOP(\omega)$ and let $s\in\hh$. If $\Q_{s}(T)$ is injective and there exist $e\in\meroInt[\sector{\omega}]_T$ and $c\in\hh$, $c\neq 0$ such that 
\[ f(q) := \Q_{c}(e(q))\, Q_{s}(q)^{-1}  \in\meroInt[\sector{\omega}]_T\]
and such that $e(T)$ and $f(T)$ are bounded, then $e(T)\Q_{s}(T)^{-1} = \Q_{s}(T)^{-1}e(T)$.
\end{lemma}
\begin{proof}
By assumption the operator $\Q_{s}(T)$ is injective and hence \ref{MoProp3} in \Cref{MoProp} implies that $\Q_{s}^{-1}\in\meroInt[\omega]_T$. Since $e(T)$ is bounded, it commutes with $T$ and so also with $\Q_{s}(T)^{-1}$. We thus have
\[ e(T)\Q_{s}(T)^{-1} \subset \Q_{s}(T)^{-1}e(T).\]
In order to show that this relation is actually an equality, it is sufficient to show that $v \in\dom\left(\Q_{s}(T)^{-1}\right)$ for any $v\in V$ with $e(T)v \in\dom(\Q_{s}(T)^{-1})$. This is indeed the case: if  $e(T)v$ belongs to $\dom(\Q_{s}(T)^{-1})$, then  there exists $u\in \dom(\Q_{s}(T))$ with $e(T)v = \Q_{s}(T)u$. Hence
\begin{equation}\label{popu3}
\begin{gathered}
 \Q_{c}(e(T)) v = e(T)^2 v -2c_0 e(T)v + |c|^2v \\
 = e(T)\Q_{s}(T)u - 2c_0 \Q_{s}(T)u + |c|^2 v = \Q_{s}(T)(e(T) u - 2c_0 u) + |c|^2 v,
\end{gathered}
\end{equation}
where the last identity follows again from \ref{MoProp1} in \Cref{MoProp} because $e(T)$ is bounded and commutes with $T$ and in turn also with $\Q_{s}(T)$. Since $f(T)\in\boundOP(V)$, we conclude on the other hand from \ref{MoProp2} of \Cref{MoProp} that 
\[
\Q_{c}(e(T)) = \Q_{s}(T) \left[\Q_{c}(e(\cdot)) \Q_{s}(\cdot)^{-1} \right](T) = \Q_{s}(T) f(T).
\]
Due to \eqref{popu3}, we then find
\begin{align*} 
v& = \frac{1}{|c|^2} \left( \Q_{c}(e(T))v - \Q_{s}(T)(e(T) u - 2c_0 u) \right) \\
&= \Q_{s}(T)\frac{1}{|c|^2} \left( f(T)v - e(T) u + 2c_0 u) \right).
\end{align*}
Hence, $v$ belongs to $ \dom(\Q_{s}(T)^{-1})$ and the statement follows.

\end{proof}

\begin{lemma}\label{RegNot0}
Let $T\in\sectOP(\omega)$ and let $f\in\meroL[\sector{\omega}]_T$. For any $s\in\clos{\sector{\omega}}$, $s\neq 0$ there exists a regulariser $e$ for $f$ with $e(s) \neq 0$.
\end{lemma}
\begin{proof}
Let $\tilde{e}$ be an arbitrary regulariser of $f$ such that $\tilde{e}\in\Eint[\sector{\omega}]$, $\tilde{e}f \in \EL[\sector{\omega}]$ and $\tilde{e}(T)$ is injective. If $\tilde{e}(s) \neq 0$, then we can set $e = \tilde{e}$ and we are done. Otherwise recall that $[s]$ is a spherical zero of $\tilde{e}$ and that its order is a finite number $n\in\nn$ since $e\not\equiv 0$ as $e(T)$ is injective. We define now $e(q) := \Q_{s}^{-n}(q)e(q)$ with $\Q_{s}(q) = q^2 - 2s_0q + |s|^2$. Then $e\in\Eint[\sector{\omega}]$ with $e(s) \neq 0$ and $ef = Q_{s}^{-n}\tilde{e}f\in\EL[\sector{\omega}]$. Furthermore, by \cref{MoProp2} in \Cref{MoProp}, we have $\tilde{e}(T) = \Q_{s}(T) e(T)$. Since $\tilde{e}(T)$ is injective, we deduce that also $e(T)$ is injective. Hence $e$ is a regulariser for $f$ with $e(s) \neq 0$.
 
\end{proof}

\begin{lemma}\label{Asix}
Let $T\in\sectOP(\omega)$ and let $s\in\clos{\sector{\omega}}$ with $s\neq 0$. If  $f(T)$  has a bounded inverse for some  $f\in\meroInt[\sector{\omega}]_T$ with $f(s) = 0$, then $s\in\rho_S(T)$.
\end{lemma}
\begin{proof}
Let $f$ be as above and let us first show that $\Q_{s}(T) = T^2 - 2s_0T + |s|^2\id$ is injective and hence invertible as a closed operator. By \Cref{RegNot0},  there exists a regulariser $e$ for $f$  with $c:= e(s)\neq 0$. We have $ef \in\Eint[\sector{\omega}]$ with $(ef)(s) = 0$. Since all zeros of intrinsic functions are spherical zeros, we find that also $h = ef \Q_{s}^{-1} = \Q_{s}^{-1}ef \in \Eint[\sector{\omega}]$. The product rule \ref{MoProp2} in \Cref{MoProp} implies therefore
\begin{equation*}
h(T)\Q_{s}(T) \subset (h\Q_{s})(T) = (ef)(T) = (fe)(T)= f(T) e(T),
\end{equation*}
where $ef = fe$ because both functions are intrinsic. Since $e(T)$ and $f(T)$ are both injective, we find that also $\Q_{s}(T)$ is injective. Moreover, $e$ is also a regulariser for $\Q_{s}^{-1}f$.

Now observe that the function $g(q):=\Q_{c}(e(q))\Q_{s}(q)^{-1} = (e(q)^2 - 2 c_0e(q) + |c|^2)(q^2-2s_0q+|s|^2)^{-1}$ belongs to $\Eint[\sector{\omega}]$. Indeed, by \Cref{EStruct}, the space $\Eint[\sector{\omega}]$ is a real algebra such  $\Q_{c}(e(q)) = e(q)^2 -2c_0e(q) + |c|^2$ belongs to it as $e$ does. The function $\Q_{c}(e(q))$ however has a spherical zero at $s$ because $e(s) = c$ such that $ g(q)= \Q_{c}(e(q))\Q_{s}^{-1}(q)$ is bounded and hence belongs to $\Eint[\sector{\omega}]$ by \Cref{EChar}. In particular this implies that $g(T)$ is bounded.

 We deduce from \Cref{HLkalulu} that $e(T)\Q_{s}(T)^{-1} = \Q_{s}(T)^{-1}e(T)$ and inverting both sides of this equation yields $ \Q_{s}(T)e(T)^{-1} = e(T)^{-1}\Q_{s}(T)$. The product rule in \cref{MoProp2} of \Cref{MoProp}, the boundedness of $h(T) = (e\Q_{s}^{-1}f)(T)$ and the fact that $\Q_{s}^{-1}$ and $e$ commute because both are intrinsic functions  imply
\begin{gather*}
f(T) =  e(T)^{-1}(e f)(T) = e(T)^{-1}\left(\Q_{s} e\Q_{s}^{-1}f\right)(T) = e(T)^{-1}\Q_{s}(T)\left(e\Q_{s}^{-1}f\right)(T) \\
= \Q_{s}(T) e(T)^{-1}(e\Q_{s}^{-1}f)(T) = \Q_{s}(T) (\Q_{s}^{-1}f)(T).
\end{gather*}
Since $f(T)$ is surjective, we find that $\Q_{s}(T)$ is surjective too. Hence $\Q_{s}(T)^{-1}$ is an everywhere defined closed operator and thus bounded by the closed graph theorem. Consequently $s\in\rho_S(T)$. 

\end{proof}

\begin{proposition}\label{SpecInc}
If $T\in\sectOP(\omega)$ and  $f\in\meroInt[\sector{\omega}]_T$, then 
\[f(\sigma_S(T)\setminus\{0\})\subset\sigma_{SX}(f(T)).\]
\end{proposition}
\begin{proof}
Let $s\in\sigma_S(T)\setminus\{0\}$ and set $c := f(s)$. If $c \neq \infty$, then \Cref{Asix} implies that $\Q_{c}(f(T))^{2} = f(T)^2 - 2c_0f(T) + |c|^2\id$ does not have a bounded inverse because $g= f^2 -2c_0f + |c|^2$ belongs to $\meroInt[\sector{\omega}]_T$ and satisfies $g(c)  =0 $. Hence $c=f(s)\in\sigma_{S}(f(T))$ for $s\in\sigma_{S}(T)\setminus\{0\}$ with $f(s) \neq \infty$.

If on the other hand $c = \infty$, then suppose that $c\notin\sigma_{SX}(f(T))$, i.e. that $f(T)$ is bounded. In this case there exists $a\in\hh$ such that $\Q_{a}(f(T))$ has a bounded inverse. By \ref{MoProp3} in \Cref{MoProp}, this implies $g(p) = \Q_{a}(f(p))^{-1}\in\meroInt[\sector{\omega}]_T$. The operator $g(T)$ is invertible as $g(T)^{-1} = \Q_{a}(f(T))$ belongs to $\boundOP(V)$ because $f(T)$ is bounded. Since moreover $g(s) = 0$ as $f(s) = \infty$, another application of \Cref{Asix} yields $s\in\rho_S(T)$. But this contradicts our assumption  $s\in\sigma_{S}(T)\setminus\{0\}$. Hence, we must have $c\in\sigma_{SX}(f(T))$.

\end{proof}

We have so far shown the spectral inclusion theorem for spectral values not equal to $0$ or $\infty$. These two values need a special treatment. They also need additional assumptions on the function $f$ for a spectral inclusion theorem to hold as we shall see in the following. (The assumptions presented here might however not be the most general ones that are possible, cf. \cite{Haase:2005}.)

 First we however have to show a technical lemma. We start with recalling the spectral projections associated with subsets of the extended $S$-spectrum.  Let $\sigma\subset\sigma_{SX}(T)$ be a spectral set, i.e. a subset that is open and closed in $\sigma_{SX}(T)$. By \Cref{SCProp}, the operator  $E_{\sigma}:= \chi_{\sigma}(T)$ is a projection that commutes with $T$, i.e. it is a projection onto a right-linear subspace of $V$ that is invariant under $T$. More precisely, if $\infty\notin \sigma$, then we can choose a bounded slice Cauchy domain $U_{\sigma}\subset\hh$ such that $\sigma\subset U_{\sigma}$ and such that $(\sigma_S(T) \setminus\sigma ) \cap U_{\sigma} = \emptyset$. The projection then $E_{\sigma}$ is given by
\begin{equation}\label{ProjLeft}
 E_{\sigma} = \frac{1}{2\pi}\int_{\partial(U_{\sigma}\cap\cc_I)} ds_I\, S_R^{-1}(s,T) = \frac{1}{2\pi} \int_{\partial(U_{\sigma}\cap\cc_I)} S_L^{-1}(p,T)\,dp_I.
 \end{equation}
 If on the other hand $\infty\in\sigma$, then we can choose an unbounded slice Cauchy domain $U_{\sigma}\subset\hh$ such that $\sigma\subset U_{\sigma}$ and such that $(\sigma_S(T) \setminus\sigma ) \cap U_{\sigma} = \emptyset$. The projection $E_{\sigma}$ is then given by
\begin{equation*}
 E_{\sigma} = \id+  \frac{1}{2\pi}\int_{\partial(U_{\sigma}\cap\cc_I)} ds_I\, S_R^{-1}(s,T) = \id + \frac{1}{2\pi} \int_{\partial(U_{\sigma}\cap\cc_I)} S_L^{-1}(p,T)\,dp_I.
 \end{equation*}

\begin{lemma}\label{fProj}
Let $T\in\sectOP(\omega)$ be unbounded and assume that $\sigma_{S}(T)$ is bounded. Furthermore let $E_{\infty}$ be the spectral projection onto the invariant subspace associated to $\infty$. If $f\in\meroInt[\sector{\omega}]_T$ has polynomial limit $0$ at infinity, then $\{f(T)\}_{\infty} = f(T)E_{\infty}$ is a bounded operator that is given by the slice hyperholomorphic Cauchy integral
\begin{equation}\label{fProjRep}
\{f(T)\}_{\infty} = \int_{\partial(\sector{\varphi}\setminus B_{r}(0))\cap\cc_I} f(s)\,ds_I\,S_{R}^{-1}(s,T),
 \end{equation}
where $B_{r}(0)$ is the ball centered at $0$ with  $r>0$ sufficiently large such that it contains $\sigma_{S}(T)$ and any singularity of $f$. Moreover, for two such functions, we have
\begin{equation}\label{fProjProd}
 \{f(T)\}_{\infty}\{g(T)\}_{\infty} =  \{(fg)(T)\}_{\infty}.
 \end{equation}
\end{lemma}
\begin{proof}
Let us first assume that $f \in \Eint[\sector{\omega}]$, i.e. $f\in\Eint(\sector{\varphi})$ with $\omega< \varphi < \pi$. Since $f$ decays regularly at infinity, it is of the form
$f(s) = \tilde{f}(s) + a(1+s)^{-1}$ with $a \in\rr$ and $f\in \intrinZI(\sector{\varphi})$.  The operator $\tilde{f}(T)$ is given by the slice hyperholomorphic Cauchy integral 
\begin{align}\label{SAMSTI}
\tilde{f}(T) = \frac{1}{2\pi} \int_{\partial(\sector{\varphi'}\cap\cc_{I})} \tilde{f}(s)\,ds_I\,S_{R}^{-1}(s,T) 
\end{align}
with $I\in\SS$ and $\varphi' \in  (\omega,\varphi)$.
 Let now $r_1 < r_2$ be such that $\sigma_S(T)\subset B_r(0)$. Cauchy's integral theorem allows us to replace the path of integration in \eqref{SAMSTI} by the union of $\Gamma_{s,1} = \partial(\sector{\varphi'}\cap B_{r_1}(0))\cap\cc_I$ and $\Gamma_{s,2} = \partial(\sector{\varphi'}\setminus B_{r_2}(0))\cap\cc_I$ such that 
\begin{equation}\label{Lky}
\tilde{f}(T) = \frac{1}{2\pi} \int_{\Gamma_s,1} \tilde{f}(s)\,ds_I\,S_{R}^{-1}(s,T) +  \frac{1}{2\pi} \int_{\Gamma_s,2} \tilde{f}(s)\,ds_I\,S_{R}^{-1}(s,T).
\end{equation}

Let us choose $R \in (r_1,r_2)$. Since $\sigma_{SX}(T) = \sigma_{S}(T)\cup\{\infty\}$, we have $E_{\infty} = \id - E_{\sigma_{S}(T)}$ and the spectral projection $E_{\sigma_{S}(T)}$ is given by the slice hyperholomorphic Cauchy integral \eqref{ProjLeft} along $\Gamma_p = \partial (B_{R}(0)\cap\cc_I)$. The subscripts $s$ and $p$ in $\Gamma_{s,1}$, $\Gamma_{s,2}$ and $\Gamma_{p}$ are chosen in order to indicate the corresponding variable of integration in the following computation.

If we write the operators $\tilde{f}(T)$ and $E_{\sigma}$ in terms of the slice hyperholomorphic Cauchy integrals defined above, we find that
\begin{equation}\label{HHgh}
\begin{aligned}
\tilde{f}(T)E_{\sigma} =& \frac{1}{2\pi} \int_{\Gamma_s,1} \tilde{f}(s)\,ds_I\,S_{R}^{-1}(s,T) \frac{1}{2\pi}\int_{\Gamma_p} S_{L}^{-1}(p,T)\,dp_I \\
&+  \frac{1}{2\pi} \int_{\Gamma_s,2} \tilde{f}(s)\,ds_I\,S_{R}^{-1}(s,T) \frac{1}{2\pi}\int_{\Gamma_p} S_{L}^{-1}(p,T)\,dp_I.
\end{aligned}
\end{equation}
If we apply the $S$-resolvent equation \eqref{SresEQ1} in the first integral, which we denote by $\Psi_1$ for neatness, we find 
\begin{equation}\label{jakutzi}
\begin{aligned}
\Psi_1 = & \frac{1}{(2\pi)^{2}} \int_{\Gamma_{s,1}} \tilde{f}(s)\,ds_I\,S_{R}^{-1}(s,T)  \int_{\Gamma_{p}}p\left(p^2-2s_0p+|s|^{2}\right)^{-1}\,dp_I\\
& - \frac{1}{(2\pi)^{2}} \int_{\Gamma_{s,1}} \tilde{f}(s)\,ds_I\,\overline{s}S_{R}^{-1}(s,T)\int_{\Gamma_{p}} \left(p^2-2s_0p+|s|^{2}\right)^{-1}\,dp_I\\
& -  \frac{1}{(2\pi)^{2}} \int_{\Gamma_{p}}\left( \int_{\Gamma_{s,1}} \tilde{f}(s)\,ds_I\,(S_L^{-1}(p,T)p - \overline{s}S_L^{-1}(p,T))\left(p^2-2s_0p+|s|^{2}\right)^{-1}\right)\,dp_I.
\end{aligned}
\end{equation}
For $s\in\Gamma_s$, the the functions $p\mapsto  \left(p^2-2s_0p+|s|^{2}\right)^{-1}$ and  $p\mapsto  p\left(p^2-2s_0p+|s|^{2}\right)^{-1}$ are rational functions on $\cc_I$ that have two singularities, namely $s = s_0 + I s_1$ and $\overline{s} = s_0 - I s_1$. Since we chose $r_1 < R$, these singularities lie inside of $B_{R}(0)$ for any $s\in\Gamma_s$. As $\Gamma_p = \partial(B_R(0) \cap \cc_{I})$, the residue  theorem  yields
\[ \frac{1}{2\pi}\int_{\Gamma_p}p\left(p^2 - 2s_0 p +|s|^2\right)^{-1}dp_I =  \lim_{\cc_I\ni p\to s} p(p-\overline{s})^{-1} + \lim_{\cc_I\ni p\to \overline{s}} p (p - s)^{-1} = 1\]
and
\[ \frac{1}{2\pi}\int_{\Gamma_p}\left(p^2 - 2s_0 p +|s|^2\right)^{-1}dp_I =  \lim_{\cc_I\ni p\to s} (p-\overline{s})^{-1} + \lim_{\cc_I\ni p\to \overline{s}}  (p - \overline{s})^{-1} = 0,\]
where $\lim_{\cc_I\ni p\to s} \tilde{f}(p)$ denotes the limit of $\tilde{f}(p)$ as $p$ tends to $s$ in $\cc_I$. If we apply the identity \eqref{ComRel} with $B = S_L^{-1}(p,T)$ in the third integral in \eqref{jakutzi}, it turns into
\begin{align*}
&  \frac{1}{(2\pi)^{2}} \int_{\Gamma_{p}}\left( \int_{\Gamma_{s,1}} \tilde{f}(s)\,ds_I\,\left(s^2-2p_0s+|p|^{2}\right)^{-1} sS_L^{-1}(p,T)\right)\,dp_I\\
&-  \frac{1}{(2\pi)^{2}} \int_{\Gamma_{p}}\left( \int_{\Gamma_{s,1}} \tilde{f}(s)\,ds_I\,\left(s^2-2p_0s+|p|^{2}\right)^{-1}S_L^{-1}(p,T)\overline{p}\right)\,dp_I = 0.
\end{align*}
The last identity follows from Cauchy's integral theorem because $\tilde{f}(s)$ is right slice hyperholomorphic and $s \mapsto (s^2 - 2p_0s + |p|^2)^{-1}S_{L}^{-1}(p,T)$ and $s \mapsto s (s^2 - 2p_0s + |p|^2)^{-1}S_{L}^{-1}(p,T)$ are left slice hyperholomorphic on $\sector{\varphi'}\cap B_{r_1}(0)$ for any $p \in \Gamma_{p}$ as we chose $R > r_1$. Hence, we find
\[
 \Psi_1 = \frac{1}{2\pi} \int_{\Gamma_{s,1}} \tilde{f}(s)\,ds_I\, S_R^{-1}(p,T). 
 \]
The second integral in \eqref{HHgh}, which we denote by $\Psi_2$ neatness, turns after an application of the $S$-resolvent equation \eqref{SresEQ1} into
\begin{equation}\label{poile}
\begin{aligned}
\Psi_2 = & \frac{1}{(2\pi)^{2}} \int_{\Gamma_{s,2}} \tilde{f}(s)\,ds_I\,S_{R}^{-1}(s,T)  \int_{\Gamma_{p}}p\left(p^2-2s_0p+|s|^{2}\right)^{-1}\,dp_I\\
& - \frac{1}{(2\pi)^{2}} \int_{\Gamma_{s,2}} \tilde{f}(s)\,ds_I\,\overline{s}S_{R}^{-1}(s,T)\int_{\Gamma_{p}} \left(p^2-2s_0p+|s|^{2}\right)^{-1}\,dp_I\\
& -  \frac{1}{(2\pi)^{2}} \int_{\Gamma_{s,2}} \left( \int_{\Gamma_{p}} \tilde{f}(s)\,ds_I\,(S_L^{-1}(p,T)p - \overline{s}S_L^{-1}(p,T))\left(p^2-2s_0p+|s|^{2}\right)^{-1}\right)\,dp_I.
\end{aligned}
\end{equation}
Since we chose $R<r_2$ the singularities of $p\mapsto  \left(p^2-2s_0p+|s|^{2}\right)^{-1}$ and  $p\mapsto  p\left(p^2-2s_0p+|s|^{2}\right)^{-1}$ lie outside of $\overline{B_{R}(0)}$ for any $s\in\Gamma_{s,2}$. Hence, these functions are right slice hyperholomorphic on $\overline{B_{R}(0)}$ and so Cauchy's integral theorem implies that the first two integrals in \eqref{poile} vanish. Since $\tilde{f}$ decays regularly at infinity, \eqref{SectCond} holds true and $\Gamma_p$ is a path of finite length, we can apply Fubini's theorem and exchange the order of integration in the third integral of \eqref{poile}.  After applying the identity \eqref{ComRel}, we find 
\[ \Psi_{2} =   \frac{1}{(2\pi)^{2}} \int_{\Gamma_{p}}\left( \int_{\Gamma_{s,2}} \tilde{f}(s) ds_I\,\left(s^2 - 2p_0s + |p|^2\right)^{-1}\left( sS_L^{-1}(p,T)-S_L^{-1}(p,T)\overline{p}\right)\right)\,dp_I.\]
However, also this integral vanishes : as  $f$ decays regularly at infinity, the integrand decays sufficiently so that we can use Cauchy's integral theorem to transform the path of integration and write
\begin{align*}
& \int_{\Gamma_{s,2}} \tilde{f}(s) ds_I\,\left(s^2 - 2p_0s + |p|^2\right)^{-1}\left( sS_L^{-1}(p,T)-S_L^{-1}(p,T)\overline{p}\right) \\
= & \lim_{\rho\to+\infty} \int_{\partial(U_{\rho}\cap\cc_I)} \tilde{f}(s) ds_I\,\left(s^2 - 2p_0s + |p|^2\right)^{-1}\left( sS_L^{-1}(p,T)-S_L^{-1}(p,T)\overline{p}\right) = 0
\end{align*}
where $U_{\rho} = ( \sector{\varphi}\setminus U_{r_2} ) \cap U_{\rho}$. The last identity follows again from Cauchy's integral theorem because the singularities $p$ and $\overline{p}$ of $s\mapsto (s^2 - 2p_0s + |p|^2)^{-1}$ and $s\mapsto (s^2 - 2p_0s + |p|^2)^{-1}s$ lie outside of $\overline{U_{\rho}}$ because we chose $R < r_2$. 

Putting these pieces together, we find that
\begin{equation}
\tilde{f}(T)E_{\sigma} = \frac{1}{2\pi}\int_{\Gamma_{s,1}}\tilde{f}(p)\,dp_I\,S_R^{-1}(p,T).
\end{equation}
We therefore deduce from \eqref{Lky} and $E_{\infty} = \id - E_{\sigma}$ that
\begin{equation}\label{ASSSA}
 \tilde{f}(T)E_{\infty} = \tilde{f}(T) - \tilde{f}(T) E_{\sigma}= \frac{1}{2\pi}\int_{\Gamma_{s,2}}\tilde{f}(p)\,dp_I\,S_R^{-1}(p,T).
 \end{equation}

Let us now consider the operator $a(\id + T)^{-1}$. Since it is slice hyperholomorphic on $\sigma_{S}(T)$ and at infinity, it is admissible for the $S$-fuctional calculus. If we set $\chi_{\{\infty\}}(s) := \chi_{\hh\setminus U_{R}(0)}$ (that is $\chi_{\{\infty\}}(s) = 1$ if $s\notin U_R(0)$ and $\chi_{\{\infty\}}(s) = 0$ if $s\in \overline{U_R(0)}$), then $\chi_{\{\infty\}}(T) = E_{\infty}$ via the $S$-functional calculus. The product rule of the $S$-functional calculus yields $a(\id + T)^{-1}E_{\infty} = g(T)$ with $g(s) =  a(1+s)\chi_{\{\infty\}}(s)$. If we set 
\[
U_{\rho,1} := (\sector{\varphi}\setminus B_{r_2}(0)) \cup (\hh\setminus B_{\rho}(0)) \quad \text{and}\quad U_{\rho, 2} = (\sector{\varphi}\cap B_{r_1}(0))\cup B_{\varepsilon}(0)
\]
 with $0 <\varepsilon < 1$ sufficiently small, then $U_{\rho} = U_{\rho,1}\cup U_{\rho,2}$ is an unbounded slice Cauchy domain that contains $\sigma_S(T)$ and such that $g$ is slice hyperholomorphic on $\overline{U_{\rho}}$. Hence, 
\begin{align*}
a(\id + T)^{-1}E_{\infty} =& g(\infty)\id + \frac{1}{2\pi}\int_{\partial(U_{\rho}\cap\cc_{I})} \,dp_I\,S_R^{-1}(p,T)\\
=& \frac{1}{2\pi}\int_{\partial(U_{\rho,1}\cap\cc_{I})}a(1+s)\,dp_I\,S_R^{-1}(p,T)
\end{align*}
and letting $\rho$ tend to infinity, we finally find
\begin{align}\label{AMSUIR}
a(\id + T)^{-1}E_{\infty} =\frac{1}{2\pi}\int_{\Gamma_{s,2}}a(1+s)\,dp_I\,S_R^{-1}(p,T).
\end{align}
Adding \eqref{ASSSA} and \eqref{AMSUIR}, we find that  \eqref{fProjRep} holds true for  $f \in \Eint[\sector{\omega}]$

Now let $f$ be an arbitrary function in $\meroInt[\sector{\omega}]_{T}$ that decays regularly at infinity and let $e$ be a regulariser for $f$. We can assume that $e$ decays regularly at infinity---otherwise, we can replace $e$ by $s\mapsto (1+s)^{-1}e(s)$, which is a regulariser for $f$ with this property. We expect that
\begin{equation}\label{Yuoya}
\begin{split}
f(T) E_{\infty} =& e^{-1}(T)(ef)(T)E_{\infty} = e^{-1}(T)\{(ef)(T)\}_{\infty} \\
\overset{(*)}{=}& e^{-1}(T)\{e(T)\}_{\infty}\{f(T)\}_{\infty} = e^{-1}(T)e(T)E_{\infty}\{f(T)\}_{\infty}\\
= & E_{\infty}\{f(T)\}_{\infty} \overset{(**)}{=} \{f(T)\}_{\infty}
\end{split}
\end{equation}
such that \eqref{fProjRep} holds true. The boundedness of $f(T)E_{\infty}$ follows then also from the boundedness of the integral $\{f(T)\}_{\infty}$. The second and the fourth of the above equalities follow from the above arguments since $ef$ and $e$ both belong to $\Eint[\sector{\varphi}]$ and decay regularly at infinity. The equalities marked with $(*)$ and $(**)$ however remain to be shown.

Let $\omega < \varphi_2<\varphi_1<\varphi$ be such that $e,f\in\Eint(\sector{\varphi})$ and let $r_1<r_2$ be such that $B_{r_1}(0)$ contains $\sigma_{S}(T)$ and any singularity of $f$. We set $U_{s} = \sector{\varphi_1}\setminus B_{r_1}(0) $ and $U_{p} = \sector{\varphi_2}\setminus B_{r_2}(0)$, where the subscripts $s$ and $p$  indicate again the respective variable of integration in the following computation. An application of the $S$-resolvent equation \eqref{SresEQ1}  shows then that
\begin{align*}
&\{e(T)\}_{\infty} \{f(T)\}_{\infty} = \frac{1}{2\pi}\int_{\partial(U_s\cap\cc_I)} e(s)\,ds_I\,S_{R}^{-1}(s,T) \frac{1}{2\pi} \int_{\partial(U_p\cap\cc_I)}S_L^{-1}(p,T)\,dp_I\,f(p)\\
=&\frac{1}{(2\pi )^{2}} \int_{\partial(U_s\cap\cc_I)}e(s)\,ds_I S_{R}^{-1}(s,T)\int_{\partial(U_p\cap\cc_I)}p(p^2-2s_0p + |s|^2)^{-1}\,dp_If(p)\\
& + \frac{1}{(2\pi )^{2}} \int_{\partial(U_s\cap\cc_I)}e(s)\,ds_I S_{R}^{-1}(s,T)\int_{\partial(U_p\cap\cc_I)}(p^2-2s_0p + |s|^2)^{-1}\,dp_If(p)\\
& + \frac{1}{(2\pi )^{2}} \int_{\partial(U_s\cap\cc_I)}e(s)\,ds_I \left(\overline{s}S_{L}^{-1}(p,T)-S_{L}^{-1}(p,T)p\right)\int_{\partial(U_p\cap\cc_I)}(p^2-2s_0p + |s|^2)^{-1}\,dp_If(p).
\end{align*}
Because of our choice of $U_s$ and $U_p$, the singularities of $p\mapsto (p^2-s_0p + |s|^{2})^{-1}$ lie outside $U_{p}$ for any $s\in\partial(U_s\cap\cc_I)$ such that $p\mapsto (p^2-2s_0p +|s|)^{-1}$ and $p\mapsto p(p^2-2s_0p +|s|)^{-1}$  are right slice hyperholomorphic on $\overline{U_p}$ for any such $s$. Since $f$ also decays regularly in $U_p$ at infinity, Cauchy's integral theorem implies that the first two of the above integrals equal zero. The fact that $e$ and $f$ decay polynomially at infinity allows us to exchange the order of integration in the third integral, such that 
\begin{gather*}
\{e(T)\}_{\infty} \{f(T)\}_{\infty} \\
= \frac{1}{(2\pi )^{2}} \int_{\partial(U_p\cap\cc_I)}\left[\int_{\partial(U_s\cap\cc_I)}e(s)\,ds_I \left(\overline{s}S_{L}^{-1}(p,T)-S_{L}^{-1}(p,T)p\right)(p^2-2s_0p + |s|^2)^{-1}\right]\,dp_If(p).
\end{gather*}
If $p\in\partial(U_p\cap\cc_I)$, then $p$ lies for sufficiently large $\rho$ in the bounded axially symmetric Cauchy domain $U_{s,\rho} = U_s\cap B_{\rho}(0)$. Since $f$ is an intrinsic function on $\overline{U_{s,\rho}}$, \Cref{AlpayLemma} implies
\begin{align*}
&\frac{1}{2\pi} \int_{\partial(U_s\cap\cc_I)}e(s)\,ds_I \left(\overline{s}S_{L}^{-1}(p,T)-S_{L}^{-1}(p,T)p\right)(p^2-2s_0p + |s|^2)^{-1} \\
=& \lim_{\rho \to \infty} \frac{1}{2\pi} \int_{\partial(U_s \cap B_{\rho}(0)\cap\cc_I)}e(s)\,ds_I \left(\overline{s}S_{L}^{-1}(p,T)-S_{L}^{-1}(p,T)p\right)(p^2-2s_0p + |s|^2)^{-1} \\
=& S_L^{-1}(p,T) e(p).
\end{align*}
Recalling the equivalence of right and left slice hyperholomorphic Cauchy integrals for intrinsic functions, c.f. \Cref{IntrinRemark}, we finally find that
\[
 \{e(T)\}_{\infty}\{f(T)\}_{\infty} = \frac{1}{2\pi} \int_{\partial(U_p\cap\cc_I)}S_L^{-1}(p,T)\,dp_Ie(p)f(p) = \{(ef)(T)\}_{\infty},
 \]
Hence, the identity $(*)$ in \eqref{Yuoya} is true.

Similar arguments show that also the equation $(**)$ holds true. We choose $0<R<r$ such that $B_{R}(0)$ contains $\sigma_S(T)$ and all singularities of $f(T)$ and we choose $\omega < \varphi'<\varphi$ such that $f\in\Eint(\sector{\varphi'})$ and set $U_p:= \sector{\varphi'}\setminus B_{r}(0)$. An application of the $S$-resolvent equation \eqref{SresEQ1}  shows that
\begin{align*}
&E_{\sigma_s(T)}\{f(T)\}_{\infty} = \frac{1}{2\pi}\int_{\partial(B_{R}(0)\cap\cc_I)} ds_I\,S_{R}^{-1}(s,T)\frac{1}{2\pi}\int_{\partial(U_p\cap\cc_I)} S_L^{-1}(p,T)\,dp_I\,f(p)\\
=& \frac{1}{(2\pi)^2} \int_{\partial(B_{R}(0)\cap\cc_I)}\, ds_I\, S_{R}^{-1}(s,T)\int_{\partial(U_p\cap\cc_I)}p(p^2-2s_0p+|s|^2)^{-1}\,dp_I\,f(p)\\
&- \frac{1}{(2\pi)^2} \int_{\partial(B_{R}(0)\cap\cc_I)} \, ds_I \,\overline{s}S_{R}^{-1}(s,T)\int_{\partial(U_p\cap\cc_I)}(p^2-2s_0p+|s|^2)^{-1}\,dp_I\,f(p) \\
& + \frac{1}{(2\pi)^2}\int_{\partial (B_{R}(0)\cap\cc_I)}\left[\int_{\partial(U_p\cap\cc_I)} \, ds_I \, \left( \overline{s}S_L^{-1}(p,T) - S_L^{-1}(p,T)p \right)(p^2-2s_0p + |s|^2)^{-1}\right]dp_I\,f(p).
\end{align*}
Again, the first two integrals vanish as a consequence of Cauchy's integral theorem because the poles of the function $p\mapsto (p^2-2s_0p +|s|^2)^{-1}$ lie outside of $\overline{U_p}$ for any $s \in \partial(B_{R}(0)\cap\cc_I)$ and $f$ decays regularly at infinity. Because of \eqref{SectCond} and the regular decay of $f$ at infinity, we can however apply Fubini's theorem  to exchange the order of integration in the third integral and find
\begin{align*}
&E_{\sigma_S(T)}\{f(T)\}_{\infty} = \\
=& \frac{1}{(2\pi)^2}\int_{\partial(U_p\cap\cc_I)}\left[\int_{\partial (B_{R}(0)\cap\cc_I)} \, ds_I \, (s^2-2p_0s + |p|^2)^{-1} \left( sS_L^{-1}(p,T) - S_L^{-1}(p,T)\overline{p} \right)\right]dp_I\,f(p).
\end{align*}
As the functions $s\mapsto (s^2-2p_0s + |p|^2)^{-1}$ and $s\mapsto (s^2-2p_0s + |p|^2)^{-1}s$ are right slice hyperholomorphic on $\overline{B_{R}(0)}$ for any $p\in\partial(U_p\cap\cc_I)$, also this integral vanishes due to Cauchy's integral theorem. Consequently, the identity $(**)$ in \eqref{Yuoya} holds also true as
\[ E_{\infty}\{f(T)\}_{\infty} = \{f(T)\}_{\infty} - E_{\sigma}\{f(T)\}_{\infty} = \{f(T)\}_{\infty}.\]

Finally, we point out that the above computations, which proved that $\{(ef)(T)\}_{\infty} = \{e(T)\}_{\infty}\{f(T)\}_{\infty}$ did not require that $e\in\Eint[\sector{\omega}]$. They also work if $e$ belongs to $\meroInt[\sector{\omega}]_T$ and decays regularly at infinity. Hence the same calculations show that \eqref{fProjProd}  holds true.

\end{proof}

\begin{theorem}\label{SpecIncSP}
Let $T\in\sectOP(\omega)$ and $s\in\{0,\infty\}$. If $f\in\meroInt[\sector{\omega}]_T$ has polynomial limit $c$  at $s$ and $s\in\sigma_{SX}(T)$, then $c\in\sigma_{SX}(f(T))$.
\end{theorem}
\begin{proof}
If $c \neq \infty$, then $c\in\rr$ because, as an intrinsic function, $f$ takes only real values on the real line. We can hence consider the function $f-c$ instead of $f$ because $\sigma_{SX}(f(T)) = \sigma_{SX}(f(T) -c\id) + c$ so that it is sufficient to consider the cases $c = 0$ or $c = \infty$.

Let us start with the case $c = 0$ and $s = \infty$. If $\infty\in\overline{\sigma_S(T)\setminus\{0\}}^{\hh_\infty}$,  then
\[
 0 \in \overline{f(\sigma_{S}(T)\setminus\{0\})}^{\hh_{\infty}} \subset \sigma_{SX}(f(T))
\]
because $f(\sigma_{S}(T)\setminus\{0\})\subset \sigma_{SX}(f(T))$ by \Cref{SpecInc} and the latter is a closed subset of~$\hh_{\infty}$. In case $\infty\notin\overline{\sigma_{S}(T)}^{\hh_{\infty}}$, we show that $0\notin\sigma_{SX}(f(T))$ implies that $T$ is bounded, i.e. that even $\infty\notin \sigma_{SX}(T)$. Let us hence assume that $\infty\notin\overline{\sigma_{S}(T)}^{\hh_{\infty}}$ and that $0\notin \sigma_{SX}(f(T))$. In this case, there exists $R>0$ such that $\sigma_{S}(T) $ is contained in the open ball  $B_{R}(0)$ of radius $R$ centered at zero. The integal
\[ E_{\sigma_{S}(T)} := \frac{1}{2\pi}\int_{\partial(B_{R}(0)\cap\cc_I)} ds_I\,S_R^{-1}(s,T)\]
defines then a bounded projection that commutes with $T$, namely the spectral projection associated with the spectral set $\sigma_S(T)\subset\sigma_{SX}(T)$ that is obtained from the $S$-functional calculus, cf. \Cref{SCProp}. The compatibility of the $S$-functional calculus with polynomials in $T$ moreover implies
\[
TE_{\sigma_{S}(T)} = (s\chi_{\sigma_{S}(T)})(T) = \frac{1}{2\pi}\int_{\partial(B_{R}(0)\cap\cc_I)} s\,ds_I\,S_R^{-1}(s,T)\in\boundOP(V),
\]
where $\chi_{\sigma_{S}(T)}$ denotes the characteristic function of an arbitrary axially symmetric bounded set that contains $\overline{B_{R}(0)}$. 

Set $E_{\infty} := \id -E_{\sigma_{S}(T)}$ and let $V_{\infty}:= E_{\infty}V$ be the range  of $E_{\infty}$. Since $T$ commutes with $E_{\infty}$, the operator $T_{{\infty}}:=T|_{V_{\infty}}$ is a closed operator on $V_{\infty}$ with domain $\dom(T_{\infty}) = \dom(T)\cap V_{\infty}$. Moreover, we conclude from \Cref{SCProp} that 
\[
\sigma_{SX}(T_{\infty})  = \sigma_{SX}(T)\setminus\sigma_{S}(T) \subset\{\infty\}
\]
and so in particular 
\begin{equation}\label{REM}
\sigma_{S}(T_{\infty}) = \sigma_{SX}(T_{\infty})\setminus\{\infty\} = \emptyset.
\end{equation}

Now observe that $f(T)$ commutes with $E_{\infty}$ because of \ref{MoProp1} in \Cref{MoProp}. Hence, $f(T)$ leaves $V_{\infty}$ invariant and $f(T)_{\infty} := f(T)|_{V_{\infty}}$ defines a closed operator on $V_{\infty}$ with domain $\dom(f(T)_{\infty}) = \dom(f(T))\cap V_{\infty}$. (Note that $f(T)_{\infty}$ intuitively corresponds to $f(T_{\infty})$. The $S$-functional calculus is  however only defined on two-sided Banach spaces. As $V_{\infty}$ is only a right-linear subspace of $V$ and hence not a two-sided Banach space, we can not use it to define the operator $f(T_{\infty})$, cf. the remark at the beginning of \Cref{SpecThmSection}!) Since $f(T)$ is invertible because we assumed $0\notin\sigma_{SX}f(T)$,  the operator $f(T)_{\infty}$ is invertible too and its inverse is $f(T)^{-1}|_{V_{\infty}}\in\boundOP(V_{\infty})$.

Our aim is now to show that $T_{\infty}$ is bounded. Any bounded operator on a nontrivial Banach space has non-empty $S$-spectrum and hence we can then conclude from \eqref{REM} that $V_{\infty} = \{0\}$. Since $f$ decays regularly at infinity, there exists $n\in\nn$ such that $sf^n(s)\in\meroInt[\omega]_T$ decays regularly at infinity too. Because of \Cref{MoProp}, the operators $Tf^n(T)$ and $(sf^n)(T)$ both commute with $E_{\infty}$. Hence, they leave $V_{\infty}$ invariant and we find, again because of   \Cref{MoProp}, that
\[ 
Tf^n(T)|_{V_{\infty}} \subset (sf^n)(T) |_{V_\infty} \in \boundOP(V_{\infty})
 \]
with
\begin{gather*}
\dom\left(Tf^n(T)|_{V_{\infty}}\right) = \dom\left(Tf^n(T)\right)\cap V_{\infty} \\
= \dom\left((sf^n)(T)\right)\cap\dom\left(f^n(T)\right)\cap V_{\infty} = \dom\left((sf^n)(T)|_{V_{\infty}}\right)\cap\dom\left(f^n(T)|_{V_{\infty}}\right).
\end{gather*}
But since $sf^n$ and $f^n$ both decay regularly at infinity in $\sector{\varphi}$, \Cref{fProj} implies that $f^n(T)|_{V_{\infty}}$ and $(sf^n)(T)|_{V_{\infty}}$ are both bounded linear operators on $V_{\infty}$. Hence their domain is the entire space $V_{\infty}$ and we find that 
\[ Tf^n(T)|_{V_{\infty}} = (sf^n)(T)|_{V_{\infty}}\in\boundOP(V_{\infty}).\]
Finally, observe that \Cref{fProj} also implies that $f^n(T)|_{V_\infty} = (f(T)|_{V_\infty})^n$. As $f(T)|_{V_{\infty}}$ has a bounded inverse on $V_{\infty}$, namely $f(T)^{-1}|_{V_{\infty}}$, we find that $T_{\infty}\in\boundOP(V_{\infty})$ too. As pointed out above, this implies $V_{\infty} = \{0\}$.

Altogether we find that $V = V_{\sigma_{S}(T)} := E_{\sigma_{S}(T)}V$ such that $T = T|_{V_{\sigma_{S}(T)}} \in \boundOP(V_{\sigma_{S}(T)}) = \boundOP(V)$ and in turn $\infty \notin\sigma_{SX}(T)$ if $0 = f(\infty) \notin \sigma_{SX}(f(T))$. 

Now let us consider the case that $s=0$ and $c = 0$, that is $f(0) = 0$. If $0\notin \sigma_{SX}(f(T))$, then $f(T)$ has a bounded inverse. Let $e$ be a regulariser for $f$ such that $ef\in\Eint[\sector{\omega}]$. Since $f(T) = e(T)^{-1}(ef)(T)$ is injective, the operator $(ef)(T)$ must be injective too. As the function $ef$ has polynomial limit $0$ at $0$, we conclude from \Cref{NonInjLem} that even $T$ is injective. If we define $\tilde{f}(p):=f(p^{-1})$, then $\tilde{f}$ has polynomial limit $0$ at $\infty$ and $\tilde{f}(T^{-1})$ is invertible as $\tilde{f}(T^{-1}) = f(T)$ by \Cref{InvCor}. Hence, $0 = \tilde{f}(\infty)\notin\sigma_{SX}(\tilde{f}(T^{-1}))$ and arguments as the ones above show that $\infty\notin\sigma_{SX}(T^{-1})$ such that $T^{-1}\in\boundOP(V)$. Thus, $T$ has a bounded inverse and in turn $0\notin \sigma_{SX}(T)$ if $0 = f(0) \notin \sigma_{SX}(f(T))$. 

Finally, let us consider the case $c = f(s)= \infty$ with $s= 0$ or $s = \infty$ and let us assume that $\infty\notin \sigma_{SX}(f(T))$, that is that $f(T)$ is bounded. If we choose $a\in\rr$ with $|a| > \|f(T)\|$, then $a\in\rho_S(f(T))$ and hence $a\id- f(T)$ has a bounded inverse. By \ref{MoProp3} in \Cref{MoProp}, the function $g(p): = (a- f(p))^{-1}$ belongs to $\meroInt[\sector{\omega}]_T$. Moreover, $g(T)$ is invertible and $g(T)$ has polynomial limit $0$ at $s$. As we have shown above, this implies $s\notin \sigma_{SX}(T)$, which concludes the proof.

\end{proof}

Combining \Cref{SpecInc} and \Cref{SpecIncSP}, we arrive at the following theorem.
\begin{theorem}\label{SpecIncAll}
Let $T\in\sectOP(\omega)$. If $f\in\meroInt[\sector{\omega}]_T$ has polynomial limits at $\sigma_{SX}(T)\cap \{0,\infty\}$, then
\[ f(\sigma_{SX}(T)) \subset \sigma_{SX}(f(T)).\] 
\end{theorem}

Let us now consider the inverse inclusion. We start with the following auxiliary lemma. 
\begin{lemma}\label{DDefLem}
Let $T\in\sectOP(\omega)$ and let $f\in\meroInt[\sector{\omega}]_T$ have finite polynomial limits  at $\{0,\infty\}\cap\sigma_{SX}(T)$ in $\sector{\varphi}$ for some $\varphi\in(\omega,\pi)$. Furthermore assume that all poles of $f$ are contained in $\rho_S(T)$.
\begin{enumerate}[label = (\roman*)]
\item \label{Case1} If $\{0,\infty\} \subset \sigma_{SX}(T)$, then $f(T)$ is defined by the $H^{\infty}$-functional calculus for sectorial operators.
\item \label{Case2} If $0\in\sigma_{SX}(T)$ but $\infty\notin\sigma_{SX}(T)$, then $f(T)$ is defined by the extended $H^{\infty}$-functional calculus for bounded sectorial operators.
\item \label{Case3} If $\infty\in\sigma_{SX}(T)$ but $0\notin\sigma_{SX}(T)$, then $f(T)$ is defined by the extended $H^{\infty}$-functional calculus for invertible sectorial operators.
\item \label{Case4} If $0,\infty\notin \sigma_{SX}(T)$, then $f(T)$ is defined by the $H^{\infty}$-functional calculus for bounded and invertible sectorial operators.
\end{enumerate}
In all of these cases $f(T)\in\boundOP(V)$.
\end{lemma}
\begin{proof}
Let us first consider the case \ref{Case1}, i.e. we assume that $\{0,\infty\}\subset\sigma_{SX}(T)$. Since $f$ has polynomial limits at $0$ and $\infty$  in $\sector{\omega}$, the functions $f$ has only finitely many poles $[s_1],\ldots,[s_n]$ in $\overline{\sector{\omega}}$. For suitably large $m_1\in\nn$, the function $f_1(p) = (1+p)^{-2m_1}\Q_{s_1}(p)^{m_1} f(p)$  has also polynomial limits at $0$ and $\infty$ and poles at $[s_2],\ldots,[s_n]$ but it does not have a pole at $[s_1]$. Moreover, if we set $r_1(p) = (1+p)^{-2m_1}\Q_{s_1}(p)^{m_1}$, then $r_1(T)$ is bounded and injective because $[s_1]\subset[\rho_S(T))]$. We can now repeat this argument and find inductively $m_2,\ldots, m_n$ such that, after setting $r_{\ell}(p) = (1+p)^{-2m_\ell} \Q_{s_\ell}(p)^{m_{\ell}}$ for $\ell = 2,\ldots,n$ and $r:= r_n\cdots r_1$, the function $\tilde{f} = r f$ belongs to $\meroInt[\sector{\omega}]_{T}$, has polynomial limits at $0$ and $\infty$ and does not have any poles in $\overline{\sector{\omega}}$. Hence it belongs to $\Eint[\sector{\omega}]$. Moreover, $r$ belongs to $\Eint[\sector{\omega}]$ too and since $r(T) = r_{n}(T)\cdots r_1(T)$ is the product of invertible operators, it is invertible itself. Hence $r$ regularises $f$ such that $f(T)$ is defined in terms of the $H^{\infty}$-functional calculus. Moreover, $f(T) = r(T)^{-1}\tilde{f}(T)$ is bounded as it is the product of two bounded operators.

Similar arguments show the other cases: in \cref{Case2} for example, the function $f$ has polynomial limit at $0$ but not at $\infty$, such that the poles of $f$ may accumulate at $\infty$. However, we integrate along the boundary of $\sector{\omega,0,R} = \sector{\omega}\cap B_{R}(0)$ in $\cc_{I}$ for sufficiently large $R$ when we define the $H^{\infty}$-functional calculus for bounded sectorial operators. Hence, only finitely many poles are contained in $\sector{\omega,0,R}$ and hence relevant. Therefore we can apply the above strategy again in order to show that $f$ is regularised by a rational intrinsic function and that $f(T)$ is hence defined and a bounded operator. Similar, we can argue for \Cref{Case3} and \Cref{Case4}, where the poles may of $f$ accumulate at $0$ resp. at $0$ and $\infty$, but only finitely many of them are relevant.

\end{proof}

\begin{proposition}\label{SpecCniAll}
Let $T\in\sectOP(\omega)$. If $f\in\meroInt[\sector{\omega}]_T$ has polynomial limits at $\sigma_S(T)\cap\{0,\infty\}$, then
\[ f(\sigma_{SX}(T)) \supset \sigma_{SX}(f(T)).\]
\end{proposition}
\begin{proof}
Let $s\in\hh$ with $s\notin f(\sigma_{SX}(T))$. The function $p\mapsto Q_{s}(f(p))^{-1}$ belongs then to $\meroInt[\sector{\omega}]_T$ and has finite polynomial limits at $\sigma_{SX}(T)\cap\{0,\infty\}$. Moreover the set of poles of $\Q_{s}(f(\cdot))$ as an element of $\meroInt[\sector{\omega}]$, which consists of those spheres $[p]$ in $\overline{\sector{\omega}}\setminus\{0\}$ for which $ f([p]) = [f(p)] = [s]$, is contained in the $S$-resolvent set of $T$ as we chose $s\notin f(\sigma_{SX}(T))$. From \Cref{DDefLem} we therefore deduce that $\Q_{s}(f(T))^{-1}$ is defined and belongs to $\boundOP(V)$. Hence $\Q_{s}(f(T))$ has a bounded inverse and so $s\in\sigma_{SX}(f(T))$.

If finally $s = \infty\notin f(\sigma_{SX}(T))$, then the poles of $f$ are contained in the $S$-resolvent set of $T$. Hence, \Cref{DDefLem} implies that $f(T)$ is a bounded operator and in turn $s = \infty\notin \sigma_{SX}(f(T))$.

\end{proof}
Combining \Cref{SpecIncAll} and \Cref{SpecCniAll}, we obtain the following spectral mapping theorem
\begin{theorem}[Spectral Mapping Theorem]\label{SpecMap}
Let $T\in\sectOP(\omega)$ and let $f\in\meroInt[\sector{\omega}]_T$ have polynomial limits at $\{0,\infty\}\cap\sigma_{SX}(T)$.  Then
\[ f(\sigma_{SX}(T)) = \sigma_{SX}(f(T)).\]
\end{theorem}

\section{Fractional powers via the $H^{\infty}$-functional calculus}\label{FracPowSect}
We apply now the $H^{\infty}$-functional calculus introduced in \Cref{HInfty} in order to define fractional powers of sectorial operators. Again we follow the strategy used in \cite{Haase} to obtain our results.
\subsection{Fractional powers with positive real part}

Let $T\in\sectOP(\omega)$ and let $\alpha\in(0,+\infty)$. The function $s \mapsto s^{\alpha}$ does then obviously belong to $\meroInt[\sector{\omega}]_T$ and we can define $T^{\alpha}$ using the quaternionic $H^{\infty}$-functional calculus introduced in \Cref{HInfty}. Precisely, we can choose $n\in\nn$ with $n>{\alpha}$ and find
\begin{equation}\label{TADef}
T^{\alpha} := s^{\alpha}(T) = (\id+T)^n \left(s^{\alpha}(1+s)^{-n}\right)(T),
\end{equation}
where $\left(s^{\alpha}(1+s)^{-n}\right)(T)$ is defined via a slice hyperholomorphic Cauchy integral as in \eqref{SectIntL} or~\eqref{SectIntR}.
\begin{definition}
Let $T\in\sectOP(\omega)$ and $\alpha >0$. We call the operator defined in \eqref{TADef} the fractional power with exponent $\alpha$ of $T$.
\end{definition}

The following properties are immediate consequences of the properties of the $H^{\infty}$-functional calculus.

\begin{lemma}\label{TaProp1} 
Let $T\in\sectOP(\omega)$ and let $\alpha \in(0,+\infty)$. 
\begin{enumerate}[label=(\roman*)]
\item If $T$ is injective, then $\left(T^{-1}\right)^{\alpha} = \left(T^{\alpha}\right)^{-1}$. Thus $0\in\rho_S(T)$ if and only if $0\in\rho_S(T^{\alpha})$.
\item Any bounded operator that commutes with $T$ commutes also with $T^{\alpha}$. 
\item \label{TaSpecMap} The spectral mapping theorem holds, namely
\[ \sigma_{S}(T^{\alpha}) = \{s^{\alpha}: s\in\sigma_{S}(T)\}. \]
\end{enumerate}
\end{lemma}

Another important property is analyticity in the exponent. Observe that, although in the classical case the mapping $\alpha\mapsto T^{\alpha}$ is holomorphic in $\alpha$, we cannot expect slice hyperholomorphicity here because the fractional powers are only defined for real exponents, cf. the comments after \Cref{ScalLogDef}.

\begin{proposition}\label{AnalyticityProp}
 If $T\in\sectOP(\omega)$, then the following statements hold true.
\begin{enumerate}[label = (\roman*)]
\item \label{Analy1} If $T$ is bounded, then $T^{\alpha}$ is bounded too and the mapping $\Lambda:\alpha\to T^{\alpha}$ is analytic on $(0,+\infty)$ and has a left and a right slice hyperholomorphic extension to $\hh^+ = \{s\in\hh: \Re(s) >0\}$. In particular, for any $\alpha_0\in(0,+\infty)$ the Taylor series expansion of $f_{\alpha}$ at $\alpha_0$ converges on $(0,2\alpha_0)$.
\item \label{Analy2} If $n\in\nn$ and $0<\alpha < n$, then $\dom(T^n)\subset \dom(T^{\alpha})$. The mapping $\Lambda_{v}: \alpha\mapsto A^{\alpha}v$ is analytic on $(0,n)$ for each $v\in\dom(T^n)$ and the power series expansion of $\Lambda_{v}$ at $\alpha_0\in(0,n)$ converges on $(-r + \alpha, \alpha + r)$ with $r_{\alpha_0}=\min\{\alpha_0, n-\alpha_{0}\}$. Hence, $\Lambda_v$ has a left and a right slice hyperholomorphic expansion to the set $\bigcup_{\alpha_0\in(0,n)} B_{r_{\alpha_0}}(\alpha_0)$.
\end{enumerate}
\end{proposition}
\begin{proof}
Let us first show \ref{Analy2}. If $n\in\nn$ and $\alpha\in(0,n)$, then $T^{\alpha} = (\id + T)^{n} \left(s^{\alpha}(1+s)^{-n}\right)(T)$. If $v\in\dom\left(T^n\right)$, then $T^n$ and $\left(s^{\alpha}(1+s)^{-n}\right)(T)$ commute because of \ref{MoProp1} in \Cref{MoProp}  such that $T^{\alpha}v = \left(s^{\alpha}(1+s)^{-n}\right)(T) (\id + T)^n v$ and hence $v\in\dom(T^{\alpha})$. 

Let now $\alpha_0 \in (0,n)$ and set $r = \min\{\alpha_0, n- \alpha_0\}$. The Taylor series expansion of $\alpha\mapsto s^{\alpha}$ at $\alpha_0$ is $s^{\alpha} = \sum_{n=0}^{+\infty} \frac{(\alpha -\alpha_0)^k}{k!} s^{\alpha_0}\log(s)^k$ and converges on $(0,2\alpha_0)$. If $\varepsilon \in (0,1)$ and $\alpha\in(0,n)$ with $|\alpha - \alpha_0| < (1-\varepsilon) r$, then we have after choosing $\varphi\in(\omega,\pi)$ that
\begin{align}
\notag T^{\alpha}v =& \left(s^{\alpha}(1+s)^{-n}\right)(T) (\id+T)^nv \\
 \notag = & \frac{1}{2\pi}\int_{\partial(\sector{\varphi}\cap\cc_I)} s^{\alpha}(1+s)^{-n}\,ds_I\,S_R^{-1}(s,T) (\id+T)^{n}v\\
\label{Waxuka}=& \frac{1}{2\pi}\int_{\partial(\sector{\varphi}\cap\cc_I)} \sum_{k=0}^{+\infty}\frac{(\alpha - \alpha_0)^k}{k!}s^{\alpha_0}\log(s)^k(1+s)^{-n}\,ds_I\,S_R^{-1}(s,T) (\id+T)^{n}v.
\end{align}
We want to apply the theorem of dominated convergence in order to exchange the integral and the series. Using \eqref{SectCond} we find that $\tilde{\Psi}(s) =  M \left\|(\id+T)^{n}v\right\| \Psi(s)$ with
\[
\Psi(s) := \sum_{k=0}^{+\infty} \frac{|\alpha - \alpha_0|^k}{k!}|s|^{\alpha_0-1}\frac{|\log(s)|^k}{|1+s|^{n}}
\]
  is a dominating function for the integrand in \eqref{Waxuka}. In order to show the integrability of $\Psi(s)$ along $\partial(\sector{\varphi}\cap\cc_I)$, we choose $ C_{est}> 1$ such that $(1-\varepsilon)C_{est} < 1$ and $0<t_0< 1 $ and $1 < t_1$ such that
\begin{equation}\label{logEst}
|\ln(t) + I \theta| < C_{est} |\ln(t)|\quad \forall t\in(0,t_0]\cup [t_1,\infty).
\end{equation}
 We then have
\begin{equation*}
\begin{split}
&\frac{1}{2}\int_{\partial(\sector{\varphi}\cap\cc_I)} \psi(s)\,d|s_I| =\int_{0}^{+\infty} \sum_{k=0}^{+\infty} \frac{|\alpha - \alpha_0|^k}{k!} t^{\alpha_0-1} \frac{|\ln(t) + I\varphi|^k }{|1 + te^{I\varphi}|^{n}} \,dt \\
\leq &\sum_{k=0}^{+\infty} \frac{|\alpha - \alpha_0|^k}{k!}\left( C_0 C_{est}^k \int_{0}^{t_0}  t^{\alpha_0-1} (-\ln(t))^k\,dt\right.\\
 &+ C_0 C_1^k\left.\int_{t_0}^{t_1}  t^{\alpha_0-1} \,dt+  C_2C_{est}^k \int_{t_1}^{+\infty}  t^{\alpha_0-(n+1)} \ln(t)^k \,dt\right) ,
\end{split}
\end{equation*}
with the constants
\[ C_0 := \max_{t\in[0,t_1]}\frac{1}{|1 + te^{I\varphi}|^n},\quad C_1:= \max_{t\in[t_0,t_1]}|\ln(t) + I \varphi|\]
and a constant $C_2>0$ such that
\[ \frac{1}{|1 + te^{I\varphi}|} < \frac{C_2}{t} \quad \forall t\in [t_1,+\infty). \]
Since
\[ \int_{0}^{t_1} t^{\alpha_0-1}(-\ln(t))^k\,dt \leq \int_{-\infty}^{0} e^{\alpha_0\xi}(-\xi)^k \,d\xi = \frac{k!}{\alpha_0^{k+1}}\]
and similarly
\[ \int_{t_1}^{+\infty} t^{\alpha_0 - (n+1)}\ln(t)^k\,dt \leq \int_{1}^{+\infty} e^{-(n-\alpha_0)\xi}\xi^k\,d\xi = \frac{k!}{(n-\alpha_0)^{k+1}},\]
we can further estimate 
\begin{align*}
&\frac{1}{2}\int_{\partial(\sector{\varphi}\cap\cc_i)}\Psi(s)\,d|s_I| \leq \\
\leq &\sum_{k=0}^{+\infty} \frac{|\alpha - \alpha_0|^k}{k!} \left( C_0C_{est}^k \frac{k!}{\alpha_0^{k+1}} +  C_0C_1^k\left(\frac{t_1^{\alpha_0}-t_0^{\alpha_0}}{\alpha_0}\right) + C_2 C_{est}^k \frac{k!}{(n-\alpha_0)^{k+1}}\right).
\end{align*}
As $|\alpha - \alpha_0| <(1-\varepsilon) r  = (1-\varepsilon)\min\{\alpha_0,n-\alpha_0\}$, we finally find 
\begin{align*}
&\frac{1}{2}\int_{\partial(\sector{\varphi}\cap\cc_i)}\Psi(s)\,d|s_I|\\
\leq &\frac{C_0}{\alpha_0} \sum_{k=0}^{+\infty} \ \left( (1-\varepsilon)C_{est}\right)^k
+   \frac{C_0(t_1^{\alpha_0}-t_0^{\alpha_0})}{\alpha_0} \sum_{k=0}^{+\infty} \frac{\left(C_1|\alpha - \alpha_0|\right)^k}{k!} 
+ \frac{C_2}{\alpha_0} \sum_{k=0}^{+\infty} \left((1-\varepsilon)C_{est}\right)^k.
\end{align*}
Since $(1-\varepsilon)C_1<1$ these series are finite and hence $\tilde{\Psi}$ is an integrable majorant of the integrand in \eqref{Waxuka}. We can thus exchange the series and the integral in \eqref{Waxuka} such that
\[ T^{\alpha}v = \sum_{k=0}^{+\infty} \frac{(\alpha - \alpha_0)^{k}}{k!} \frac{1}{2\pi}\int_{\partial(\sector{\varphi}\cap\cc_I)} s^{\alpha_0}\log(s)^k(1+s)^{-n}\,ds_I\,S_{R}^{-1}(s,T)(\id
T)^{n}v\]
and that as we shoed above this series converges uniformly for $|\alpha - \alpha_0| < (1-\varepsilon)r $. Since $\varepsilon \in (0,1)$ was arbitrary, we obtain the statement.

If $T$ is bounded, then \eqref{TADef} is the composition of two bounded operators and hence bounded. With arguments as the ones used above one can show that the power series expansion of $\Lambda$ at $\alpha_0$ converges in $\boundOP(V)$ on $(0,2\alpha_0)$. If we  write the variable $(\alpha - \alpha_0)$ in the power series expansion on the left or on the right side of the coefficients and extend $\alpha$ to a quaternionic variable, we find that $\Lambda$ has a left resp. a right slice hyperholomorphic extension to $B_{\alpha_0}(\alpha_0)$. Finally, any point in $\hh^+$ is contained in a ball of this form, and hence we find that we can extend $\Lambda$ to a left or to a right slice hyperholomorphic function on all of $\hh^+$.

\end{proof}

We show now that the usual computational rules that we expect to hold for fractional powers of an operator hold true with our approach.

\begin{proposition}[First Law of Exponents]Let $T\in\sectOP(\omega)$.
 For all $\alpha,\beta>0$ the identity $T^{\alpha+\beta} = T^{\alpha}T^{\beta}$ holds. In particular $\dom(T^{\gamma} )\subset \dom(T^{\alpha})$ for $0<\alpha<\gamma$.
\end{proposition}
\begin{proof}
Because of \ref{MoProp2} in \Cref{MoProp}, we have $T^{\alpha}T^{\beta}\subset T^{\alpha+\beta}$ with $\dom\left(T^{\alpha}T^{\beta}\right) = \dom\left(T^{\alpha+\beta}\right)\cap\dom\left(T^{\beta}\right)$. We choose $n\in\nn$ with $\alpha,\beta < n$ and define the bounded operators $\Lambda_{\alpha} := \left(s^{\alpha}(1+s)^{-n}\right)(T)$ and $\Lambda_{\beta} := \left(s^{\beta}(1+s)^{-n}\right)(T)$.  If now $v\in\dom\left(T^{\alpha+\beta}\right)$, then \ref{MoProp2} in \Cref{MoProp} implies
\begin{align*}
 T^{\alpha+\beta} v = & (\id+T)^{2n} (\id + T)^{-2n}T^{\alpha+\beta} v = (\id+T)^{2n}T^{\alpha+\beta}  (\id + T)^{-2n}v \\
  = & (\id + T)^{2n}\left(s^{\alpha+\beta}(1+s)^{-2n}\right)(T) v=(\id +T)^{2n} \Lambda_{\alpha}\Lambda_{\beta} v
 \end{align*}
and hence $\Lambda_{\alpha}\Lambda_{\beta} v \in\dom\left((\id + T^{2n})\right) = \dom\left(T^{2n}\right)$. Sinc $\left( s^{n-\alpha} (1+s)^{-n}\right)(T)$ commutes with $T^{2n}$ because of \ref{MoProp1} in \Cref{MoProp}, we thus find 
\[
 T^{n}(\id+T)^{-2n}\Lambda_{\beta} v = \left(s^{n+\beta}(1+s)^{-3n}\right)(T) v = \left( s^{n-\alpha} (1+s)^{-n}\right)(T) \Lambda_{\alpha}\Lambda_{\beta} v \in\dom\left(T^{2n}\right).
\]
Since $T$ and $T(\id+T)^{-1}$ commute, we have $T (\id+T)^{-1} Tv = T^2 (\id + T)v$ and hence $v\in\dom(T)$ implies $T(\id + T)^{-1} v\in\dom(T)$. If on the other hand $T(\id + T)^{-1}v\in\dom(T)$, then the identity
\[
T(\id+T)^{-1}v = v  - (\id + T)^{-1}v
\]
implies $v\in\dom(T)$ and hence $v\in\dom(T)$ if and only if $T(\id+T)^{-1}v\in\dom(T)$. By induction, we find that $v\in\dom(T^m)$ if and only if $T^n(\id+T)^{-n}v\in\dom(T^m)$. We thus conclude that $(\id + T)^{-n}\Lambda_{\beta}v\in\dom(T^{2n})$ which in turn implies $\Lambda_{\beta}v\in\dom(T^{n}) )= \dom((\id + T)^n)$. Thus, $T^{\beta}v = (\id+T)^n \Lambda_{\beta}v$ is defined, such that in turn $v\in\dom\left(T^\beta\right)$ for any $v\in\dom(T^{\alpha+\beta})$. We conclude that
\[
\dom\big(T^{\alpha}T^{\beta}\big) = \dom\big(T^{\alpha+\beta}\big)\cap\dom\big(T^{\beta}\big) = \dom\big(T^{\alpha +\beta}\big)
\]
and in turn $T^{\alpha}T^{\beta} = T^{\alpha + \beta}$.

\end{proof}

\begin{proposition}[Scaling Property]\label{ScalProp}
Let $T\in\sectOP(\omega)$ and let $\Lambda = [\delta_1,\delta_2]\subset (0,\pi/\omega)$ be a compact interval. Then the family $(T^{\alpha})_{\alpha\in\Lambda}$ is uniformly sectorial of angle $\delta_2\omega$. In particular, for every $\alpha\in(0,\pi/\omega_T)$, the operator $T^{\alpha}$ is sectorial with angle $\omega_{T^{\alpha}} = \alpha\omega_T$.
\end{proposition}
\begin{proof}
The second statement obviously follows from the first by choosing $\lambda = [\alpha,\alpha]$. Because of \ref{TaSpecMap} in \Cref{TaProp1}, we know that $\sigma_{S}(T^{\alpha}) = (\sigma_{S}(T))^{\alpha}\subset \overline{\sector{\alpha\omega}}\subset \overline{\sector{\delta_2\omega}}$ for $\alpha\in\Lambda$. What remains to show are the uniform estimates \eqref{SectCond} for the $S$-resolvents. 

We choose $\varphi\in(\delta_2\omega,\pi)$. In order to show that $\|S_L^{-1}(s,T^{\alpha})s\|$ is uniformly bounded for $s\notin\overline{\sector{\varphi}}$ and $\alpha\in\Lambda$, we define for $\alpha\in\Lambda$ and $s\notin\sector{\varphi}$ the function
\begin{equation}\label{NJe}
\begin{split}
 \Psi_{s,\alpha}(p) = &S_L^{-1}\left(s,p^{\alpha}\right)s + S_L^{-1}\left(-|s|^{\frac{1}{\alpha}},p\right)|s|^{\frac{1}{\alpha}}\\
  = & \Q_{s}\left(p^{\alpha}\right)^{-1}\left(|s|^{\frac{1}{\alpha}} + p\right)^{-1}\left(p(\overline{s}-p^{\alpha})s + p^{\alpha}(\overline{s}-p^{\alpha})|s|^{\frac{1}{\alpha}}\right).
 \end{split}
 \end{equation}
This function  belongs to $\lholZI[\sector{\omega}]$: as $s\notin \sector{\varphi}$, it is left slice hyperholomorphic on $\sector{\theta_0} $ with $\theta_0 := \min\{\alpha^{-1}\varphi,\pi\} > \omega$. The first line in \eqref{NJe} implies that $\Psi_{s,\alpha}$ has polynomial limit $0$ at infinity because $S_L^{-1}\left(s,p^{\alpha}\right)$ and $S_L^{-1}\left(|s|^{1/\alpha},p\right)$ have polynomial limit $0$ at infinity and the second line in \eqref{NJe} implies that $\Psi_{s,\alpha}$ has polynomial limit $0$ at $0$ because $\Q_{s}(p^{\alpha})^{-1}$ and $\left(|s|^{1/\alpha}-p\right)^{-1}$ are bounded for $p$ sufficiently close to $0$.  Since the function $S_L^{-1}\left(|s|^{1/\alpha},p\right) = \left(|s|^{1/\alpha} - p\right)^{-1}$ belongs to $\EL[\sector{\omega}]$, we find that also $S_L^{-1}(s,p^{\alpha})s =  \Psi_{s,\alpha}(p) + S_L^{-1}(|s|^{1/\alpha},T)|s|^{1/\alpha}$ belongs to $\EL[\sector{\omega}]$ and that
\[
 S_L^{-1}(s,T^{\alpha})s  = S_L^{-1}\left(|s|^{\frac{1}{\alpha}},T\right)|s|^{\frac{1}{\alpha}} + \Psi_{s,\alpha}(T).
 \]
The function $\Psi_{s,\alpha}$ satisfies the scaling property $\Psi_{t^{\alpha}s,\alpha}(tp) = \Psi_{s,\alpha}(p)$ such that $\Psi_{\frac{s}{|s|},\alpha}\big(|s|^{-\frac{1}{\alpha}},p\big) = \Psi_{s,\alpha}(p)$. If we choose $\theta \in(\omega, \min\{\pi,\delta_2^{-1}\varphi\})$ and $I = I_s$, we therefore find that
\begin{align*}
\|S_{L}^{-1}(s,T^{\alpha})s\| \leq & C_{\theta',T} + \left \| \Psi_{s/|s|,\alpha}\left(|s|^{-\frac{1}{\alpha}}T\right)\right\|\\
\leq & C_{\theta',T} + \frac{1}{2\pi} \left\| \int_{\partial(\sector{\theta}\cap\cc_I)} S_L^{-1}(p,T)\,dp_I\, \Psi_{s/|s|,\alpha}\left(|s|^{-\frac{1}{\alpha}}p\right) \right\|\\
\leq & C_{\theta',T} + \frac{C_{\theta',T}}{2\pi} \int_{\partial(\sector{\theta}\cap\cc_I)} |p|^{-1} d|p|\, \left|\Psi_{s/|s|,\alpha}\left(p\right) \right|,
\end{align*}
where $C_{\theta',T}$ is the respective constant in \eqref{SectCond} for some $\theta'\in(\omega,\theta)$, which is independent of $s$ and $\alpha \in \Lambda$. Hence, if we are able to show that
\begin{equation}\label{CxY}
 \sup\left\{ \int_{\partial(\sector{\theta}\cap\cc_{I_s})} |p|^{-1} d|p|\, \left|\Psi_{s,\alpha}\left(p\right) \right| : |s| = 1, s\notin\sector{\varphi}, \alpha\in\Lambda \right\} < \infty,
 \end{equation}
then we are done. Since we integrate along a path in the complex plane $\cc_{I_s}$, we find that $p$ and $s$ commute and $\Psi_{s,\alpha}\left(p\right)$ simplifies to $\Psi_{s,\alpha}\left(p\right) = \left(s-p^{\alpha}\right)^{-1}\left(p+|s|^{1/\alpha}\right)^{-1}\left(ps + |s|^{1/\alpha}p^{\alpha}\right)$. As $|s| = 1$, we can therefore estimate
\begin{equation*}
 \left|\Psi_{s,\alpha}\left(p\right)\right| \leq \frac{|p|^{1-\varepsilon}}{\left|s-p^{\alpha}\right|}\frac{|p|^{\varepsilon}}{|1+p|} + \frac{|p|^{\alpha - \varepsilon}}{\left| s-p^{\alpha}\right|}\frac{|p|^{\varepsilon}}{|1+ p|} \leq K \frac{|p|^{\varepsilon}}{|1+p|}
 \end{equation*}
with $\varepsilon \in (0,\delta_1)$, because $|p|^{1-\varepsilon}/\left|s-p^{\alpha}\right|$ and $|p|^{\alpha - \varepsilon}/\left| s-p^{\alpha}\right|$ are uniformly bounded by some constant $K>0$ for our parameters $s$, $\alpha$ and $p$. Thus we have an estimate for the integrand in \eqref{CxY} that is independent of the parameters such that \eqref{CxY} is actually true.

With analogous arguments using the right slice hyperholomorphic version of the $S$-functional calculus for sectorial operators, we can show that also $\|s S_R^{-1}(s,T^{\alpha} )\|$ is uniformly bounded for $s\notin\overline{\sector{\varphi}}$ and $\alpha\in\Lambda$. Since $\varphi\in(\delta_2\omega,\pi)$ was arbitrary, the proof is finished.

\end{proof}

As immediate consequences of \Cref{ScalProp} and the composition rule \Cref{CompRul}, we obtain the following two results.
\begin{proposition}
Let $T\in\sectOP(\omega)$ for some $\omega\in(0,\pi)$ and let $\alpha\in(0,\pi/\omega)$ and $\varphi \in(\omega,\pi/\alpha)$. If  $f\in\lholZI(\sector{\alpha\varphi})$ (or $f\in\meroL[\sector{\alpha\omega}]_{T^{\alpha}})$, then the function $p\mapsto f\left(p^{\alpha}\right)$ belongs to $\lholZI(\sector{\varphi})$ (resp. $\meroL[\sector{\omega}]_{T}$) and
\[f\left(T^{\alpha}\right) = \left(f\left(p^{\alpha}\right)\right)(T).\]
\end{proposition}
\begin{corollary}[Second Law of Exponents]\label{SecExp}
Let $T\in\sectOP(\omega)$ with $\omega\in(0,\pi)$ and let $\alpha\in(0,\pi/\omega)$. For all $\beta >0$, we have
\[ \left(T^{\alpha}\right)^{\beta} =  T^{\alpha\beta}.\]

\end{corollary}

\begin{corollary}
Let $T\in\sectOP(\omega)$ and $\gamma>0$. For any $v\in\dom(T^{\gamma})$, the mapping $\Lambda_{v}:\alpha\mapsto T^{\alpha}v$ defined on $(0,\gamma)$ is analytic in $\alpha$. Moreover, the power series expansion of $\Lambda_{v}$ at any point $\alpha_0\in(0,\gamma)$ converges on $(-r + \alpha_0,\alpha_0 + r)$ with  $ r  = \min\{\gamma-\alpha_0,\alpha_0\}$.
\end{corollary}
\begin{proof}
Let $n>\gamma$ and set $A:= T^{\gamma/n}$. Because of \Cref{SecExp}, we have $T^{\alpha} = A^{\alpha n/\gamma}$. If $v\in\dom(T^{\gamma})$, then $v\in\dom(A^n)$ and the mapping $\Upsilon(\beta):= A^{\beta}v$ is analytic on $(0,n)$ by \Cref{AnalyticityProp}. The radius of convergence of its power series expansion at $\beta_0\in(0,n)$ is greater than or equal to $r' = \min\{\beta_0, n-\beta_0\}$. Hence, $\Lambda_v(\alpha) = \Upsilon(n\alpha/\gamma)$ is also an analytic function and the radius of convergence of its power series expansion at any point $\alpha_0\in(0,\gamma)$ is greater than or equal to $\min\{\alpha_0,\gamma-\alpha_0\}$, which is exactly what we wanted to show.

\end{proof}

We conclude this section with the generalization of the famous Balakrishnan representation of fractional powers and some of its consequences. This formula was introduced in \cite{Balakrishnan} as one of the first approaches to define fractional powers of sectorial operators.
\begin{theorem}[Balakrishnan Representation]
Let $T\in\sectOP(\omega)$. For $0<\alpha <1$, we have
\begin{equation}\label{Bala1}
 T^{\alpha}v = \frac{\sin(\alpha\pi)}{\pi}\int_{0}^{+\infty} t^{\alpha-1}(t+T)^{-1}Tv\, dt,\qquad\forall v\in\dom(T).
 \end{equation}
More general, for $0<\alpha <n\leq m$, we have
\begin{equation}\label{Bala2}
 T^{\alpha}v = \frac{\Gamma(m)}{\Gamma(\alpha)\Gamma(m-\alpha)}\int_{0}^{+\infty}t^{\alpha-1}[T(t+T)^{-1}]^mv\,dt,\qquad \forall v \in\dom(T^n).
 \end{equation}
\end{theorem}
\begin{proof}
We first show \eqref{Bala1} and hence assume that $\alpha\in(0,1)$. For $v\in\dom(T)$, we have because of \ref{MoProp2} in \Cref{MoProp} and with arbitrary $\varphi\in(\omega,\pi)$ and $\varepsilon >0$ that
\begin{align*}
 T^{\alpha} v = & \left(p^{\alpha}(p+\varepsilon)^{-1}\right)(T) (T+\varepsilon\id)v \\
 =& \left(p^{\alpha}(p+\varepsilon)^{-1}\right)(T)Tv + \varepsilon \left(p^{\alpha}(p+\varepsilon)^{-1}(1+p)^{-1}\right)(T)(\id+T)v\\
 =& \frac{1}{2\pi}\int_{\partial(\sector{\varphi}\cap\cc_I)}  s^{\alpha-1} s(s+\varepsilon)^{-1} \, ds_I \, S_R^{-1}(s,T)Tv \\
 &+  \frac{1}{2\pi} \int_{\partial(\sector{\varphi}\cap\cc_I)} s^{\alpha}\varepsilon(s+\varepsilon)^{-1}(1+s)^{-1} \,ds_I\, S_R^{-1}(s,T)(\id+T)v.
\end{align*}
Now observe that there exists a positive constant $K<+\infty$ such that 
\[
\left|\varepsilon(s+\varepsilon)^{-1}\right| \leq \frac{K}{|s|} \qquad \forall \varepsilon>0,s\in\partial(\sector{\varphi}\cap\cc_I).
\] 
 Together with  the estimate \eqref{SectCond}, this implies that the integrand in the second integral is bounded for all $\varepsilon>0$ by the functions $ s\mapsto KC_{\varphi,T}|s|^{\alpha-1}(|1+s|)^{-1}\|(\id+T)v\|$, which is integrable along $\partial(\sector{\varphi}\cap\cc_{I})$  because of the assumption $\alpha\in(0,1)$. Hence, we can apply Lebesgue's theorem in order to exchange the integral with the limit and find that the second integral vanishes as $\varepsilon$ tends to $0$. In the first integral on the other hand, we find that
\begin{equation}\label{TJuL}
\begin{split}
S_{R}^{-1}(s,T)Tv =& (\overline{s}\id -T)\Q_{s}(T)^{-1}Tv = \overline{s}T\Q_{s}(T)^{-1}v - T^2\Q_{s}(T)^{-1}v\\
 =& \overline{s}T\Q_{s}(T)^{-1} - \Q_{s}(T)\Q_{s}(T)^{-1}v + \left(-2s_0T + |s|^{-1}\id\right)\Q_{s}(T)^{-1}v\\
  =&  - v + \overline{s}T\Q_{s}(T)^{-1}v - \overline{s}T\Q_{s}(T)^{-1}v + s(\overline{s}\id - T)\Q_{s}(T)^{-1}v\\
  =& -v + s S_{R}^{-1}(s,T)v.
\end{split}
\end{equation}
Hence, the function $s\mapsto S_{R}^{-1}(s,T)Tv$ for $s\in\partial(\sector{\varphi}\cap\cc_I)$ is bounded at $0$ because of \eqref{SectCond}. Since it decays as $|s|^{-1}$ as $s\to\infty$ and since the function $s\mapsto s(s+\varepsilon)^{-1}$ is uniformly bounded in $\varepsilon$ on $\partial(\sector{\varphi}\cap\cc_I)$, we can apply Lebesgue's theorem also in the first integral in order to take the limit as $\varepsilon\to 0$ and obtain
\[ T^{\alpha} v  = \frac{1}{2\pi} \int_{\partial(\sector{\varphi}\cap\cc_I)} s^{\alpha-1} \, ds_I\, S_{R}^{-1}(s,T) Tv.\]
Choosing the standard parametrisation of the path of integration, we thus find
\begin{align*}
 T^{\alpha} v  =& \frac{1}{2\pi} \int_{-\infty}^{0} \left(-t e ^{I\varphi}\right)^{\alpha-1} e^{I\varphi}I S_{R}^{-1}\left(te^{I\varphi},T\right) Tv \,dt\\
 & + \frac{1}{2\pi} \int_{0}^{+\infty} \left(t e ^{-I\varphi}\right)^{\alpha-1} e^{-I\varphi}(-I) S_{R}^{-1}\left(te^{-I\varphi},T\right) Tv \,dt.
 \end{align*}
Once more \eqref{SectCond} and the fact that $S_{R}^{-1}(s,T)Tv$ is bounded at $0$ allow us to apply Lebesgue's theorem in order to take the limit as $\varphi$ tends to $\pi$. We finally find after a change of variables in the first integral that
\begin{align*}
 T^{\alpha} v  =& \frac{1}{2\pi} \int_{0}^{+\infty} t^{\alpha-1} \big(-e ^{I\pi\alpha}\big) (- I) S_{R}^{-1}\left(-te^{I\pi},T\right) Tv \,dt\\
 & + \frac{1}{2\pi} \int_{0}^{+\infty} t^{\alpha-1} e ^{-I\pi\alpha}(-I) S_{R}^{-1}\left(te^{-I\pi},T\right) Tv \,dt\\
 =& - \frac{\sin(\alpha\pi)}{\pi}\int_{0}^{+\infty} t^{\alpha-1} S_R^{-1}(-t,T)Tv\,dt,
 \end{align*}
 which equals \eqref{Bala1} as $S_{R}^{-1}(-t,T) = (-t\id-T)^{-1} = - (t\id+T)^{-1}$  for $t\in\rr$.
 
 Let us now prove \eqref{Bala2} and let us for now assume that $n-1 < \alpha < n $ and $n=m$. For $v\in\dom(T^n)$, we then have
 \begin{align*}
 T^{\alpha} v = T^{\alpha-(n-1)}T^{n-1}v = \frac{\sin((\alpha - n +1)\pi)}{\pi} \int_{0}^{+\infty}t^{\alpha-n}(t\id+T)^{-1}T^nv\,dt.
 \end{align*}
 Integrating $n-1$ times by parts, we find
 \begin{align}
\notag T^{\alpha}v &= \frac{(n-1)!\sin((\alpha-n+1)\pi)}{\pi(\alpha-n+1)\cdots(\alpha-1)}\int_{0}^{+\infty}t^{\alpha-1}(t\id+T)^{-n}T^nv\,dt\\
\label{GHas} &= \frac{\Gamma(n)}{\Gamma(\alpha)\Gamma(n-\alpha)}\int_{0}^{+\infty}t^{\alpha-1}(t\id+T)^{-n}T^nv\,dt,
 \end{align}
 where the second identity follows from the identities $\sin(z\pi)/\pi = 1/(\Gamma(z)\Gamma(1-z))$ and $z\Gamma(z) = \Gamma(z+1)$ for the gamma function. Hence \eqref{Bala2} holds true if $ n-1 < \alpha < n=m$.
 
  Now observe that, because of \eqref{SectCond} and because $(t\id+T)^{-n}T^n = \left((t\id+T)^{-1}T\right)v$ is bounded near $0$ due to \eqref{TJuL}, the integral \eqref{GHas} defines a real analytic function in $\alpha$ on the entire interval $(0,n)$. From \Cref{AnalyticityProp} and the identity principle for real analytic functions, we conclude that \eqref{Bala2} holds also if $0<\alpha < n = m$.
  
  Finally, let us show by induction on $m$ that \eqref{Bala2} holds true for any  $m\geq n$. For $m=n$ we have just shown this identity, so let us assume that  it holds true for some $m\geq n$. We introduce the notation 
  \[
   c_m := \frac{\Gamma(m)}{\Gamma(m-\alpha)\Gamma(\alpha)}\quad \text{and}\quad I_m:= \int_{0}^{+\infty}t^{\alpha-1}\left[T(t\id+T)^{-1}\right]^{m}v\,dt
   \]
  so that $T^{\alpha}v = c_m I_m$. We want to show that $T^{\alpha}v = c_{m+1}I_{m+1}$. By integration by parts, we deduce
  \begin{align*}
   I_m = &\left(\left. \frac{t^{\alpha}}{\alpha} \left[T(t\id+T)^{-1}\right]^mv\right)\right|_{0}^{+\infty} + \frac{m}{\alpha}\int_{0}^{+\infty}t^{\alpha}\left[T(t\id+T)^{-1}\right]^{m}(t\id+T)^{-1}v\,dt\\
   =&  \frac{m}{\alpha}\int_{0}^{+\infty}t^{\alpha}\left[T(t\id+T)^{-1}\right]^{m}(t\id+T)^{-1}v\,dt\\
   =&  \frac{m}{\alpha}\int_{0}^{+\infty}t^{\alpha-1}\left(\left[T(t\id+T)^{-1}\right]^{m}v - \left[T(t\id+T)^{-1}\right]^{m+1}v\right) \,dt\\
   =& \frac{m}{\alpha}(I_m - I_{m+1}).
   \end{align*}
Hence $I_m = \frac{m}{m-\alpha} I_{m+1}$ and so
   \[
   T^{\alpha}v = c_mI_m = c_m \frac{m}{m-\alpha}I_{m+1} = c_{m+1}I_{m+1}.
   \]
The induction is complete. 
  
\end{proof}

\subsection{Fractional powers with negative real part}

If $\alpha<0$ the fractional power $p^{\alpha}$ has polynomial limit infinity at 0 in any sector $\sector{\varphi}$ with $\varphi > \pi$. Because of \Cref{NonInjLem} it does therefore not belong to $\meroL[\sector{\omega}]_T$ if $T$ is not injective. If on the other hand $T$ is injective, then it is regularisable by some power of $p\left(1+p\right)^{-2}$ such that $p^{\alpha}\in\meroL[\sector{\omega}]_T$. We can thus define $T^{\alpha}$ for injective sectorial operators via the $H^{\infty}$-functional calculus.
\begin{definition}
Let $T\in\sectOP(\omega)$ be injective. For $\alpha\in\rr$ we call the operator $T^{\alpha} : = (p^{\alpha})(T)$ the fractional power with exponent $\alpha$ of $T$
\end{definition}
 The properties of the fractional powers of $T$ in this case are again analogue to the complex case, cf. \cite{Haase}. We state the most important properties for the sake of completeness, but we omit the proofs since they are either immediate consequences of the preceding results or can be shown with exactly the same arguments as in the complex case without makeing use of any quaternionic techniques. For the special case that the operator is not only injective, but does even have a bounded inverse, properties of the fractional powers $T^{\alpha}$ with negative real part were already studied in \cite{FJTAMS}.
\begin{proposition}\label{TaInjProp}
Let $T\in\sectOP(\omega)$ be injective and let $\alpha,\beta\in\rr$.
\begin{enumerate}[label=(\roman*)]
\item The operator $T^{\alpha}$ is injective and $(T^{\alpha})^{-1} = T^{-\alpha} = \left(T^{-1}\right)^{\alpha}$.
\item We have $T^{\alpha}T^{\beta}\subset T^{\alpha+\beta}$ with $\dom\left(T^{\alpha}T^{\beta}\right) = \dom\left(T^{\beta}\right)\cap\dom\left(T^{\alpha+\beta}\right)$.
\item If $\overline{\dom(T)} = V = \overline{\ran(T)}$, then $T^{\alpha+\beta} = \overline{T^{\alpha}T^{\beta}}$.
\item If $0<\alpha < 1$, then
\[ T^{-\alpha}v = \frac{\sin(\alpha\pi)}{\pi}\int_{0}^{+\infty} t^{-\alpha}(t\id+T)^{-1}v\,dt\quad \forall v\in\ran(A).\]
\item If $\alpha\in\rr$ with $|\alpha|<\pi/\omega$, then $T^{\alpha}\in\sectOP(|\alpha|\omega)$ and for all $\beta\in\rr$
\[ 
(T^{\alpha})^{\beta} = \left(T^{\alpha\beta}\right).
\]
\item If $ 0 < \alpha_1,\alpha_2$, then $\dom(T^{\alpha_2})\cap\ran(T^{\alpha_1})\subset \dom(T^{\alpha})$ for each $\alpha\in(-\alpha_1,\alpha_2)$, the mapping $\alpha\mapsto T^{\alpha}v$ is analytic on $(-\alpha_1,\alpha_2)$ for any $v\in\dom(T^{-\alpha_2})\cap\ran(T^{\alpha_1})$. 
\end{enumerate}
\end{proposition}

\begin{proposition}[Komatsu Representation]
Let $T\in\sectOP(\omega)$ be injective. For $v\in\dom(A)\cap\ran(A)$ and $\alpha\in(-1,1)$, one has
\begin{align*}
T^{\alpha}v =  \frac{\sin(\alpha\pi)}{\pi}&\left[ \frac{1}{\alpha}v - \frac{1}{1+\alpha}T^{-1}v \right.\\
& + \left. \int_{0}^{1}t^{\alpha+1}(t\id + T)^{-1}T^{-1}v \,dt + \int_{1}^{+\infty}t^{\alpha-1}(t\id+T)^{-1}Tv\,dt\right]\\
= \frac{\sin(\alpha\pi)}{\pi}&\left[ \frac{1}{\alpha}v + \int_{0}^{1}t^{-\alpha}(\id + tT)^{-1}Tv\,dt - \int_{0}^{1}t^{\alpha}\left(\id + tT^{-1}\right)^{-1}T^{-1}v\,dt\right].
\end{align*}

\end{proposition}

\section{Spectral theory of the nabla operator}
Our goal in this section is to define fractional powers of the gradient operator using the quaternionic theory introduced above. The gradient of a function $v: \rr^3 \to \rr$ is the vector-valued function
\[ 
\nabla v (x) = \begin{pmatrix}\partial_{x_1} v(x)\\ \partial_{x_2} v(x) \\ \partial _{x_3} v(x) \end{pmatrix} , \qquad  x = \begin{pmatrix} x_1 \\ x_2 \\  x_3\end{pmatrix}.  
\]
If we identify $\rr$ with the set of real quaternions and $\rr^3$ with the set of purely imaginary quaternions, this corresponds to the quaternionic nabla operator
\[
 \nabla = \partial_{x_1} e_{1} + \partial_{x_2} e_{2} + \partial_{x_3} e_{3}.
 \]
In the following we shall often denote the standard basis of the quaternions by $\I:= e_1$, $\J:=e_2$ and $\K:=e_3  = \I\J = - \J\I$. This suggests a relation with the complex theory, which we shall use intensively. With this notation, we have
\[ \nabla = \partial_{x_1} \I + \partial_{x_2} \J + \partial_{x_3} \K.\]

We study the properties of a quaternionic nabla operator on the space $L^2(\rr^3,\hh)$ of all square-integrable quaternion-valued functions on $\rr^3$, which is a quaternionic right Hilbert space when endowed with the scalar product 
\[
 \langle w, v \rangle = \int_{\rr^3} \overline{w(x)} v(x)\, dx. 
\]

Let $v\in L^2(\rr^3,\hh)$ and write $v(x) = v_{1}(x) + v_{2}(x)\J$ with two $\cc_{\I}$-valued functions $v_{1}$ and $v_{2}$. As $|v(x)|^2 = |v_{1}(x)|^2 + |v_{2}(x)|^2$, we have 
\begin{equation}\label{IScal}
\| v \|_{L^2(\rr^3,\hh)}^2 = \|v_1\|_{L^2(\rr^3,\cc_{\I})}^2  + \|v_2\|_{L_{2}(\rr,^3,\cc_{\I})}^2,
\end{equation}
 where $L^2(\rr^3,\hh)$ denotes the complex Hilbert space over $\cc_{\I}$ of all square-integrable $\cc_{\I}$-valued functions on $\rr^3$. Hence, $v\in L^2(\rr^3,\hh)$ if and only if $v_1,v_2\in L^2(\hh,\cc_{\I})$.

\begin{theorem}\label{NablaSpec}
The $S$-spectrum of $\nabla$ as an operator on $L^2(\rr^3,\hh)$ is
\[ \sigma_{S}(\nabla) = \rr.\]
\end{theorem}
\begin{proof}
Let us consider $L^2(\rr^3,\hh)$ as a Hilbert space over $\cc_{\I}$ by restricting the right scalar multiplication to $\cc_{\I}$ and setting $ \langle w,v\rangle_{\I} := \{\langle w, v\rangle_{L^2(\rr^3,\hh)} \}_{\I}$, where $\{ \cdot \}_{\I}$ denotes the $\cc_{\I}$-part of quaternion, i.e. if $a = a_1 + a_2\J = a_1 + \J\overline{a_2}$ with $a_1,a_2\in\cc_{\I}$, then $\{a\}_{\I} := a_1$. If we write $v,w\in L^2(\rr^3,\hh)$ as $v = v_{1} + \J v_{2}$ and $w = w_{1} + \J w_{2}$ with $v_1,v_2,w_1,w_2\in L^2(\rr^3,\cc_{\I})$, then
\begin{align*}
\langle w, v \rangle_{L^2(\rr^3,\hh)} := &\int_{\rr^3} \overline{(w_1(x) + \J w_2(x))}(v_1(x) + \J v_2(x))\, dx \\
=&\int_{\rr^3} \overline{w_1(x)}v_1(x)\, dx + \int_{\rr^3} \overline{w_2(x)}(-\J)v_1(x)\, dx \\
&+ \int_{\rr^3} \overline{w_1(x) }\J v_2(x)\, dx + \int_{\rr^3} \overline{ w_2(x)}(-\J^2)v_2(x)\, dx \displaybreak[2]\\
=&\int_{\rr^3} \overline{w_1(x)}v_1(x)\, dx + \int_{\rr^3} \overline{ w_2(x)}v_2(x)\, dx\\
&+\J\left( - \int_{\rr^3}w_2(x)v_1(x)\, dx + \int_{\rr^3} w_1(x)  v_2(x)\, dx \right).
\end{align*}
Therefore we have
\[
\langle w, v \rangle_{\I} := \langle w_1, v_1 \rangle_{L^2(\rr^3,\cc_{\I})}  +   \langle w_2, v_2 \rangle_{L^2(\rr^3,\cc_{\I})} 
\]
and hence $ L^2(\rr^3,\hh)$ considered as a $\cc_{\I}$-complex Hilbert space with the scalar product $\langle\cdot,\cdot\rangle_{\I}$ equals $ L^2(\rr^3,\cc_{\I})\oplus L^2(\rr^3,\cc_{\I})$. Moreover, because of \eqref{IScal}, the quaternionic scalar product $\langle\cdot,\cdot\rangle$ and the $\cc_{\I}$-complex scalar product $\langle\cdot,\cdot\rangle_{\I}$ induce the same norm on $L^2(\rr^3,\hh)$.
Applying the nabla operator to $v = v_1 + \J v_2$ we find
\begin{align*}
\nabla v(x) =& (\I\partial_{x_1} + \J \partial_{x_2} + \K\partial_{x,3}) (v_1(x) + \J v_2(x)) \\
= & \I\partial_{x_1}v_1(x) + \J \partial_{x_2}v_1(x) + \K\partial_{x_3}v_1(x) + \I\partial_{x_1}\J v_2(x) + \J \partial_{x_2}\J v_2(x) + \K\partial_{x_3} \J v_2(x)\\
= & \I\partial_{x_1}v_1(x) -\partial_{x_2} v_2(x) - \I\partial_{x_3} v_2(x) 
+ \J \left( - \I\partial_{x_1} v_2(x)  + \partial_{x_2}v_1(x)   -  \I\partial_{x_3}v_1(x)\right).
\end{align*}
Writing this in terms of the components $L^2(\rr^3,\hh)\cong L^2(\rr^3,\cc_{\I})\oplus L^2(\rr^3,\cc_{\I})$, we find
\[
\nabla \begin{pmatrix} v_{1}(x) \\ v_{2}(x)\end{pmatrix} = \begin{pmatrix}  \I\partial_{x_1}v_1(x) -\partial_{x_2} v_2(x) - \I\partial_{x_3} v_2(x) \\  - \I\partial_{x_1} v_2(x)  + \partial_{x_2}v_1(x)   -  \I\partial_{x_3}v_1(x)
\end{pmatrix}.
\]
If we apply the Fourier transform on $L^2(\rr^3,\cc_{\I})$ componentwise, this turns into
\[
\widehat{\nabla} \begin{pmatrix} \widehat{v_1}(x)\\ \widehat{v_2}(x) \end{pmatrix} = 
\begin{pmatrix}  -\xi_1 &  - \I\xi_2  +\xi_{3}  \\  \I\xi_2   + \xi_3  &  \xi_1 
\end{pmatrix} \begin{pmatrix} \widehat{v_1}(\xi) \\ \widehat{v_2}(\xi)\end{pmatrix}.
\]
Hence, in the Fourier space, the Nabla operator corresponds to the multiplication operator $M_{G}: \widehat{v}\mapsto G\widehat{v}$ on $\widehat{V}:=L^2(\rr^3,\cc_{\I})\oplus L^2(\rr^3,\cc_{\I})$ that is generated by the matrix valued function
\begin{equation}\label{GMat}
G(\xi) := \begin{pmatrix}  -\xi_1 &  - \I\xi_2  +\xi_{3}  \\  \I\xi_2   + \xi_3  &  \xi_1  \end{pmatrix}.
\end{equation}
For $s\in\cc_{\I}$, we find 
\[
 s \id_{\widehat{V}} - G(\xi) =  \begin{pmatrix}  s + \xi_1 &   \I\xi_2  -\xi_{3}  \\  -\I\xi_2 -   \xi_3  & s -  \xi_1  \end{pmatrix}.
\]
For $s\in\cc_{\I}$, the inverse of $s\id_{\widehat{V}} - M_G$ is hence given  by multiplication operator $M_{(s\id-G)^{-1}}$ determined the matrix-valued function
\[
 (s \id_{\widehat{V}} - G(\xi))^{-1} =  \frac{1}{s^2 - \xi_1^2  - \xi_2^2 - \xi_3^2} \begin{pmatrix}  s - \xi_1 &   - \I\xi_2  + \xi_{3}  \\  \I\xi_2 +   \xi_3  & s +  \xi_1  \end{pmatrix}.
\]
This operator is bounded if and only if the function  $\xi \mapsto (s\id - G(\xi))^{-1}$ is bounded on $\rr^3$, i.e. if and only $s\notin\rr$. Hence, $\sigma(M_G) = \rr$.

The componentwise Fourier transform $\Psi$ is a unitary $\cc_{\I}$-linear operator from $L^2(\rr^3,\hh)\cong\linebreak[2] L^2(\rr^3,\cc_{\I})\oplus L^2(\rr^2,\cc_{\I})$ to $\widehat{V}$ under which $\nabla$ corresponds to $M_G$, that is  $\nabla = \Psi^{-1}M_G \Psi$. The spectrum $\sigma_{\cc_{\I}}(\nabla)$ of $\nabla$ considered as a $\cc_{\I}$-linear operator on $L^2(\rr^3,\hh)$ therefore equals $\sigma_{\cc_{\I}}(\nabla) = \sigma(M_G) = \rr$.  By \cite[Theorem~3.2]{SpecOP}, we however have $\sigma_{\cc_{\I}}(\nabla)  = \sigma_{S}(\nabla) \cap\cc_{\I}$ and so $\sigma_{S}(\nabla) = \rr$.

\end{proof}

The above result shows that the gradient does not belong to the class of sectorial operators as $(-\infty, 0)\not\subset\rho_{S}(T)$ so that the theory developed in \Cref{FracPowSect} is not directly applicable. Even worse, we cannot find any other slice hyperholomorphic functional calculus that allows us to define fractional powers $\nabla^{\alpha}$ of $\nabla$ because the scalar function $s^{\alpha}$ is not slice hyperholomorphic on $(-\infty,0]$ and hence not slice hyperholomorphic on $\sigma_{S}(\nabla)$.

In the following we shall nevertheless introduce a method that allows us to deduce the fractional heat equation based on quaternionic techniques and the considerations made above. We shall however need to introduce the $\mathcal{SC}$-functional calculus, the version of the $S$-functional calculus for operators with commuting components introduced in \cite{Where}.

We consider a two-sided quaternionic Banach space $V$. By \cite{Ng}, this space is of the form $V = V_{\rr} \otimes \hh$, i.e. any $v\in V$ is of the form $v = v_{0} + \sum_{\ell=1}^3 v_{\ell}e_{\ell}$, where the components $v_{\ell}$ are elements of the real Banach space $V_{\rr}:=\{v\in V: av = va\ \forall a \in\hh\}$. An operator $A$ on the space $V_{\rr}$ can then be extended to a quaternionic right linear operator on $V$ by componentwise application, i.e. $A v = Av_{0} + \sum_{\ell=1}^3 Av_{\ell} e_{\ell}$. As an operator on $V$, the operator $A$ commutes with any quaternionic scalar.

Let now $T\in\boundOP(V)$. We can then write  $T = T_{0} + \sum_{\ell=0}^{3}T_{\ell}e_{\ell}$ with components $T_{\ell}\in\boundOP(V_{\rr})$, $\ell = 0,\ldots,3$. Operators on $V_{\rr}$ are therefore also called scalar operators as the do not contain any imaginary units. Let us now set $\overline{T} := T_0 - \sum_{\ell=1}^3 T_{ \ell}e_{\ell}$. If the components $T_{\ell}$ commute mutually, then $T + \overline{T} = 2T_0$ and $T\overline{T} = \overline{T}T = \sum_{\ell = 0}^3 T_{\ell}^2$ are  scalar operators  and
\begin{equation}\label{AnsZwoDrei}
\begin{split}
&\left(\overline{s}\id - T\right)\left(s^2\id - 2sT_0 + T\overline{T}\right)\\
=& |s|^2s\id - Ts^2 - 2|s|^2T_0 + 2TT_0s + \overline{s}T\overline{T} - T^2\overline{T}\\
=& |s|^2s\id - Ts^2 - |s|^2T - |s|^2\overline{T} + T^2s +T\overline{T}s + \overline{s}T\overline{T} - T^2\overline{T}\\
=& |s|^2\left(s\id-\overline{T}\right)  -2s_0T \left(s\id - \overline{T} \right)+ T^2\left( s\id  - \overline{T}\right)\\
=& \left(T^2 - 2s_0T + |s|^2 \id\right)\left(s\id-\overline{T}\right),
\end{split}
\end{equation}
where we used the identities $2s_0 = s + \overline{s}$ and $|s|^2 = s\overline{s}$. Recalling $\Q_{s}(T) = T^2 - 2s_0 T + |s|^2\id$ and setting 
\[
\Q_{c,s}(T) = s^2\id - 2sT_0 + T\overline{T},
\]
  this reads as
\[ \left(\overline{s}-T\right)\Q_{c,s}(T) = \Q_{s}(T)\left(s\id - \overline{T}\right).\]
If $\Q_{s}(T)$ and $\Q_{c,s}(T)$ are both invertible, we therefore have
\begin{equation}\label{SResNabla}
 S_{L}^{-1}(s,T) = \Q_{s}(T)^{-1}(\overline{s}\id-T) =  \left(s\id-\overline{T}\right)\Q_{c,s}(T)^{-1}.
\end{equation}
Similarly, one also shows
\[
S_{R}^{-1}(s,T) = \Q_{c,s}(T)^{-1}(s\id - \overline{T}).
\]
Indeed, for a bounded operator $T$ with commuting components, one can show that  $\Q_{s}(T)$ has a bounded inverse if and only if $\Q_{c,s}(T)$ has one \cite[Proposition~4.6]{Where} . Hence, one has
\[ \rho_S(T) = \left\{s\in\hh: \Q_{c,s}(T)^{-1} = \left(s^2\id - 2sT_0 + T\overline{T}\right)^{-1} \in\boundOP(V)\right\} .\]
As the next theorem shows, this also true for an unbounded operator with commuting components.

\begin{theorem}
Let $T = T_0 + \sum_{\ell =1}^3 T_{\ell} e_{\ell}\in\closOP(V)$ with $T_{\ell}\in\closOP(V_{\rr})$ such that the components $T_{\ell}, \ell = 0,\ldots,3$ commute mutually, i.e. $T_{\ell}T_{\kappa}v = T_{\kappa}T_{\ell}v$ for all $v \in\dom(T^2) = \bigcap_{r,s=0,\ldots,3}\dom(T_rT_s)$ and all $\ell,\kappa\in\{0,\ldots,3\}$.  If we set $\Q_{c,s}(T) = s^2\id - 2s T_0  + T\overline{T}$ and $e_0 = 1$, then
\begin{equation}\label{SSpecCommut}
\rho_S(T) = \left\{ s\in \hh : \Q_{c,s}(T)^{-1}\in\boundOP(V) \right\}
\end{equation}
and
\begin{equation}\label{ZwoaDreiVia}
S_L^{-1}(s,T) = (s\id- \overline{T})\Q_{c,s}(T)\quad\text{and}\quad S_R^{-1}(s,T) =  \Q_{c,s}(T)^{-1}s - \sum_{\ell=1}^3 T_{\ell}Q_{c,s}(T)^{-1}e_{\ell}.
\end{equation}
\end{theorem}
\begin{proof}
McIntosh showed in \cite[Theorem~3.3]{McI1} that an operator $A = A_0 + \sum_{\ell=1}^3A_{\ell}e_{\ell}\in\boundOP(V)$ with commuting components is invertible if and only if $A\overline{A} = \overline{A}A= \sum_{\ell=0}^3 A_{\ell}^2$ is invertible. This holds true also for an unbounded operator with commuting components: if $A\overline{A}$ is invertible, then  $(A\overline{A})^{-1} = \big(\sum_{\ell=0}^3 A_{\ell}^2\big)^{-1}$ commutes with each of the components $A_{\ell}$  and it  also commutes with the imaginary units $e_{\ell}$ because it is a scalar operator. Hence, it commutes with $A$ and so $A^{-1} = \overline{A} (A\overline{A})^{-1}$ as
\[
\left(\overline{A}(A\overline{A})	^{-1}\right) Av = \overline{A}A (\overline{A}A)^{-1}v \qquad \forall v\in\dom(A) 
\]
and
\[
A \left(\overline{A} (A\overline{A})^{-1}\right) v = (A \overline{A} )(A\overline{A})^{-1} v = v \quad\forall v\in V.
\]
Consequently, the invertibility of $A\overline{A}$ implies the invertibility of $A$. 

If on the other hand $A$ is invertible and $A^{-1} = B_0 + \sum_{\kappa=1}^3 B_{\kappa}e_{\kappa}\in\boundOP(V)$, then
\begin{align*}
\id|_{\dom (A)} =& A^{-1} A = \left(B_0 + \sum_{\kappa=1}^3 B_{\kappa}e_{\kappa}\right)\left( A_0 + \sum_{\ell =1}^3A_{\ell}e_{\ell} \right)\\
=& B_0A_0 - \sum_{\ell=1}^3 B_{\ell}A_{\ell} +  (B_2A_3-B_3A_2)e_{1}\\
&+ (B_{3}A_{1} - B_{1}A_{3})e_{2} + (B_{1}A_{2}-B_{2}A_{1})e_{3},
\end{align*}
from which we conclude that
\[
\id|_{\dom(A)} = B_0A_0 - \sum_{\ell=1}^3 B_{\ell}A_{\ell}\qquad\text{and}\qquad B_{\ell}A_{\kappa}-B_{\kappa}A_{\ell} = 0\quad 1\leq \ell < \kappa \leq 3.
\]
Therefore
\begin{align*}
\overline{B}\,\overline{A} = & \left(B_0 - \sum_{\ell=1}^3B_{\ell}e_{\ell}\right)\left(A_{0} - \sum_{\ell=1}^{3}A_{\ell}e_{\ell}\right)\\
=& B_0A_0 - \sum_{\ell=1}^3 B_{\ell}A_{\ell} +  (B_2A_3-B_3A_2)e_{1}\\
&+ (B_{3}A_{1} - B_{1}A_{3})e_{2} + (B_{1}A_{2}-B_{2}A_{1})e_{3} = \id|_{\dom(A)}.
\end{align*}
Similarly, we see that $AB = \id$ also implies $\overline{A}\,\overline{B} = \id$. Hence, the invertibility of $A$ implies the invertibility of $\overline{A}$ and $\overline{A}^{-1} = \overline{A^{-1}}$. Thus, if $A$ is invertible, we have $(A\overline{A})^{-1}= \overline{A}^{-1}A^{-1}\in\boundOP(V)$. Altogether, we find that $A$ is invertible if and only if $A\overline{A} = \overline{A}A$ is invertible.

Let us now turn our attention back to the operator $T\in\closOP(V)$ with commuting components. Since $T$ and $\overline{T}$ commute, we have $\overline{\Q_{s}(T)} = \Q_{s}(\overline{T})$ and $\overline{\Q_{c,s}(T) } = \Q_{c,\overline{s}}(T)$ and so
\begin{align*}
\Q_{c,s}(T) \overline{\Q_{c,s}(T)}=& (s^2\id - 2 sT_0 + T\overline{T})(\overline{s}^2\id - 2\overline{s}T_0 + T\overline{T})\\
=& |s|^4\id - 2s|s|^2 T_0 + s^2 T\overline{T}\\
& - 2 |s|^2 T_0 \overline{s} + 4 |s|^2 T_0^2 - 2s T_0 T\overline{T} \\
&+ \overline{s}^2 T\overline{T} - 2\overline{s}T_0T\overline{T} + (T\overline{T})^2\\
=& |s|^4\id - 2s_0|s|^2T - 2s_0 |s|^2\overline{T} + 2\Re(s^2) T\overline{T}\\
&  + 4 |s|^2 T_0^2 - 2s_0 T^2\overline{T} - 2s_0 T\overline{T}^2  + T^2\overline{T}^2,
\end{align*}
where we used in the last identity that $2s_0 = s+\overline{s}$, that $|s|^2 = s\overline{s}$, and that $2T_0 = T + \overline{T}$. As 
\[
2\Re(s^2)T\overline{T}  =  2s_0^2T\overline{T} - 2s_1^2 T\overline{T}
\]
and
\[
4|s|^2 T_0^2 = |s|^2(T+\overline{T})^2 = |s|^2T^2+ 2 s_0^2T\overline{T} + s_1^2 T\overline{T} + |s|^2\overline{T}^2
\]
we further find 
\begin{align*}
\Q_{c,s}(T) \overline{\Q_{c,s}(T)}=& |s|^2(|s|^2\id  - 2s_0T + T^2)\\
& - 2s_0 \overline{T}(|s|^2\id -2s_0T + T^2) \\
& +\overline{T}^2(|s|^2\id- 2s_0 T  + T^2) = \Q_{s}(T)\overline{\Q_{s}(T)}.
\end{align*}
By the above arguments, we hence have
\begin{gather*}
\Q_{c,s}(T)^{-1}\in\boundOP(V)\Longleftrightarrow \left(\Q_{c,s}(T)\overline{Q_{c,s}(T)}\right)^{-1}\in\boundOP(V) \\
\Longleftrightarrow \left(\Q_{s}(T)\overline{Q_{s}(T)}\right)^{-1}\in\boundOP(V) \Longleftrightarrow \Q_{s}(T)^{-1}\in\boundOP(V)
\end{gather*}
and hence \eqref{SSpecCommut} holds true.

Computations as in \eqref{AnsZwoDrei} show that 
\[
S_L^{-1}(s,T)v = (s\id-\overline{T})\Q_{c,s}(T)^{-1}v \quad \text{and}\quad S_{R}^{-1}(s,T)v = \Q_{c,s}(T)^{-1}(s\id - \overline{T})v
\]
 for $v\in\dom(T)$. These operators can be extended to continuous operators on the entire space $V$ by writing them as in \eqref{ZwoaDreiVia}, cf. \Cref{RkResExtension}. This yields the identity in~\eqref{ZwoaDreiVia}.

\end{proof}

Let us now turn back to the nabla operator on the quaternionic right Hilbert space $L^2(\rr^3,\hh)$. If $I\in\SS$ is an arbitrary imaginary unit and $J\in\SS$ with $J\perp I$, then any $v\in L^2(\rr^3,\hh)$ can be written as $v = v_{1} + v_{2}J$ with components $v_1,v_2$ in $L^2(\rr^3,\cc_{I})$, i.e. $L^2(\rr^3,\hh)  = L^2(\rr^3,\cc_{I}) \oplus L^2(\rr^3,\cc_{I})J$. Contrary to the decomposition $v = v_{1} + J v_{1}$, which we used in the proof of \Cref{NablaSpec} with $I = \I$ and $J=\J$, this decomposition is not compatible with the $\cc_{I}$-right vector space structure of $L^2(\rr^3,\hh)$ as $va = v_{1}a + v_{2}\overline{a} J$ for any $a\in\cc_{I}$. Howeover, this identification has a different advantage: any closed $\cc_{I}$-linear operator $A: \dom(A)\subset L^2(\rr^3,\cc_{\I}) \to L^2(\rr^3,\cc_{\I})$ extends to a closed $\hh$-linear operator on $L^2(\rr^3,\hh)$ with domain $\dom(A)\oplus\dom(A)J$, namely to $A(v_1+v_2J) := A(v_{1}) + A(v_{2}) J$. Moreover, if $A$ is bounded, then its extension to $L^2(\rr^3,\hh)$ has the same norm as $A$. We shall denote an operator on $L^2(\rr^3,\cc_{I})$ and its extension to $L^2(\rr^3,\hh)=L^2(\rr^3,\cc_{I})\oplus L^2(\rr^3,\cc_{I})J$ via componentwise application by the same symbol. This will not cause any confusion as it will be clear from the context to which we refer.
\begin{theorem}
Let $\Delta$ be the Laplace operator on $L^2(\hh,\cc_{I})$ and let $R_{z}(-\Delta)$ be the resolvent of $-\Delta$ at $z\in\cc_{I}$. We have
\begin{equation}\label{DeltaSpec}
 \sigma_{S}(\nabla)^2 = \left\{ s^2\in\hh: s\in\sigma_{S}(T) \right\} = \sigma(-\Delta)
\end{equation}
 and 
\begin{equation}\label{QsRs}
\Q_{c,s}(\nabla)^{-1} = R_{s^2}(-\Delta)\qquad \forall s\in\cc_{I}\setminus \rr.
\end{equation}
\end{theorem}
\begin{proof}
Since the components of $\nabla$ commute and $e_{\kappa}e_{\ell} = -e_{\ell}e_{\kappa}$ for $1\leq \kappa,\ell\leq 3$ with $\kappa\neq\ell$, we  have
\[
\nabla^2 = \sum_{\ell,\kappa=1}^{3} \partial_{x_\ell}\partial_{x_\kappa} e_{\ell}e_{\kappa} = \sum_{\ell=1}^3 -\partial_{x_\ell}^2 + \sum_{1\leq \ell <\kappa\leq 3} \left( \partial_{x_\ell}\partial_{x_{\kappa}} - \partial_{x_{\kappa}}\partial_{x_\ell}\right)e_{\ell}e_{\kappa} = \sum_{\ell=1}^3-\partial_{x_{\ell}}^2 = - \Delta.
\]
As $\nabla_0 = 0$, we have $\overline{\nabla} = -\nabla$ and in turn
\[
\Q_{c,s}(\nabla) = s^2\id - 2s\nabla_0 + \nabla\overline{\nabla} = s^2\id - \nabla^2 = s^2\id - (-\Delta)
\]
Hence, $\Q_{c,s}(\nabla)$ is invertible if and only if $s^2\id - (-\Delta)$ is invertible. In this case 
\[
\Q_{c,s}(\nabla) = (s^2 \id - (-\Delta))^{-1} = R_{s^2}(-\Delta).
\]

\end{proof}

As one can easily verify, the nabla operator is selfadjoint on $L^2(\rr^3,\hh)$. From the spectral theorem for quaternionic linear operators in \cite{ack}, we hence deduce the existence of a unique spectral measure $E$ on $\sigma_{S}(\nabla) = \rr$, the values of which are orthogonal quaternionic linear projections on $L^2(\rr^3,\hh)$, such that
\[
\nabla = \int_{\rr} s\, dE(s).
\]
Via the measurable functional calculus for intrinsic slice functions, it is now possible to define $f_{\alpha}(s) = s^{\alpha}\chi_{[0,+\infty)}(s)$ of $T$ as
\[
f_{\alpha}(\nabla) = \int_{\rr} s^{\alpha}\chi_{[0,+\infty)}(s)\,dE(s),
\]
where $\chi_{[0,+\infty)}$ denotes the characteristic function of the set $[0,+\infty)$.
This corresponds to defining $\nabla^{\alpha}$ at least on the subspace associated with the spectral values $[0,+\infty)$, on which $s^{\alpha}$ is defined. (Observe that even with the measurable functional calculus the operator $\nabla^{\alpha}$ cannot be defined, as $s^{\alpha}$ is not defined on $(-\infty,0)$.)

We shall now give an integral representation for this operator via an approach similar to the one of the slice hyperholomorphic $H^{\infty}$-functional calculus. Surprisingly, this yields a possibility to obtain the fractional heat equation via quaternionic operator techniques applied to the nabla operator.

For $\alpha\in(0,1)$, we define
\begin{equation}\label{FAlfa}
f_{\alpha}(\nabla)v := \frac{1}{2\pi} \int_{-\I\rr}   S_L^{-1}(s,\nabla)\,ds_{\I}\, s^{\alpha-1} \nabla v\qquad \forall v\in\dom(\nabla).
\end{equation}
Intuitively, this corresponds to Balakrishnan's formula for $\nabla^{\alpha}$, where only spectral values on the positive real axis are taken into account, i.e. points where $s^{\alpha}$ is actually defined, because the path of integration surrounds only the positive real axis. 

\begin{theorem}\label{FAlfaConv}
The integral \eqref{FAlfa} converges for any $v\in \dom(\nabla)$ and hence defines a quaternionic linear operator on $L^2(\rr^3,\hh)$. 
\end{theorem}
\begin{proof}
If we   write the integral \eqref{FAlfa} explicitly, we have
\begin{equation}\label{KALST}
\begin{split}
f_{\alpha}(\nabla)v =& \frac{1}{2\pi} \int_{-\infty}^{+\infty}   S_L^{-1}(-\I t,\nabla)\,(-\I )^2\, (-\I t)^{\alpha-1} \nabla v\\
 =&-\frac{1}{2\pi} \int_{0}^{+\infty}   S_L^{-1}(-\I t,\nabla) (-\I t)^{\alpha-1} \nabla v\,dt\\
& -  \frac{1}{2\pi} \int_{0}^{+\infty}   S_L^{-1}(\I t,\nabla) (\I t)^{\alpha-1} \nabla v\,dt\\
=&-\frac{1}{2\pi} \int_{0}^{+\infty}   S_L^{-1}(-\I t,\nabla) t^{\alpha-1} e^{-\I \frac{(\alpha-1)\pi}{2}}\nabla v\,dt \\
 & -  \frac{1}{2\pi} \int_{0}^{+\infty}   S_L^{-1}(\I t,\nabla) t^{\alpha-1}e^{\I \frac{(\alpha-1)\pi}{2}} \nabla v\,dt,
\end{split}
\end{equation}
 where $f_{\alpha}(\nabla)v$ is defined if and only if the last two integrals converge in $L^2(\rr^3,\hh)$.
 
Let us consider $L^2(\rr^3,\hh)$ as a Hilbert space over $\cc_{\I}$ as in the proof of \Cref{NablaSpec}. If we write $v\in L^2(\rr^3,\hh)$ as $v = v_{1} + \J v_{2}$ with $v_{1},v_{2}\in L^2(\rr,\cc_{\I})$ and apply the Fourier-transform componentwise, we obtain an isometric $\cc_{\I}$-linear isomorphism $\Psi: v\mapsto(\widehat{v}_1,\widehat{v}_2)^T$ between $L^2(\rr^3,\hh)$ and $ \widehat{V} :=L^2(\rr^3,\cc_{\I})\oplus L^2(\rr^3,\cc_{\I})$.  For any quaternionic linear operator $T$ on $L^2(\rr^3,\hh)$, the composition  $\Psi T \Psi^{-1}$ is a $\cc_{\I}$-linear operator on $\widehat{V} $ with $\dom(\Psi T \Psi^{-1}) = \Psi \dom(T)$. 

Applying $\nabla$ to $v\in\dom(\nabla)\subset L^2(\rr^3,\hh)$ corresponds to applying the multiplication operator $M_{G}$ associated with the matrix-valued function $G(\xi)$  defined in \eqref{GMat} to $\widehat{v}(\xi) = (\widehat{v_{1}}(\xi),\widehat{v_{2}}(\xi))^T$. Hence, $\nabla = \Psi^{-1} M_{G} \Psi$ and 
\begin{equation}\label{GradDom}
\Psi \dom(\nabla) = \dom(M_{G}) = \left\{ \widehat{v}\in\widehat{V} : G(\xi)\widehat{v}(\xi) \in \widehat{V} \right\} = \left\{ \widehat{v}\in\widehat{V} : |\xi|\widehat{v}(\xi) \in \widehat{V} \right\}.
\end{equation}
That is last identity holds, as for $\widehat{v}(\xi) = (\widehat{v_{1}}(\xi),\widehat{v_{2}}(\xi))^T \in \widehat{V}$ straightforward computations show that
\begin{equation}\label{RamSTI}
\begin{split}
|G(\xi)\widehat{v}(\xi) |^2 =& \left|  \begin{pmatrix} - \xi_1\widehat{v_{1}}(\xi)  + (-\I\xi_2 + \xi_3)\widehat{v_{2}}(\xi)\\ (\I\xi_2 + \xi_3)\widehat{v_1}(\xi) + \xi_1 \widehat{v_2}(\xi)\end{pmatrix}\right|^2\\
 =& (\xi_1^2 + \xi_2^2 + \xi_3^2)(|\widehat{v_1}(\xi)|^2 + |\widehat{v_2}(\xi)|^2) = |\xi|^2| \widehat{v}(\xi)|^2.
\end{split}
\end{equation}
Because of \eqref{KALST},  we  have
\begin{equation}\label{Ratz1}
\begin{split}
f_{\alpha}(\nabla)v =&-\Psi\frac{1}{2\pi} \int_{0}^{+\infty} \left( \Psi^{-1}S_L^{-1}(-It,\nabla) t^{\alpha-1} e^{-\I\frac{(\alpha-1)\pi}{2}}\nabla\Psi^{-1}\right) \Psi v\,dt \\
 & - \Psi \frac{1}{2\pi} \int_{0}^{+\infty}\left( \Psi^{-1}  S_L^{-1}(It,\nabla) t^{\alpha-1}e^{\I\frac{(\alpha-1)\pi}{2}} \nabla \Psi^{-1}\right)\Psi v\,dt,
\end{split}
\end{equation}
Since $\I v = \I (v_1 + \J v_2) = v_1\I - \J (v_2\I)$ and $\Psi$ is $\cc_{\I}$-linear, we find $\Psi \I \Psi^{-1} (\widehat{v_1}, \widehat{v_2})^T = (\widehat{v_1}\I, \widehat{v_{2}}(-\I))^T$, i.e. multiplication with $\I$ on $L^2(\rr^3,\hh)$ from the left corresponds to the multiplication with the matrix $E:=\mathrm{diag}(\I,-\I)$ on $\widehat{V}$. 
As $\Q_{-\I t}(\nabla)^{-1} = (\nabla^2 + t^2)^{-1} = (-\Delta + t^2)^{-1}$ is a scalar operator and hence commutes with any quaternion, we have
\[
S_L^{-1}(-\I t,\nabla) = \Q_{-\I t}(\nabla)^{-1}\I t - \nabla\Q_{-\I t}(\nabla)^{-1} = (\I t - \nabla)\Q_{-\I t}(\nabla)^{-1},
\]
and in turn
\begin{align*}
&\Psi^{-1}S_L^{-1}(-It,\nabla) t^{\alpha-1} e^{-\I\frac{(\alpha-1)\pi}{2}}\nabla\Psi^{-1}\\
=& \Psi^{-1}\left(\I t \Q_{-\I t}(\nabla)^{-1} - \nabla \Q_{-\I t}(\nabla)^{-1} \right) t^{\alpha-1} e^{-\I\frac{(\alpha-1)\pi}{2}}\nabla\Psi^{-1}\\
=& \left( t M_E  \Q_{-\I t}(M_G)^{-1} - M_G \Q_{-\I t}(M_G)^{-1} \right) t^{\alpha-1}M_{ \exp\left(-\frac{(\alpha-1)\pi}{2}E\right)}M_G.
\end{align*}
The operator $\Q_{\I t}(M_G)^{-1}$ is
\[
\Q_{\I t}(M_G)^{-1} = (M_G^2 + t^2\id)^{-1} =  M_{(G^2 +t^2\id)^{-1}} = M_{(t^2 + |\xi|^2)^{-1}\id} 
\]
with $|\xi|^2 = \xi_1^2 + \xi_2^2+\xi_3^2$ and 
 the operator in the first integral of \eqref{Ratz1} equals therefore
\begin{align*}
&\Psi^{-1}S_L^{-1}(-It,\nabla) t^{\alpha-1} e^{-\I\frac{(\alpha-1)\pi}{2}}\nabla\Psi^{-1}\\
=& M_{tE (t^2 + |\xi|^2)^{-1} - G(t^2 + |\xi|^2)^{-1}}t^{\alpha-1} M_{\exp\left(-\frac{(\alpha-1)\pi}{2} E\right)} M_G.
\end{align*}
It is hence the multiplication operator $M_{A_1(t,\xi)}$ determined by the matrix-valued function
\begin{align*}
&A_1(t,\xi)= \frac{t^{\alpha-1}}{t^2 + |\xi|^2}\left(tE  - G(\xi)\right) \exp\left(-\frac{(\alpha-1)\pi}{2} E\right)G(\xi)\\
=&\frac{ t^{\alpha -1}}{t^2+\xi_1^2+\xi_{2}^2 + \xi_{3}^2}\times \\
&\times\begin{pmatrix}
   e^{-\I\frac{ \alpha  \pi }{2} }\xi_{1}(t - \I \xi_{1})+\I e^{ \I\frac{ \alpha  \pi }{2}} \left(\xi_{2}^2+\xi_{3}^2\right) &  \left(e^{\I\frac{ \alpha  \pi }{2}} \xi_{1}+e^{-\I\frac{ \alpha  \pi }{2}}(\xi_{1}+\I t)\right) (\xi_{2}+\I \xi_{3}) \\
 \left(\I e^{-\I\frac{ \alpha  \pi }{2}} \xi_{1}+e^{\I\frac{ \alpha  \pi }{2} }(-t+\I \xi_{1})\right) (\I \xi_{2}+\xi_{3}) &  e^{\I\frac{ \alpha  \pi }{2}} (-t+ \I \xi_{1}) \xi_{1}-\I e^{-\I\frac{ \alpha  \pi }{2}} \left(\xi_{2}^2+\xi_{3}^2\right) 
\end{pmatrix}.
\end{align*}
Similarly the operator in the second integral of \eqref{Ratz1}  is
\begin{align*}
& \Psi^{-1}  S_L^{-1}(It,\nabla) t^{\alpha-1}e^{\I\frac{(\alpha-1)\pi}{2}} \nabla \Psi^{-1}\\
= & M_{-tE (t^2 + |\xi|^2)^{-1} - G(t^2 + |\xi|^2)^{-1}}t^{\alpha-1} M_{\exp\left(\frac{(\alpha-1)\pi}{2} E\right)} M_G.
\end{align*}
It is hence the multiplication operator $M_{A_{2}(t,\xi)}$ determined by the matrix-valued function
\begin{align*}
&A_2(t,\xi) = \frac{t^{\alpha}}{t^2 + |\xi|^2}\left(-tE  - G(\xi)\right) \exp\left(\frac{(\alpha-1)\pi}{2} E\right)G(\xi)
\\
=&\frac{ t^{\alpha -1}}{t^2+\xi_1^2+\xi_{2}^2 + \xi_{3}^2}\times \\
&\times\begin{pmatrix}
  e^{\I\frac{  \alpha  \pi }{2} }\xi_{1}(t + \I \xi_{1}) -\I e^{- \I\frac{ \alpha  \pi }{2}} \left(\xi_{2}^2+\xi_{3}^2  \right)&  -\left( e^{\I\frac{ \alpha  \pi }{2}}(-\I t + \xi_{1}) + e^{-\I\frac{ \alpha  \pi }{2}} \xi_{1}\right) (\xi_{2} + \I\xi_{3}) \\
 \left( e^{\I\frac{ \alpha  \pi }{2}} \xi_{1}+e^{-\I\frac{ \alpha  \pi }{2} }(-\I t+  \xi_{1})\right) ( \xi_{2}-\I\xi_{3}) &  -e^{-\I\frac{ \alpha  \pi }{2}} (t+\I \xi_{1}) \xi_{1} + \I e^{\I\frac{ \alpha  \pi }{2}} \left(\xi_{2}^2+\xi_{3}^2\right) 
\end{pmatrix}.
\end{align*}
Hence, we have $f_{\alpha}(\nabla)v = \Psi^{-1} f_{\alpha}(M_G) \Psi v$ with 
\begin{equation}\label{AzTU}
\begin{split}
f_{\alpha}(M_G)\widehat{v} :=& -\frac{1}{2\pi} \int_{0}^{+\infty}M_{A_{1}(t,\xi)}\widehat{v}\,dt  - \frac{1}{2\pi} \int_{0}^{+\infty}M_{A_{2}(t,\xi)} \widehat{v}\,dt
\end{split}
\end{equation}
for $\widehat{v} = \Psi v \in \Psi \dom( \nabla)$.

We show now that these integrals converge for any $\widehat{v} \in \Psi\dom(\nabla)$. As $\Psi$ is isometric, this is equivalent to \eqref{FAlfa} converging for any $v\in\dom(\nabla)$. Since  all norms on a finite-dimensional vector space are equivalent, there exists a constant $C>0$ such that 
\begin{equation}\label{MatrixNormEst}
 \| M \| \leq C \max_{\ell,\kappa\in\{1,2\}} |m_{\ell,\kappa}| \qquad \forall M = \begin{pmatrix} m_{1,1} & m_{1,2}\\ m_{2,1} & m_{2,2}\end{pmatrix}\in\cc_{\I}^{2\times 2}.
\end{equation}
The modulus of the $(1,1)$-entry of $A_{1}(t,\xi)$ with $t\geq 0$ is
\begin{align*}
&\frac{ t^{\alpha -1}}{t^2+\xi_1^2+\xi_{2}^2 + \xi_{3}^2}\left| e^{-\I\frac{ \alpha  \pi }{2} }\xi_{1}(t - \I \xi_{1})+\I e^{ \I\frac{ \alpha  \pi }{2}} \left(\xi_{2}^2+\xi_{3}^2\right)\right|\\
=& \frac{ t^{\alpha -1}}{t^2+\xi_1^2+\xi_{2}^2 + \xi_{3}^2}\left(|\xi_{1}t|  + |\xi|^2 \right) \leq \frac{ t^{\alpha -1}}{t^2+|\xi|^2}\left(|\xi |t  + |\xi|^2\right).
\end{align*}
Similarly, one sees that also the $(2,2)$-entry of $A_1(t,\xi)$ satisfies this estimate. For the $(1,2)$-entry we have on the other hand
\begin{align*}
&\frac{ t^{\alpha -1}}{t^2+\xi_1^2+\xi_{2}^2 + \xi_{3}^2} \left|  \left(\I e^{-\I\frac{ \alpha  \pi }{2}} \xi_{1}+e^{\I\frac{ \alpha  \pi }{2} }(-t+\I \xi_{1})\right) (\I \xi_{2}+\xi_{3}) \right|\\
\leq & \frac{ t^{\alpha -1}}{t^2+\xi_1^2+\xi_{2}^2 + \xi_{3}^2}  \left( 2|\xi_{1} | | \xi_{2}+\I \xi_{3} |  + t |\xi_{2}+\I \xi_{3}|\right) \leq  \frac{2 t^{\alpha -1}}{t^2+|\xi|^2} \left( |\xi|^2 + t|\xi|\right).
\end{align*}
Similar computations show that the $(2,1)$-entry does also satisfy this estimate and hence we deduce from \eqref{MatrixNormEst} that
\[
\| A_1(t,\xi) \| \leq 2C \frac{ t^{\alpha -1}}{t^2+ |\xi |^2}\left(|\xi |t  + |\xi|^2\right).
\]
Analogous arguments show that this estimate is also satisfied by $\|A_2(t,\xi)\|$. 

For the integrals in \eqref{AzTU} we hence obtain
\begin{align*}
&\int_{0}^{+\infty}\| M _{A_1(t,\xi)} \widehat{v}\|_{\widehat{V}} dt + \int_{0}^{+\infty}\| M _{A_2(t,\xi)} \widehat{v}\|_{\widehat{V}} \,dt \\
\leq &2 \int_{0}^{+\infty}  2C \left\|  \frac{ t^{\alpha -1}}{t^2+|\xi|^2}\left(|\xi |t  + |\xi|^2\right) |\widehat{v}(\xi)| \right\|_{L^2(\rr^3) }\, dt\\
\leq & 4 C \int_{0}^{1} t^{\alpha-1} \left\| \frac{|\xi|t}{t^2 + |\xi|^2}  |\widehat{v}(\xi)| + \frac{|\xi|^2}{t^2 + |\xi|^2}|\widehat{v}(\xi)| \right\|_{L^2(\rr^3)}\,dt\\
&+  4 C \int_{1}^{+\infty} t^{\alpha-2} \left\| \frac{t^2}{t^2 + |\xi|^2}  |\xi\widehat{v}(\xi)| + \frac{t|\xi|}{t^2 + |\xi|^2}|\xi \widehat{v}(\xi)| \right\|_{L^2(\rr^3)}\,dt.
\end{align*}
Now observe that
\[
\frac{t^2}{t^2 + |\xi|^2} \leq 1, \qquad \frac{|\xi|^2}{t^2 + |\xi|^2}\leq 1,\qquad \frac{t|\xi|}{t^2 + |\xi|^2} \leq \frac{1}{2} <1.
\]
Because of \eqref{GradDom}, the relation $\widehat{v}\in\Psi\dom(\nabla)$ implies that $|\widehat{v}(\xi)|$ and $||\xi| \widehat{v}(\xi)|$ both belong to $L^2(\rr^3)$ and hence we finally find
\begin{align*}
&\int_{0}^{+\infty}\| M _{A_1(t,\xi)} \widehat{v}\|_{\widehat{V}} dt + \int_{0}^{+\infty}\| M _{A_2(t,\xi)} \widehat{v}\|_{\widehat{V}} \,dt\\
\leq & 8C \| v(\xi)\|_{L^2(\rr^3)} \int_{0}^{1} t^{\alpha-1}\,dt + 8C \|\xi \widehat{v}(\xi)\|_{L^2(\rr^3)} \int_{1}^{+\infty} t^{\alpha-2}\,dt,
\end{align*}
which is finite as $\alpha \in (0,1)$. Hence \eqref{AzTU} converges for any $\widehat{v}\in\Psi\dom(\nabla)$ and  \eqref{FAlfa} converges  in turn for any $v\in\dom(\nabla)$.

\end{proof}

\begin{theorem}
The operator $f_{\alpha}(\nabla)$ can be extended to a closed operator on $L^2(\rr^3,\hh)$. For $v\in\dom(\nabla^2) = \dom(-\Delta)$, it is moreover given by
\begin{equation}\label{PaULa}
f_{\alpha}(\nabla)v =  (-\Delta)^{\frac{\alpha}{2}-1}  \left[ \frac{1}{2} (-\Delta)^{\frac{1}{2}} + \frac{1}{2} \nabla\right]\nabla v.
\end{equation}
\end{theorem}
\begin{proof}
Let $v\in\dom(\nabla^2) = \dom(-\Delta)$. We have because of \eqref{SResNabla} that
\begin{equation}
\begin{split}\label{Juju}
f_{\alpha}(\nabla)v
=& \frac{1}{2\pi}\int_{-\infty}^{+\infty} (-\I t\id + \nabla) \Q_{c,-\I t}(\nabla)^{-1}(-\I )^2(-t\I )^{\alpha-1}\nabla v\,dt\\
=& -\frac{1}{2\pi}\int_{0}^{+\infty} (-\I t\id +\nabla) \Q_{c,-\I t}(\nabla)^{-1}t^{\alpha-1}e^{-\I (\alpha-1)\frac{\pi}{2}}\nabla v\,dt\\
&- \frac{1}{2\pi}\int_{0}^{+\infty} (\I t\id +\nabla) \Q_{c,\I t}(\nabla)^{-1}t^{\alpha-1}e^{\I (\alpha-1)\frac{\pi}{2}}\nabla v\,dt.
\end{split}
\end{equation}
Due to \eqref{QsRs}, we have moreover 
\[
 \Q_{c,\I t}(\nabla)^{-1} = (-t^2+\Delta)^{-1} = \Q_{c,-\I t}(\nabla)^{-1}
\]
and hence
\begin{equation}
\begin{split}\label{Juju2}
f_{\alpha}(\nabla)v =& -\frac{1}{2\pi}\int_{0}^{+\infty} t^{\alpha} \Q_{c,\I t}(\nabla)^{-1}\I \left(e^{\I (\alpha-1)\frac{\pi}{2}} - e^{-\I (\alpha-1)\frac{\pi}{2}}\right)\nabla v\,dt\\
& -\frac{1}{2\pi}\int_{0}^{+\infty} \nabla \Q_{c,\I t}(\nabla)^{-1}t^{\alpha-1}\left(e^{\I (\alpha-1)\frac{\pi}{2}}+e^{-\I (\alpha-1)\frac{\pi}{2}}\right)\nabla v\,dt\\
=& \frac{\sin\left((\alpha-1)\frac{\pi}{2}\right)}{\pi}\int_{0}^{+\infty} t^{\alpha} \Q_{c,\I t}(\nabla)^{-1}\nabla v\,dt\\
& -\frac{\cos\left((\alpha-1)\frac{\pi}{2}\right)}{\pi}\int_{0}^{+\infty} \nabla \Q_{c,\I t}(\nabla)^{-1}t^{\alpha-1}\nabla v\,dt.
\end{split}
\end{equation}
For the first integral, we obtain
\begin{equation}\label{ASSE}
\begin{split}
&\frac{\sin\left((\alpha-1)\frac{\pi}{2}\right)}{\pi}\int_{0}^{+\infty} t^{\alpha} \Q_{c,\I t}(\nabla)^{-1} \nabla v\,dt\\
=&\frac{\sin\left((\alpha-1)\frac{\pi}{2}\right)}{\pi}\int_{0}^{+\infty} t^{\alpha} (-t^2 +\Delta)^{-1} \nabla v\,dt\\
=&\frac{\sin\left((\alpha-1)\frac{\pi}{2}\right)}{\pi}\int_{0}^{+\infty} \tau^{\frac{\alpha-1}{2}} (-\tau + \Delta)^{-1}\nabla v\,d\tau\\
=& \frac{1}{2} (-\Delta)^{\frac{\alpha - 1}{2}} \nabla v.
\end{split}
\end{equation}
The last identity follows from the integral representation for the fractional power $A^{{\beta}}$ with $\Re(\beta) \in (0,1)$ of a complex linear sectorial operator $A$ given in \cite[Corollary~3.1.4]{Haase},  namely
\begin{equation}\label{ZUZUjk}
 A^{\beta}v = \frac{\sin(\pi\beta)}{\pi} \int_{0}^{+\infty}\tau^{\beta}\left(\tau+A^{-1}\right)^{-1}v\,d\tau,\qquad v\in\dom(A).
\end{equation}
As $-\Delta$ is an injective  sectorial operator on $L^2(\rr^3,\cc_{\I})$, also its closed inverse $(-\Delta)^{-1}$  is a sectorial operator. Its fractional power $\left((-\Delta)^{-1})\right)^{\frac{1-\alpha}{2}}$ is, because of \eqref{ZUZUjk},  given by the last integral in \eqref{ASSE}. Since $(-\Delta)^{\frac{\alpha-1}{2}} = \left((-\Delta)^{-1}\right)^{\frac{1-\alpha}{2}}$, we obtain the last equality. 

Observe that the expression $\frac{1}{2}(-\Delta)^{\frac{\alpha-1}{2}}\nabla v$ is actually meaningful as we chose $v\in\dom(\nabla^2)$. Indeed, if we consider the operators in the Fourier space $\widehat{V}$ as in the proof of \Cref{FAlfaConv}, then $-\Delta$ corresponds to the multiplication operator $M_{|\xi|^2}$ generated by the scalar function $|\xi|^2$. The operator $(-\Delta)^{\frac{\alpha-1}{2}}$ is then the multiplication operator $M_{|\xi  |^{\alpha-1}}$ generated by the function $(|\xi|^2)^{\frac{\alpha-1}{2}} = |\xi|^{\alpha-1}$. Hence
\[
\dom(-\Delta)^{-\frac{\alpha-1}{2}} = \left\{v\in L^2(\rr^3,\hh) :  \widehat{v}\in \dom(M_{|\xi|^{\alpha-1}})\right\} = \left\{v\in L^2(\rr^3,\hh) :  |\xi|^{\alpha-1} \widehat{v}(\xi) \in \widehat{V}\right\}.
\]
If  $G(\xi)$ is as in \eqref{GMat}, then  $\widehat{\nabla v}(\xi) = M_{G}\widehat{v}(\xi) = G(\xi) \widehat{v}(\xi)\in\widetilde{V}$ and  because of \eqref{RamSTI} we have $|G(\xi)\widehat{v}(\xi)| = |\xi | |\widehat{v}(\xi)|\in L^2(\rr)$. As $\alpha\in(0,1)$, we therefore find that  $\ |\xi|^{\alpha-1} |M_G\widehat{v}(\xi)| =  |\xi|^{\alpha}|\widehat{v}(\xi)|$ belongs to $L^2(\rr^3)$ and so $\widehat{\nabla v}\in\dom( M_{|\xi|^{\alpha-1}})$. This is equivalent to $\nabla v \in \dom\left((-\Delta)^{\frac{\alpha-1}{2}}\right)$.

As $v\in\dom(\nabla^2) = \dom(-\Delta)$, we obtain similarly that the second integral in \eqref{Juju2} equals
\begin{equation}\label{ASSE1}
\begin{split}
&- \frac{\cos\left((\alpha-1)\frac{\pi}{2}\right) }{\pi} \int_{0}^{+\infty} \nabla \Q_{c,\I t}(\nabla)^{-1}t^{\alpha-1}\nabla v\,dt\\
=&   \frac{\sin\left((\alpha-2)\frac{\pi}{2}\right)} {\pi}\int_{0}^{+\infty} \nabla (-t^2 + \Delta)^{-1} t^{\alpha-1}\nabla v\,dt\\
=& \frac{\sin\left((\alpha-2)\frac{\pi}{2}\right)}{2\pi} \int_{0}^{+\infty}  (-\tau + \Delta)^{-1}  \tau^{\frac{\alpha-2}{2}} \nabla^2 v\,d\tau\\
= &\frac{1}{2}  (-\Delta)^{\frac{\alpha}{2}-1} \nabla^2 v.
\end{split}
\end{equation}
Again this expression is meaningful as we assumed $v\in\dom(\nabla^2)$. This is equivalent to $|\xi|^2 \widehat{v}(\xi)\in \widehat{V}$ because $\widehat{\nabla^2 v}(\xi) = |\xi|^2\widehat{v}(\xi)$. Since $\alpha \in (0,1)$ and $\widehat{v}\in\dom(M_{|\xi|^2})$, the function $|\xi|^2\widehat{v}(\xi)$  belongs to the domain of the multiplication operator $M_{|\xi|^{\alpha-2}}$  because $M_{|\xi|^{\alpha-2}} |\xi|^2\widehat{v}(\xi) = |\xi|^{\alpha}\widehat{v}(\xi) \in \widehat{V}$. Since $(-\Delta)^{\frac{\alpha}{2}-1}$ corresponds to $M_{|\xi|^{\alpha-2}}$ on the Fourier space $\widehat{V}$, we find $\nabla^2 v$ in $\dom\big((-\Delta)^{\frac{\alpha}{2}-1}\big)$. 

Altogether, we find
\begin{equation}\label{JF}
f_{\alpha}(\nabla)v = (-\Delta)^{\frac{\alpha}{2}-1} \left[ \frac{1}{2} (-\Delta)^{\frac{1}{2}} + \frac{1}{2} \nabla\right]  \nabla v\qquad \forall v\in\dom(\nabla^2).
\end{equation}

Finally, we show that $f_{\alpha}(\nabla)$ can be extended to a closed operator. We need to show that for any sequence $v_n\in\dom(f_{\alpha}(\nabla)) = \dom(\nabla)$ that converges to $0$ and for which also the sequence $f_{\alpha}(\nabla) v_n$ converges, we have $z := \lim_{n\to+\infty} f_{\alpha}(\nabla)v_n = 0$. In order to do this, we write as in \eqref{Juju2}
\begin{align*}
f_{\alpha}(\nabla) v =&  \frac{\sin\left((\alpha-1)\frac{\pi}{2}\right)}{\pi}\int_{0}^{+\infty} t^{\alpha} (t^2 + \Delta)^{-1}\nabla v\,dt \\
- & \frac{\cos\left((\alpha-1)\frac{\pi}{2}\right)}{\pi}\int_{0}^{+\infty} \nabla (t^2 + \Delta)^{-1}t^{\alpha-1}\nabla v\,dt.
\end{align*}
If we choose an arbitrary, but fixed $r>0$, then $(r + \Delta)^{-1}$ commutes with $(t^2 + \Delta)^{-1}$ and $\nabla$ and we deduce from the above integral representation that 
\[
(r + \Delta)^{-1} f_{\alpha}(\nabla) v= f_{\alpha}(\nabla) (r+\Delta)^{-1}v \qquad\forall v\in\dom(\nabla).
\]

 We show now that the mapping $v\mapsto f_{\alpha}(\nabla) (r+\Delta)^{-1} v$ is a bounded linear operator on $L^2(\rr^3,\hh)$. Since $(r+\Delta)^{-1}$ maps $L^2(\rr^3,\hh)$ to $\dom(\Delta) = \dom(\nabla^2)$, the composition $\nabla^2( r+\Delta)^{-1}$ of the bounded operator $(r+\Delta)^{-1}$ and the closed operator $\nabla^2$ is bounded itself. As we have seen above, $\nabla^2$ and in turn also the bounded operator $\nabla^2(r+\Delta)^{-1}$ map  $L^2(\rr^3,\hh)$ into the domain of the closed operator $(-\Delta)^{\frac{\alpha}{2}-1}$. Hence, also their composition $(-\Delta)^{-\frac{\alpha}{2}-1} \nabla^2 (r+\Delta)^{-1}$  is therefore bounded. Similarly, $\nabla(r+\Delta)^{-1}$ is a bounded operator that maps $L^2(\rr^3,\hh)$ to $\dom(\big(-\Delta)^{\frac{\alpha-1}{2}}\big)$ as we have seen above, and so the composition $(-\Delta)^{\frac{\alpha-1}{2}}\nabla(r+\Delta)^{-1}$ is also bounded. Because of \eqref{JF}, the operator
\[
f_{\alpha}(\nabla)(r+\Delta)^{-1} = \frac{1}{2}(-\Delta)^{\frac{\alpha-1}{2}}\nabla(r+\Delta)^{-1} + \frac{1}{2}(-\Delta)^{\frac{\alpha}{2}-1} \nabla^2  (r+\Delta)^{-1}
\]
 is the linear combination of bounded operators and hence bounded itself.
 
 If a sequence $v_n\in\dom(f_{\alpha}(\nabla))$ converges to $0$ and $z = \lim_{n\to+\infty}f_{\alpha}(\nabla)v_n \in L^2(\rr^3,\hh)$ exists, then 
 \[
 (r+ \Delta)^{-1} z = \lim_{n\to +\infty} (r + \Delta)^{-1}f_{\alpha}(\nabla)v_n = \lim_{n\to+\infty} f_{\alpha}(\nabla)(r + \Delta)^{-1}v_n = 0.
 \]
 But as $(r+\Delta)^{-1}$ is the inverse of a closed operator, its kernel is trivial and so $z = \lim_{n\to+\infty}f_{\alpha}(\nabla)v_n = 0$ . Hence, $f_{\alpha}(\nabla)$ can be extended to a closed operator.

\end{proof}

\begin{remark}
The identity  \eqref{PaULa} might seem surprising at the first glance, but it is actually rather intuitive. By the spectral theorem there exist two spectral measures $E_{(-\Delta)}$ and $E_{\nabla}$ on $[0,+\infty)$ resp. $\rr$ such that $-\Delta = \int_{[0,+\infty)} t\, dE_{-\Delta}(t)$ and $\nabla = \int_{\rr} r\,dE_{\nabla}(r)$. As $\nabla^2 = -\Delta$, the spectral measure $E_{(-\Delta)}$ is furthermore the push-forward measure of $E_{\nabla}$ under the mapping $t\mapsto t^2$ such that
\[
\int_{[0,+\infty)} f(t)\,dE_{(-\Delta)}(t) = \int_{\rr} f\left(t^2\right)\,dE_{\nabla}(t)
\]
for any measurable function $f$. Hence, we have for  $v\in\dom(\nabla^2)$ that
\begin{align*}
f^{\alpha}(\nabla) =& \int_{\rr}t^{\alpha}\chi_{[0,+\infty)}(t)\,dE_{\nabla}(t) v\\
=& \int_{\rr}t^{\alpha-2}\frac{1}{2}(|t| + t) t\,dE_{\nabla}(t)v\\
=&  \int_{\rr}t^{\alpha-2}\,dE_{\nabla}(t) \frac{1}{2} \left(\int_{\rr}|t|\,dE_{\nabla}(t) + \int_{\rr} t\,dE_{\nabla}(t) \right)\int_{\rr}t\,dE_{\nabla}(t)v\\
=& \int_{[0,+\infty)}t^{\frac{\alpha}{2}-1}\,dE_{(-\Delta)}(t)\frac{1}{2}\left(\int_{[0,+\infty)}|t|^{\frac{1}{2}}\,dE_{(-\Delta)}(t) + \int_{\rr} t\,dE_{\nabla}(t)\right)\int_{\rr}t\,dE_{\nabla}(t)v\\
= & (-\Delta)^{-\frac{\alpha}{2}-1} \left[ \frac{1}{2} (-\Delta)^{\frac{1}{2}} + \frac{1}{2}\nabla\right]\nabla v.
\end{align*}

\end{remark}
The vector part of $f_{\alpha}(\nabla)$ is because of \eqref{PaULa} given by
\[
\VEC f_{\alpha}(\nabla) v = \frac{1}{2}(-\Delta)^{\frac{\alpha-1}{2}}\nabla v.
\]
If we apply the divergence to this equation with sufficiently regular $v$, we find
\[
\DIV\left(\VEC f_{\alpha}(\nabla)v\right) =  \frac{1}{2}(-\Delta)^{\frac{\alpha -1}{2}}  \Delta v = - \frac{1}{2} (-\Delta)^{\frac{\alpha + 1}{2}}.
\]
We can thus reformulate the fractional heat equation \eqref{fracHeatEQ} with $\alpha \in (1/2, 1)$ as
\[
\partial_t v - 2 \DIV\left(\VEC f_{\beta}(\nabla)v\right) = 0, \qquad \beta = 2\alpha - 1.
\]

\subsection{An example with nonconstant coefficients} As pointed out before, the advantage of the above procedure is that is does not only apply to the gradient to reproduce the fractional Laplacian. Instead it applies to a large class of vector operators, in particular generalized gradients with nonconstant coefficients. As a first example, we consider the operator
\[ T: = \xi_1\frac{\partial}{\partial \xi_1}{e}_1 + \xi_2\frac{\partial}{\partial \xi_2} {e}_2 + \xi_3\frac{\partial}{\partial \xi_3} {e}_3\]
on the space $L^2(\rr^{3}_+, \hh, d\mu)$ of $\hh$-valued functions on $\rr^{3}_{+} = \{\xi = (\xi_1, \xi_2, \xi_3)^T\in\rr^3: \xi_{\ell}>0\}$ that are square integrable with respect to $ d\mu(\xi) = \frac{1}{\xi_1,\xi_2,\xi_3}d\lambda(\xi)$, where $\lambda$ denotes the Lebesgue measure on $\rr^3$. In order to determine $\Q_{s}(T)^{-1}$ we observe that the operator given by the change of variables $J: f \mapsto f\circ\iota$ with $\iota(x) = (e^{x_1},e^{x_2},e^{x_3})^T$ is an isometric isomorphism between $L^2(\rr^{3},\hh,d\lambda(x))$ and $L^2(\rr^{3}_+, \hh, d\mu(\xi))$. Moreover, $T = J^{-1} \nabla J$ such that
\[Q_{s}(T) = (s^2\id + T\overline{T}) = J^{-1} (s^2\id + \Delta) J \]
and in turn
\[ Q_{s}(T)^{-1} := (s^2\id - T\overline{T})^{-1} = J^{-1} (s^2\id + \Delta)^{-1} J.\]
We therefore have for sufficiently regular $v$ with calculations analogue to those in \eqref{Juju} and \eqref{Juju2} that
\begin{align*}
 f_{\alpha}(T)v =& \frac{\sin((\alpha-1)\pi)}{\pi}\int_{0}^{+\infty}t^{\alpha}(-t^2+T\overline{T})^{-1}T \,dt \\
 &+ \frac{\cos((\alpha-1)\pi)}{\pi}\int_{0}^{+\infty} t^{\alpha-1}T(-t^2+T\overline{T})^{-1}Tv\,dt.
 \end{align*}
Clearly, the vector part of this operator is again given by the first integral such that
\begin{align*}
\VEC  f_{\alpha}(T) v =& \frac{\sin((\alpha-1)\pi)}{\pi}\int_{0}^{+\infty}t^{\alpha}(-t^2+T\overline{T})^{-1}Tv \,dt\\
=& \frac{\sin((\alpha-1)\pi)}{\pi}\int_{0}^{+\infty}t^{\alpha}J^{-1}(-t^2+\Delta)^{-1}JTv \,dt\\
=& J^{-1} \frac{\sin((\alpha-1)\pi)}{\pi}\int_{0}^{+\infty}t^{\alpha}(-t^2+\Delta)^{-1} \,dt \, JTv\\
=& \frac{1}{2}J^{-1} (-\Delta)^{\frac{\alpha-1}{2}} J Tv,
 \end{align*}
 where the last equation follows from computations as in \eqref{JF}. Choosing $\beta = 2\alpha +1$ we thus find for sufficiently regular $v$ that
 \begin{align*}
 &\VEC f_{\beta}(T) v(\xi) \\
 =& \frac{1}{2}J^{-1} (-\Delta)^{\alpha} J Tv(\xi_1,\xi_2,\xi_3) \displaybreak[2]\\
 =& \frac{1}{2}J^{-1} (-\Delta)^{\alpha} \begin{pmatrix} e^{x_1}v_{\xi_1}(e^{x_1},e^{x_2},e^{x_3}) \\ e^{x_2} v_{\xi_2}(e^{x_1},e^{x_2},e^{x_3})\\ e^{x_3} v_{\xi_3}(e^{x_1},e^{x_2}, e^{x_3})\end{pmatrix}\displaybreak[2]\\
  =& \frac{1}{2}J^{-1} \frac{1}{(2\pi)^3} \int_{\rr^3}\int_{\rr^3}-|y|^{2\alpha} e^{i z\cdot y}e^{-x\cdot y}   \begin{pmatrix} e^{x_1}v_{\xi_1}(e^{x_1},e^{x_2},e^{x_3}) \\ e^{x_2} v_{\xi_2}(e^{x_1},e^{x_2},e^{x_3})\\ e^{x_3} v_{\xi_3}(e^{x_1},e^{x_2}, e^{x_3})\end{pmatrix}\, dx\, dy\displaybreak[2]\\
 =&  \frac{1}{2(2\pi)^3} \int_{\rr^3}\int_{\rr^3}-|y|^{2\alpha} e^{i \sum_{k=1}^3\xi_k y_k}e^{-ix\cdot y}   \begin{pmatrix} e^{x_1}v_{\xi_1}(e^{x_1},e^{x_2},e^{x_3}) \\ e^{x_2} v_{\xi_2}(e^{x_1},e^{x_2},e^{x_3})\\ e^{x_3} v_{\xi_3}(e^{x_1},e^{x_2}, e^{x_3})\end{pmatrix}\, dx\, dy.
 \end{align*}

\bibliographystyle{plain}

\end{document}